\documentclass[reqno]{amsart}
\usepackage{amssymb}
\usepackage[left=2.3cm,right=2.3cm,top=2.9cm]{geometry}
\usepackage{amsfonts}
\usepackage{dsfont}
\usepackage{amsmath}
\usepackage{stmaryrd}
\usepackage{enumitem}
\usepackage{hyperref}
\usepackage{soul}
\usepackage{todonotes}

\newtheorem{theorem}{Theorem}
\newtheorem{cor}[theorem]{Corollary}
\theoremstyle{proposition}
\newtheorem{proposition}{Proposition}
\newtheorem{lemma}[proposition]{Lemma}
\theoremstyle{definition}
\newtheorem{fact}{Fact}
\newtheorem{definition}[fact]{Definition}
\newtheorem{remark}{Remark}

\renewcommand{\equiv}{\overset{\rm d}{=}}

\renewcommand{\i}{\mathrm{i}}
\newcommand{\1}{\mathbf{1}}
\newcommand{\A}{\mathcal{A}}
\newcommand{\X}{\mathbf{X}}
\newcommand{\bL}{\mathbf{L}}
\newcommand{\bR}{\mathbb{R}}
\newcommand{\bN}{\mathbb{N}}
\newcommand{\bC}{\mathbb{C}}
\newcommand{\bD}{\mathbb{D}}
\newcommand{\bU}{\mathbb{U}}
\newcommand{\bT}{\mathbb{T}}
\newcommand{\bZ}{\mathbf{Z}}

\newcommand{\bP}{\mathbf{P}}
\newcommand{\cB}{\mathcal{B}}

\newcommand{\bE}{\mathbf{E}}
\renewcommand{\O}{\mathcal{O}}
\newcommand{\F}{\mathcal{F}}
\newcommand{\bphi}{\breve{\varphi}}
\newcommand{\bmu}{\boldsymbol{\mu}}

\newcommand{\hvarphi}{\hat{\varphi}}
\newcommand{\hpi}{\hat{\varpi}}
\newcommand{\hpsi}{\hat{\psi}}
\newcommand{\hchi}{\hat{\chi}}

\newcommand{\dd}{\mathrm{d}}

\newcommand{\blue}{}

\author{Gaultier Lambert}
\address[Gaultier Lambert]{KTH Royal Institute of technology, Matematik} 
\thanks{G.L. acknowledges the support of the
starting grant 2022-04882 from the Swedish Research Council and of the starting grant from the Ragnar S\"oderbergs Foundation.}
\email{glambert@kth.se}

\author{Joseph Najnudel}
\address[Joseph Najnudel]{University of Bristol, School of Mathematics}
\email{joseph.najnudel@bristol.ac.uk}

\title{Subcritical multiplicative chaos and the characteristic polynomial of the C$\beta$E}
\begin{document}

\maketitle

\begin{abstract}
The goal of this article is to expand on the relationship between random matrix and multiplicative chaos theories using the integrability properties of the circular $\beta$-ensembles.
We give a comprehensive proof of the multiplicative chaos convergence for the characteristic polynomial and eigenvalue counting function of the circular $\beta$-ensembles throughout the subcritical phase, including negative powers. 
This generalizes recent results in the unitary case, \cite{NSW20,BF}, to any $\beta>0$ and for the  eigenvalue counting field. 
\end{abstract}

\tableofcontents

\section{Intro}

\subsection{Main result}
Let $\bT=\bR/2\pi\mathbb Z$. 
The \ul{circular $\beta$-ensembles (C$\beta$E)} is the family of probability measures, for $\beta>0$ and $n\in\bN$, 
\begin{equation} \label{CbE}
\frac{1}{Z_n(\beta)} \prod_{1\le k<j \le n} |e^{\i\theta_k} - e^{\i\theta_j}|^\beta  \prod_{1\le k \le n} \frac{d\theta_k}{2\pi} , 
\qquad (\theta_1,\dots,\theta_n) \in \bT^n . 
\end{equation}
This model was introduced by Dyson \cite{Dyson}, it generalizes the circular unitary ensembles (CUE -- for $\beta=2$, \eqref{CbE} corresponds to the law of the eigenvalues of a Haar-distributed random matrix on the group $\mathbb U_n$) and it is a fundamental model in random matrix theory with a rich analytic structure. 
For any $\beta>0$, \eqref{CbE} can be viewed as the equilibrium distribution of a Coulomb gas {\blue on the unit circle $\bU=\{z\in\bC:|z|=1\}$} or as the distribution of the eigenvalues of some random unitary matrices, and the partition function $Z_n(\beta) $  is known explicitly by \ul{Selberg's integral}.
We refer to Section~\ref{sec:rmt} for additional background and motivations on circular ensembles.

\medskip

In this paper, we take a random matrix perspective and the goal is to study a central object; the large $n$ limit of characteristic polynomials. 
Such limits are expected to be universal for $\beta$-ensembles or log-gases.
Our analysis exploits the fact that \eqref{CbE} can be realized consistently as the eigenvalues of a sequence of (sparse) unitary matrices  on the same probability space $(\Omega,\F,\bP)$. These matrices have been introduced in \cite{KN04}, using \ul{CMV representation}, and they depend only on a sequence of random coefficients  $\{\alpha_k\}_{k=0}^\infty$ in $\bD=\{z\in\bC:|z|<1\}$ and an independent uniform random variable $\eta$ in $\bT$. We review the main details of the underlying theory in Section~\ref{sec:opuc}. 
{\blue In particular, the specific law of the coefficients  $\{\alpha_k\}_{k=0}^\infty$ is given in Fact~\ref{fact:Verb} below.}
The focal point is that this provides an explicit coupling to study the asymptotics as $n\to\infty$ of the eigenvalues and characteristic polynomials of the circular $\beta$-ensembles. 
{\blue We record the following description of the C$\beta$E from \cite{KN04,KS09}.}

\begin{proposition}[Pr\"ufer phases] \label{prop:phase}
For $n\in\bN_0$, we define the quantity $\varpi_n : [0,2\pi] \to \bR$ by $\varpi_0(\theta) =\theta$ and  
\begin{equation} \label{Phase}
\varpi_{n+1}(\theta)  = \theta + \varpi_{n}(\theta) -2\Im\log\big(1-\alpha_n e^{\i  \varpi_{n}(\theta) } \big) ,\qquad \theta \in [0,2\pi]. 
\end{equation}
For every $n\in\bN$, the function $\theta\in[0,2\pi] \mapsto \varpi_{n-1}(\theta)$ is continuously increasing, $ \varpi_{n-1}(2\pi) =  \varpi_{n-1}(0)+2\pi n$, and the set
\[
\mathfrak{Z}_n = \big\{ \theta \in [0,2\pi]:  \varpi_{n-1}(\theta)  = -\eta [2\pi] \big\}  
\] 
is distributed according to \eqref{CbE}. 
In the sequel, we write $\mathfrak{Z}_n = (\theta_j)_{j=1}^n$ with the order $0<\theta_1<\cdots<\theta_n<2\pi$. 
\end{proposition}

The quantities $\varpi_n $ are called the \ul{Pr\"ufer phases}, they play the role of the eigenvalue counting functions for the circular $\beta$-ensembles and their behavior can be analyzed  via  the recursion\footnote{Observe that for  $\alpha \in \bD$, the map $\omega\in\bR\mapsto \Im\log(1-\alpha_n e^{\i  \omega }) $ is smooth. Then, the recursion \eqref{Phase} can be extended for all $\theta\in\bR$.} \eqref{Phase}. 
This type of analysis has been initiated in the seminal paper \cite{KS09} and further continued in e.g.~\cite{CMN,CN,VV20,VV22,NPS23,PZ}.

\medskip

We define the C$\beta$E \ul{characteristic polynomials}; for every $n\in\bN$, 
\begin{equation}  \label{charpoly0}
\mathcal{X}_n(z) = {\textstyle\prod_{j=1}^n}\big(1-ze^{-\i \theta_j}\big) . 
\end{equation}
Then, we can define {\blue $\log\mathcal{X}_n(z)=  {\textstyle\sum_{j=1}^n}\log(1-ze^{-\i \theta_j}) $ for $z\in\bD$, using the principle branch of $\log(\cdot)$,} and by continuity,  $\mathcal{Y}_n : \theta \in\bT \mapsto \Im\log\mathcal{X}_n(e^{\i\theta})$.
We verify that for $n\in\bN$, 
\begin{equation}  \label{Y0}
\mathcal{Y}_n(\theta) ={\textstyle \sum_{j=1}^n}\mathrm{h}(\theta-\theta_j)
\end{equation}
where $\mathrm{h}: \bT \mapsto [-\frac\pi2, \frac\pi2]$ is defined by $\mathrm{h}(\theta) =\frac{\theta-\pi}{2}$ for $\theta\in(0,2\pi)$, \eqref{Y2}.    
In particular, the function $\mathcal{Y}_n$ is piecewise linear on $\bT$ with slope $\frac n2$ and it jumps by $-\pi$ on $\mathfrak{Z}_n$. Then, $\mathcal{Y}_n$ is called the \ul{eigenvalue counting function} and we verify that 
for any function $f \in C^1(\bT)$, 
\[
\frac1\pi \int_{\bT} f'(\theta) \mathcal{Y}_n(\theta)  \dd\theta  = {\textstyle \sum_{j=1}^n } f(\theta_j)  -  n\int_{\bT} f(\theta) \frac{\dd\theta}{2\pi}
=  {\textstyle \sum_{j=1}^n } f(\theta_j)-\bE\big( {\textstyle \sum_{j=1}^n } f(\theta_j)\big)
\]
where the last identity follows by rotation-invariance of \eqref{CbE}. 

\medskip

An important property of the C$\beta$E and other log-gases is that the eigenvalues are \ul{logarithmically correlated}. 
At first, this is a consequence of the central limit theorem (CLT) for eigenvalue statistics.
Given $f \in C^{1+\epsilon}(\bT,\bR)$ for a $\epsilon>0$, it holds in distribution as $n\to\infty$, 
\begin{equation} \label{Joh}
\frac1\pi \int_{\bT} f'(\theta) \mathcal{Y}_n(\theta) \dd\theta   \Rightarrow \mathcal N\big(0,\tfrac2\beta\Sigma(f)\big) 
\end{equation}
where the variance
\[
\Sigma(f) = \iint_{\bT^2} f'(\theta) f'(\vartheta) \log|1-e^{\i(\theta-\vartheta)}|^{-1} \frac{\dd\theta}{2\pi}\frac{\dd\vartheta}{2\pi}
= \sum_{k\ge 1}  k |f_k|^2 
\]
and $(f_k)_{k\in\mathbb Z}$ denotes the Fourier coefficients of the test function. 
The interpretation of this result is that the random field $\{ \mathcal{Y}_n(\theta) :\theta\in\bT\}$ converges weakly  towards a Gaussian generalized function with mean-zero and correlation kernel, $(\theta,\vartheta)\in \bT^2 \mapsto \beta^{-1} \log|e^{\i\theta}-e^{\i\vartheta}|^{-1}$. 
A similar convergence result also holds for $\theta\in\bT \mapsto \log|\mathcal{X}_n(e^{\i\theta})|$, which is a continuous function taking values $-\infty$ on~$\mathfrak{Z}_n$.  
Theorem \eqref{Joh} was first obtained for $\beta=2$ by Johansson \cite{Joh88} (although the same proof can be applied to any $\beta>0$, \cite{L21}) 
and the interpretation in terms of a log-correlated field is due to \cite{HKO01}.
The result for general $\beta>0$ follows from \cite{JM15} (see also \cite{Killip,FTW19}).
In particular, the parameter $\beta$ simply controls the variance of limiting Gaussian field.
{\blue To further elaborate on these results, we recall the following property from \cite[Sect 3]{CN}.}

\begin{proposition} \label{prop:GAF}
For a fixed $\beta>0$, within the coupling from \cite{KN04}, it holds almost surely as $n\to\infty$, 
\begin{equation*} 
\mathcal{X}_n(z) \to e^{\varphi(z)} \qquad\text{locally uniformly for $z\in\bD$,}
\end{equation*}
where $\{\varphi(z):z\in\bD\}$ is a mean-zero \ul{Gaussian analytic function (GAF)} with covariances
\begin{equation} \label{GAF}
\bE\big[\varphi (z)\overline{\varphi(w)} \big]= -\tfrac2\beta \log(1-z\overline{w})   , \qquad
\bE\big[\varphi(z)\varphi(w)\big]= 0. 
\end{equation}
\end{proposition}

For completeness, we provide the main steps of the proof of  
Proposition~\ref{prop:GAF} in the Appendix~\ref{A:model}. 
We emphasize again that in this framework, the GAF $\{\varphi(z):z\in\bD\}$ and the 
C$\beta$E$_n$ for $n\in\bN$ are defined on the same probability space.

\medskip

{\blue The limits $\chi(u) : u \in \bU \mapsto  \displaystyle\lim_{r\to1} \Re\varphi (ru)$ and $\psi : u \in \bU \mapsto \displaystyle\lim_{r\to1} \Im\varphi(ru)$ exist}, for instance, in the Sobolev space $H^{-\epsilon}(\bU,\bR)$ for any $\epsilon>0$ and $(\chi,\psi)$ are independent Gaussian log-correlated fields on~$\bU$.  
These fields are identically distributed with covariance kernel
\[
(u,v)\in \bU^2 \mapsto \beta^{-1} \log|u-v|^{-1} .
\]
Then, we can associate to these fields two families of \ul{Gaussian multiplicative chaos measures (GMC)}. 
The proof of the next proposition follow directly\footnote{Since $\varphi$ is a GAF in $\bD$, $\Re\varphi(z)= P_z\chi$ (and similarly for $\Im\varphi$) where $P$  is the Poisson kernel of $\bD$ ($P_z$ is a smooth mollifier on $\bU$ for $z\in\bD$).} from \cite[Thm~1.1]{B15}.

\begin{proposition}[GMC] \label{prop:GMC}
Let $\gamma\in\bR$ and $\hat\gamma := \gamma/\sqrt{2\beta}$. 
Suppose that $|\hat\gamma|<1$.
There are two random measures on $\bT$ such that 
\begin{equation} \label{GMCdef}
\mu^{\hat\gamma}(\dd\theta) := \lim_{r\to1} \frac{| e^{\gamma \varphi(re^{\i\theta})}|^\gamma}{\bE| e^{\gamma \varphi(r)}|^\gamma}\frac{\dd\theta}{2\pi}, \qquad\qquad
\nu^{\hat\gamma}(\dd\theta) := \lim_{r\to1} \frac{e^{\gamma\Im\varphi(re^{\i\theta})}}{\bE e^{\gamma \Im\varphi(r)}}\frac{\dd\theta}{2\pi} .
\end{equation}
Both limits hold in  probability with respect to the topology of weak convergence for non-negative measures on $\bT$. 
Moreover, the GMC  $\mu^{\hat\gamma},\nu^{\hat\gamma}$   are independent, with the same law. {\blue The distribution of  the random measure $\mu^{\hat\gamma}$ depends on $(\gamma,\beta)$ only via $\hat\gamma$
and $0<\mu^{\hat\gamma}(\bT)<\infty$ for $|\hat\gamma|<1$.}
\end{proposition}

Gaussian multiplicative chaos has many crucial applications in modern probability and we refer to the survey \cite{RV14} for an overview of the theory. 
For now, we recall that these random measures  
have exact Hausdorff dimension $1-\hat\gamma^2$ (almost surely). They are continuous with respect to the parameter $\hat\gamma$ in the appropriate topology and 
$\bE\mu^{\hat\gamma}(\dd\theta) = \mu^{0}(\dd\theta) = \frac{\dd\theta}{2\pi}$. 
For $|\hat\gamma|\ge1$, the limits \eqref{GMCdef} also exists in probability, but they are equal to~0. 
The non-trivial regime $|\hat\gamma|<1$ is called the \ul{subcritical phase}.

\medskip

Proposition~\ref{prop:GAF} raises the question whether the C$\beta$E characteristic polynomials also give raise to suitable approximations of the GMC measures \eqref{GMCdef}.  
Our main goal  is to obtain the following results. 

\begin{theorem} \label{thm:charpoly}
Recall \eqref{charpoly0} and the subsequent definition of $(\mathcal{Y}_n)_{n\in\bN}$. 
Let $\gamma \in\bR$, $\hat\gamma := \gamma/\sqrt{2\beta}$ such that $|\hat\gamma|<1$.\\
If $\gamma >-1$, it holds in probability as $n\to\infty$, 
\[
\frac{|\mathcal{X}_n(e^{\i\theta})|^\gamma}{\bE|\mathcal{X}_n(1)|^\gamma} \frac{\dd\theta}{2\pi} \to \mu^{\hat\gamma}(\dd\theta) . 
\]
It holds in probability as $n\to\infty$, 
\[
\frac{e^{\gamma \mathcal{Y}_n(\theta)}}{\bE e^{\gamma \mathcal{Y}_n(0)}}\frac{\dd\theta}{2\pi}\to \nu^{\hat\gamma}(\dd\theta) . 
\]
\end{theorem}

Both results hold with respect to the topology of weak convergence for non-negative measures on $\bT$ and they cover the whole subcritical phase as in Proposition~\ref{prop:GMC}. Observe that the convergence to $\mu^{\hat \gamma}$ only makes sense for $\gamma>-1$, since obviously $\bE|\mathcal{X}_n(1)|^\gamma<\infty$ if and only if $\gamma>-1$. 
Prior to Theorem~\ref{thm:charpoly}, the only complete results on multiplicative chaos in random matrix theory are due to \cite{Webb15,NSW20,BF} for the modulus of the CUE characteristic polynomial with $\gamma\ge 0$. 
These results rely on the determinantal structure of the model ($\beta=2$) 
and more analytic methods.
In contrast, the proof of Theorem~\ref{thm:charpoly} is probabilistic, based on martingale convergence arguments, albeit being specific to circular $\beta$-ensembles, it yields stronger convergence results. 

\medskip 

This article is devoted to the proof of Theorem~\ref{thm:charpoly} and the arguments are organized as follows;

\begin{itemize}[leftmargin=*] \setlength\itemsep{.5em}
\item In Subsection~\ref{sec:CbE}, we review in details the results from \cite{KN04}, further developed in \cite{CMN,CN}. 
This section relies on the theory of \ul{orthogonal polynomials on the unit circle} (OPUC) \cite{Sim04}. 
In particular, this yields a martingale approximation $\varphi_n = \bE(\varphi|\F_n)$, with $\F_n = \sigma(\alpha_0,\dots, \alpha_{n-1})$ for the GAF from Proposition~\ref{prop:GAF}.
We also discuss some additional results which follow from our analysis, included an alternative elementary proof of the main result from \cite{CN}.

\item In Subsection~\ref{sec:rmt}, we review the main results related to Theorem~\ref{thm:charpoly} in random matrix theory.
Given the extend of the literature, we focus mostly on  the log-correlated structure of the eigenvalues, multiplicative chaos, circular $\beta$-ensembles and the Fyodorov Bouchaud conjecture.

\item In Section~\ref{sec:phi}, we obtain a GMC convergence for the martingale sequence $(\varphi_n)_{n\in\bN_0}$; see Theorem~\ref{thm:phi} below. 
The method is analogous for both real and imaginary part of the martingale $(\varphi_n)_{n\in\bN_0}$ so we explain the main steps for the real part.
Define $  \mu^{\gamma}_n(\dd\theta)  := \frac{e^{\gamma \Re\varphi_n(e^{\i\theta})}}{\bE e^{\gamma \Re\varphi_n(1)}}\frac{\dd\theta}{2\pi}$ for $n\in\bN$. 
Formally, the goal is to prove that for $\gamma \in\bR$ with $|\hat\gamma|<1$ and $f\in C(\bT,\bR_+)$, 
\begin{equation} \label{exch}\begin{aligned}
\lim_{n\to\infty} \mu^{\gamma}_n(f)
&=\lim_{n\to\infty} \int \bigg(\lim_{r\to1} \frac{e^{\gamma \Re\varphi_n(re^{\i\theta})}}{\bE e^{\gamma \Re\varphi_n(r)}}\bigg) f(\theta)\frac{\dd\theta}{2\pi} \\
&=\lim_{r\to1}  \int \bigg(\lim_{n\to\infty} \frac{e^{\gamma \Re\varphi_n(re^{\i\theta})}}{\bE e^{\gamma \Re\varphi_n(r)}}\bigg) f(\theta)\frac{\dd\theta}{2\pi}
=  \lim_{r\to1}  \int  \frac{e^{\gamma \Re\varphi(re^{\i\theta})}}{\bE e^{\gamma \Re\varphi(r)}} f(\theta)\frac{\dd\theta}{2\pi} = \mu^{\hat\gamma}(f) . 
\end{aligned}
\end{equation}
The first step follows by continuity and the third step from a martingale convergence theorem (the last step is Proposition~\ref{prop:GMC}). The main technical challenge is to justify exchanging the limits \eqref{exch}.\
A similar analysis has been performed in \cite{CN} with $\gamma=-2$ (in the regime $\beta>2$). However, we give a new and \emph{elementary} proof of \eqref{exch} based on martingale arguments (Proposition~\ref{prop:phi}). 

\item The final step is to relate the asymptotics of $\frac{|\mathcal{X}_n(e^{\i\theta})|^\gamma}{\bE|\mathcal{X}_n(1)|^\gamma} \frac{\dd\theta}{2\pi}$ to that of $  \mu^{\gamma}_n$.  
In Section~\ref{sec:char}, we show that these measures have the same limit, in probability as $n\to\infty$. The starting point is the relationship between the characteristic polynomial $\mathcal{X}_{n+1}$ and $\varphi_n$; one has for $n\in\bN_0$, with $u=e^{\i\theta}$, 
\begin{equation*}
\mathcal{X}_{n+1}(u)=  e^{\varphi_n(u)}\big(1-e^{\i \eta+\i  \varpi_{n}(\theta)}\big) , \qquad \theta\in\bT , 
\end{equation*}
with $\varpi_n$ and $\eta$ as in Proposition~\ref{prop:phase}. 
If $\gamma >-1$, the function $\mathrm f :\theta \in\bT \mapsto |1-e^{\i\theta}|^\gamma$ is $L^1$ with mean $\mathrm f_0>0$, so the mass of the random measure associated with the characteristic polynomial is well-defined; 
\[
\int_{\bT} \frac{|\mathcal{X}_{n+1}(e^{\i\theta})|^\gamma}{\bE|\mathcal{X}_{n+1}(1)|^\gamma} \frac{\dd\theta}{2\pi}
= \mathrm f_0^{-1} \int_{\bT}  \mathrm f(\eta+\varpi_{n}(\theta)) \mu^{\gamma}_n(\dd\theta) . 
\]
By density of the trigonometric polynomials in $L^1(\bT)$, it will suffice to show that for any $\kappa \in\bN$, it holds in probability as $n\to\infty$,  
\begin{equation} \label{Fourier}
\int_{\bT}  \mathrm e^{i\kappa\varpi_{n}(\theta)} \mu^{\gamma}_n(\dd\theta) \to 0 . 
\end{equation}
This implies convergence of the mass; $\displaystyle\lim_{n\to\infty}\int_{\bT} \frac{|\mathcal{X}_{n+1}(e^{\i\theta})|^\gamma}{\bE|\mathcal{X}_{n+1}(1)|^\gamma} \frac{\dd\theta}{2\pi}= \displaystyle\lim_{n\to\infty}\mu_n^\gamma(\bT)= \mu^{\hat\gamma}(\bT)$ in probability.
By standard results, such arguments imply Theorem~\ref{thm:charpoly} (an analogous result holds for the imaginary part).
The property \eqref{Fourier} is formulated as Proposition~\ref{prop:RL} below and its proof relies on a \emph{second moment method} using the \emph{branching properties} of the Pr\"ufer phases; this strategy is explained in details in Subsection~\ref{sec:strat}. 
\end{itemize} 

For convenience, the main notations used throughout this article are gathered in Section~\ref{sec:not} and we review some previews results on C$\beta$E  in the Appendix~\ref{A:model}. 
In the Appendix~\ref{A:conc}, we gather some concentration inequalities for martingales that we use in the proofs. 
We emphasize that our proof is independent from previous works on this model, including \cite{CN} and the arguments are self-contained.

\subsection{C$\beta$E coupling; martingale approximations} \label{sec:CbE}
In this section, we give a short introduction to the theory of orthogonal polynomials on the unit circle (OPUC) and explain how this relates to the circular $\beta$-ensembles. As already emphasized, this idea originates from \cite{KN04} and also \cite{Sim04}. 
Starting from a probability measure $\bmu$ on $\bT$, with infinite support, by applying the Gram-Schmidt procedure to the sequence $1,z,z^2,\dots$, one obtains a sequence of (analytic) polynomials $(\Phi_k)_{k\ge0}$, with $\Phi_k(z) =z^k+\cdots$  for $k\in\bN_0$, orthogonal with respect to $\mu$. 
By \cite[Thm~1.5.2]{Sim04}, there is a sequence of coefficients $\{\alpha_k\}_{k=0}^\infty$ in $\bD$, such that the sequence $(\Phi_k)_{k\ge 0}$ satisfies the recursion
\begin{equation} \label{S0}
\begin{cases} \Phi_{k+1}(z) =  z\Phi_{k}(z) - \overline{\alpha_k}  \Phi_{k}^*(z) \\
\Phi_{k}^*(z) = z^k  \overline{\Phi_{k}(1/\overline{z})}
\end{cases}
\qquad z\in\bC , k\in\bN_0. 
\end{equation}

Conversely, one can associate to $\{\alpha_k\}_{k=0}^\infty$, {\blue a unitary operator $\mathcal U_\alpha$ whose  \ul{spectral measure} is $\bmu$.}
The sequence $\{\alpha_k\}_{k=0}^\infty$ are called \ul{Verblunsky coefficients} and they characterize the measure~$\bmu$; see \cite[Chap.~4]{Sim04}. 
Moreover, the regularity properties of $\bmu$ are intimately related to the properties of the sequence $\{\alpha_k\}_{k=0}^\infty$, \cite[Part 2]{Sim04}.  
An alternative way to reconstruct the probability measure $\bmu$ is using the sequence of orthogonal polynomials, via the so-called
\ul{Bernstein-Szeg\H{o}} approximation \cite[Thm~1.7.8]{Sim04}; for the weak convergence of measures on~$\bT$, 
\begin{equation} \label{BS}
\bmu(\dd\theta) = \lim_{n\to\infty} \mathbf{c}_n^2 |\Phi_{n}^*(e^{\i\theta})|^{-2} \frac{\dd\theta}{2\pi} , \qquad 
\mathbf{c}_n^2 = \|\Phi_n^*\|^2_{L^2(\bT)}. 
\end{equation} 

In this framework, the following description of the circular $\beta$-ensembles follows from \cite[Thm~1]{KN04}, see also \cite[(10)]{KS09}. 
We will use these conventions throughout this article. 

\begin{fact} \label{fact:Verb}
Let $\beta>0$ be a fixed parameter. Consider the Verblunsky coefficients 
$\alpha_k = |\alpha_k| e^{\i \eta_k}$, for $k\in\bN_0$,  where $\{\eta_k\}_{k\in\bN_0}$ and $\{|\alpha_k|\}_{k\in\bN_0}$ are independent;
\begin{itemize}[leftmargin=*] \setlength\itemsep{.3em}
\item $\eta_k$ are i.i.d.~uniform in $\bT$.
\item $|\alpha_k|^2$ are independent Beta-distributed random variables with distribution function, for $\beta_k := \beta \tfrac{k+1}{2}$, 
\begin{equation}  \label{verb0}
\bP\left[ |\alpha_k|^2 \ge r\right] = (1-r)^{\beta_k} ,\qquad r\in[0,1]  ,\, k\in\bN_0. 
\end{equation}
\end{itemize}
Let $\bmu_\beta$ be the spectral measure associated with these Verblunsky coefficients.  For this model, the Verblunsky coefficients $\alpha_k\to0$ as $k\to\infty$ (almost surely) and $\bE|\alpha_k|^2 = (1+\beta_k)^{-1}$. 
\end{fact}

Let $\eta$ be a random variable, uniform in $\bT$, independent of $\{\alpha_k\}_{k=0}^\infty$. 
{\blue One can construct a sequence of unitary matrices $\{\mathcal U_\alpha^{(n)}\}_{n=1}^\infty$, called \ul{CMV matrices}, such that for every $n\in\bN$,  $\mathcal U_\alpha^{(n)}\in \mathbb U_n$ is a function of $(\alpha_0,\dots, \alpha_{n-2},\eta)$, its eigenvalue set $\mathfrak{Z}_n$ is distributed according to \eqref{CbE}, and 
$\mathcal U_\alpha = \varprojlim\mathcal U_\alpha^{(n)}$ as ${n\to\infty}$.}
Moreover, the characteristic polynomials satisfy
\begin{equation} \label{charpoly1}
\mathcal{X}_n(z) := \det(1-z\mathcal U_\alpha^{(n)*}) =  \Phi_{n-1}^*(z) - e^{\i\eta} z \Phi_{n-1}(z) \qquad  z\in\bC,
\end{equation}

{\blue
In the sequel, the probability space $(\Omega,\F,\bP)$ carries the random variables $\{\alpha_k\}$, $\eta$ from Fact~\ref{fact:Verb} and we consider the filtration}
\begin{equation} \label{filt}
\F_{k} := \sigma\big(\alpha_{0},\cdots, \alpha_{k-1}\big) , \qquad k\in\bN_0. 
\end{equation}

\medskip

One of the goal of this article is to understand, just as \eqref{BS}, the asymptotics of powers $ |\Phi_{n}^*(e^{\i\theta})|^{\gamma}$  for the C$\beta$E. We establish that for suitable values of $\gamma\in\bR$, after an appropriate renormalization, these powers converge to a family of random measures, which are Gaussian multiplicative chaos; see Theorem~\ref{thm:phi} below.
In particular, the asymptotics of the characteristic polynomial (Theorem~\ref{thm:charpoly}) will be obtained via the recursion \eqref{S1} rather than using the CMV matrices. 
In the sequel, we will study the sequence of polynomials $(\Phi_k^*)_{k\ge 0}$ (or rather its logarithm), instead of $(\Phi_k)_{k\ge 0}$, since one has the following properties (see \cite[Sec~1.7]{Sim04});
\begin{itemize}[leftmargin=*] \setlength\itemsep{.3em}
\item[1.] For every $k\in\bN_0$, $\Phi_k^*(0)=1$, $\Phi_k^*$ has no zeros in $\overline\bD$  so one can define
\begin{equation} \label{phase1}
\varphi_k(z) := \log \Phi_{k}^*(z) , \qquad    z\in\overline{\bD} 
\end{equation}
{\blue in such a way that the functions $\varphi_k$ are analytic in a neighborhood of $\overline\bD$  with $\varphi_k(0)=0$ for all  $k\in\bN_0$.}
\item[2.] For $k\in\bN_0$, define $B_k(z) : = z  \Phi_{k}(z)/ \Phi_{k}^*(z) $ for $ z\in\overline{\bD}$. 
One has $|B_k(z)|=1$ for $z\in \bU$ and (by the maximum principle) $|B_k(z)| \le |z|$ for $z\in \overline{\bD}$. 
Then, {\blue with the principal branch of $\log(\cdot)$}, one can rewrite the recursion \eqref{S0} as; 
\begin{equation} \label{S1}
\varphi_{k+1}(z) =  \varphi_{k}(z) + \log(1-\alpha_k B_k(z)) 
, \qquad z\in\overline{\bD},\, k\in\bN_0. 
\end{equation}
In particular, since $|\alpha_k|<1$ for $k\in\bN_0$, the quantity $\log(1-\alpha_k B_k(z)) $ is analytic for $z\in\overline{\bD}$ and it vanishes at $z=0$. 
\item[3.] The Pr\"ufer phase from Proposition~\ref{prop:phase} are defined by $B_n(e^{\i\theta}) :=e^{\i\varpi_n(\theta)}$  for an appropriate determination of $\varpi_n(\theta)$ for  $\theta\in[0,2\pi]$. 
\item[3.] The characteristic polynomial \eqref{charpoly1} satisfies
\begin{equation*} 
\mathcal{X}_n(z) = \Phi_{n-1}^*(z)(1-e^{\i\eta}B_{n-1}(z)) , \qquad z\in\bC.
\end{equation*}  
For every $n\in\bN_0$, since $\Phi_{n}^*$ has no zero in $\overline\bD$, we recover that eigenvalues of $\mathcal U_\alpha^{(n)}$ are given by 
$\mathfrak{Z}_n= \big\{ \theta \in \bT:  \varpi_{n-1}(\theta)  = -\eta [2\pi] \big\} $; see Proposition~\ref{prop:phase}. 
\end{itemize}

\medskip

Properties 1 and 2 hold for general OPUC whose Verblunsky coefficients  $\alpha_k \in\bD$ for $k\in\bN_0$, while Property 4 is specific to rotation-invariant models (see Remark~\ref{rk:GS} below).
A fundamental observation from \cite{KS09}, also used in subsequent work on this model, is that the sequence $(\varphi_k)_{k\ge 0}$ from \eqref{phase1} is a  martingale (uniformly integrable inside $\bD$).
This property follows directly from the recursion \eqref{S1} and the fact that $\alpha_k$ is independent of $\F_k$ and $\bE[\log(1-\alpha_k B_k(z)| \F_k]=0$ by rotation-invariance. Then,  by a martingale convergence theorem (see Proposition~\ref{prop:cvg} in appendix), almost surely  $\varphi_k \to \varphi $ locally uniformly on $\bD$ as $k\to\infty$, and
\[
\varphi_k(z)= \bE_k\varphi(z) , \qquad   k\in\bN_0,\, z\in \bD. 
\] 

The limit $\varphi$ is a GAF as in Proposition~\ref{prop:GAF}  and this generates two independent Gaussian multiplicative chaos $(\mu , \nu )$ associated with $(\Re\varphi,\Im\varphi)$, measurable with respect to $\F_\infty = \sigma(\alpha_{k} ; k\in\bN_0)$, defined as in Proposition~\ref{prop:GMC}. 
In Section~\ref{sec:phi}, we prove that these GMC can be directly approximated in terms of $\varphi_n := \log \Phi_{n}^*$ as $n\to\infty$.  

\begin{theorem} \label{thm:phi}
For $n\in\bN$, let  
$\psi_{n}(\theta)=\Im\varphi_n(e^{\i\theta})$ for $\theta\in\bT$.  
Let $\gamma \in\bR$ with $\hat\gamma = \gamma/\sqrt{2\beta}$ such that $|\hat\gamma|<1$.\\
It holds almost surely  $($with respect to the topology of weak convergence for non-negative measures on $\bT)$, as $n\to\infty$, 
\[ 
\frac{|\Phi_n^*(e^{\i \theta})|^\gamma}{\bE |\Phi_n^*(1)|^\gamma} \frac{\dd\theta}{2\pi} \to \mu^{\hat\gamma}(\dd\theta) 
\qquad\text{and}\qquad
\frac{e^{\gamma \psi_n(\theta)}}{\bE e^{\gamma \psi_n(0)}}\frac{\dd\theta}{2\pi}\to \nu^{\hat\gamma}(\dd\theta) . 
\]
\end{theorem}

As a consequence of Theorem~\ref{thm:phi} with $\gamma=-2$, using \eqref{BS}, we recover the spectral measure $\mu_\beta$ in the \emph{subcritical phase}~$\beta>2$. 
For C$\beta$E, also for $\beta>2$, the asymptotics of \eqref{BS} have been investigated in \cite{CN}, using the same coupling but different arguments. \cite[Thm~2.1]{CN} establishes that the spectral measure $\bmu_\beta$ is a (normalized) GMC with index $\hat\gamma= -\sqrt{2/\beta}$  for $\beta \ge 2$ (by continuity). 
This recovers the fact that in this regime, $\bmu_\beta$ is singular continuous on $\bT$ with   exact dimension $1-\hat\gamma^2$ and this yields a proof of the Fyodorov-Bouchaud formula \cite{FB08} for the mass of the (classical) GMC on $\bT$. 
We review these results, which are also consequences of Theorem~\ref{thm:phi}. 

\begin{cor} \label{cor:phi}
For $\beta>2$, the spectral measure of the C$\beta$E is a normalized GMC measure on $\bT$, $\bmu_\beta=\mathbf{c}_\beta^{-1} \mu^{\hat \gamma}$, with index $\hat\gamma=-\sqrt{2/\beta}$. Moreover the mass $\mathbf{c}_\beta=  \mu^{\hat \gamma}(\bT)$ satisfies almost surely
\begin{equation} \label{FB}
\mathbf{c}_\beta^{-1}= \lim_{n\to\infty}  \mathbf{c}_n^2 \bE |\Phi_{n}^*(1)|^{-2} . 
\end{equation}
It follows that $\mathbf{c}_\beta^{-1} \equiv \Gamma(1-\hat\gamma^2) \mathbf{e}^{\hat\gamma^2}$, where $\Gamma$ denotes the Gamma function and $\mathbf{e}$  is a standard exponential
random variable.
\end{cor} 

The distribution of $\mathbf{c}_\beta^{-1}$ is known as the Fyodorov-Bouchaud formula and it is elementarily deduced from \eqref{FB} using the explicit law of the C$\beta$E Verblunsky coefficients; see \cite[Sec~2.4]{CN} for details. 
There is an alternative proof based on ideas from conformal field theory to compute negative moments of the GMC mass on $\bT$, \cite{Remy}.


\medskip

The proof of Theorem~\ref{thm:phi} amounts to the exchange of limits \eqref{exch} as described above. This is the strategy used in \cite{CN} for $\gamma=-2$. In Section~\ref{sec:phi}, we follow a different route and obtain the following martingale approximations;

\begin{proposition}\label{prop:phi}
Let $f\in L^1(\bT,\bR)$, $\gamma \in\bR$ with $\hat\gamma = \gamma/\sqrt{2\beta}$ such that $|\hat\gamma|<1$. 
Then, for any $n\in\bN$, 
\begin{equation*} 
\bE_n\mu^{\hat\gamma}(f)
=  \mu_{n}^{\gamma}(f) =\int_{\bT} \frac{|\Phi_n^*(e^{\i \theta})|^\gamma}{\bE |\Phi_n^*(1)|^\gamma} f(\theta)\frac{\dd\theta}{2\pi} ,
\end{equation*}
and
\begin{equation*} 
\bE_n\nu^{\hat\gamma}(f)
=  \nu_{n}^{\gamma}(f) =\int_{\bT} \frac{e^{\gamma \psi_n(\theta)}}{\bE e^{\gamma \psi_n(0)}} f(\theta)\frac{\dd\theta}{2\pi} . 
\end{equation*}
\end{proposition}

Observe that in Proposition~\ref{prop:phi}, the condition $|\hat\gamma|<1$ is necessary, otherwise the random variable $\mu^{\hat\gamma}=0$.
Then, there is $\delta>0$ so that for $f\in L^1(\bT,\bR)$,
\begin{equation} \label{cvg}
\mu_n^\gamma(f) \to \mu^{\hat\gamma}(f) \qquad\text{as $n\to\infty$ almost surely and in $\bL^{1+\delta}$},
\end{equation}  
and similarly for $\nu$. 
Here, the random variable $\mu^{\hat\gamma}(f)\in\bL^{1+\delta}$ provided that  $|\hat\gamma|<1$; see \eqref{GMC}.
Hence, we deduce Theorem~\ref{thm:phi} by standard arguments 
(the topological space $C(\bT)$ is separable).
\medskip

We close this section by several remarks concerning the C$\beta$E model.

\begin{remark}[OPUC theory] \label{rk:GS}
We focus on the case where the support of the probability measure $\bmu$  is infinite.
This condition guarantees that the Verblunsky coefficients are defined for every $k\in\bN_0$ with $\alpha_k \in\bD$. In contrast, if $\bmu$ has a finite support, with say $n$ points on $\bT$, then the Gram-Schmidt procedure stops at step $n$ and $\alpha_0,\dots, \alpha_{n-2} \in \bD$ while $\alpha_{n-1}\in\bT$ -- this is the case for the spectral measure of the matrix $\mathcal U_\alpha^{(n)}$. In particular, for C$\beta$E, by rotation-invariance, $\alpha_{n-1}$ is uniformly distributed on $\bT$. 
Then, observe that for any $n\in\bN$,  upon replacing $\alpha_{n-1}\leftarrow e^{\i\eta}$ and $\Phi_n^* \leftarrow \mathcal{X}_n$ in the recursion \eqref{S0}, this yields formula \eqref{charpoly1} for the C$\beta$E characteristic polynomials. 
\end{remark}

\begin{remark}[Non-universality]\label{rk:Gauss}
The convergence of the sequence $\{\varphi_k\}$ follows from the martingale convergence theorem using the properties of the Verblunsky coefficients. However, the fact that the limit $\varphi$ is a GAF in $\bD$ (Proposition~\ref{prop:GAF}) is an exceptional property of the C$\beta$E model. 
This property cannot be directly deduced from the specific law of the Verblunsky coefficients and it is a consequence of the CLT \eqref{Joh}; see Appendix~\ref{A:model}.   
For another CMV model with independent, rotation-invariant, Verblunsky coefficients with  $\bE |\alpha_k|^2 \sim \beta_k^{-1}$ as $k\to\infty$, we expect that the limit $\varphi$ is a non-Gaussian analytic function in~$\bD$. 
\end{remark}

\begin{remark}[Non-Gaussian multiplicative chaos]
Our analysis relies crucially on three properties of the Verblunsky coefficients $\{\alpha_k\}$; independence, rotation-invariance and a specific decay rate.  
Even though, we rely on the specific distributions of the C$\beta$E Verblunsky coefficients at different stages of the proof for simplicity, these arguments can be adapted if the sequence $\{\alpha_k\}$ satisfies these three properties. 
Then, by \cite{AN22}, the spectral measure $\bmu_\beta$ also has Hausdorff dimension exactly $1-2/\beta$ for $\beta>2$ (subcritical phase), but it is not expected to be a normalized GMC. The question whether it is absolutely continuous with respect to a GMC is of interest for future research. 
\end{remark}

\begin{remark}[Deterministic case]\label{rk:uni}
The limit $\beta=\infty$ corresponds (by continuity) to the case where the Verblunsky coefficients $\alpha_k=0$ for all $k\in\bN$. 
Then, by \eqref{S0}, $\Phi_{k}^*(z) =1$ for all $k\in\bN$ (that is, the orthogonal polynomials are $\Phi_k(z)=z^k$ for ${k\ge 0}$) and the spectral measure $\mu^0 = \frac{\dd\theta}{2\pi}$ is the uniform measure on $\bT$.  
In this case, by Proposition~\ref{prop:phase}, the point configuration is  $\mathfrak{Z}_n= \big\{ \frac{2\pi k +\eta}{n} : k\in [n] \big\} $ for $n\in\bN$, where $\eta$ is uniform in $[0,2\pi]$.
Observe that this configuration is \emph{the minimizer} of the Coulomb energy from \eqref{CbE}: 
\[
(\theta_1,\dots,\theta_n) \in \bT^n \mapsto \sum_{1\le k<j \le n} \log |e^{\i\theta_k} - e^{\i\theta_j}|^{-1}  . 
\] 
\end{remark}

\medskip

\subsection{Related results and state of the art.} \label{sec:rmt}
There are many important works on fluctuations of eigenvalues  and characteristic polynomials of $\beta$-ensembles and we focus on the most relevant recent results in the context of this paper, that is, in relation to log-correlated fields and multiplicative chaos. 

\paragraph{\bf CUE}
An important motivation to study the measure \eqref{CbE} is the case $\beta=2$, which corresponds to the distribution of the eigenvalues of a Haar distributed random matrix in $\bU_n$, this is known as the circular unitary ensemble (CUE).
This is arguably the most basic model in random matrix theory  \cite{Meckes}, it can be analyzed via many different methods and there are notable connections with functional analysis, through Toeplitz determinants, representation theory and probabilistic model for the Riemann $\zeta$ function.
In addition, the heuristics of \cite{FB08,FHK12} which treats the CUE characteristic polynomial as a log-correlated landscapes to make predictions about its extreme values by analogy with the statistical mechanics of random energy models have stimulated a lot of recent developments in random matrix theory.
 

\medskip

\paragraph{\bf Gaussian fluctuations}
The central limit theorem \eqref{Joh} was first obtained in \cite{Joh88} using the \emph{Coulomb gas method}. 
At first, this result is surprising  because the eigenvalue field is asymptotically Gaussian without renormalization, this is due to the log range correlations of log-gases. 
Johansson's method is written for $\beta=2$, but it is easily generalized to arbitrary $\beta>0$, \cite{L21}. 
There is an alternative approach to the CLT based on the moment method and representation theory \cite{DE01,JM15}. 
{\blue This approach is explained in a concise way in \cite[Appendix~A]{CN}.} 
Moreover, another remarkable property of the CUE is that moments of trace in polynomials of Haar-distributed random matrix exactly match moments of Gaussian random variables (see Remark~\ref{rk:traces}).
In fact, these traces approximate Gaussians with a super-exponential rate, we refer to \cite{JL21,CJL24} for quantitative results. 
We refer to \cite[Chap.6]{Sim04} for a comprehensive discussion of the CLT for CUE eigenvalues and related \emph{Szeg\H{o}'s asymptotics}, including several different proofs. 
There are also an alternative approach based on Stein's method and transport which applies to general $\beta$-ensembles \cite{Webb16,LLW19,BLS18,Poly24}. 
It is also worth to mention that some non-integrable generalizations have been studied recently, this includes the fluctuations of $\beta$-ensembles on a regular curve in $\bC$ \cite{CJ23} and the circular Riesz gases \cite{Boursier21,Boursier22}. 

\medskip

\paragraph{\bf Sine$_\beta$ process}
For general $\beta>0$, the measure  \eqref{CbE} was introduced by  Dyson as a simple statistical model for a one-dimensional gas with long range interaction \cite{Dyson}. It can be interpreted as the equilibrium distribution of a two-dimensional Coulomb gas confined on $\bU$, which  corresponds to the stationary measure for
the Dyson Brownian motion on $\bT$, where the strength of the interaction term is determined by $\beta$. 
In many body quantum mechanics,
\eqref{CbE} also corresponds to the ground state of the Calogero-Sutherland Hamiltonian. 
The CMV matrix models for \eqref{CbE} have been introduced in \cite{KN04}.
As explained in Section~\ref{sec:CbE}, these models depend on a single sequence of independent random variables $\{\alpha_k\}$ and this provides a useful coupling to study the asymptotic properties of \eqref{CbE} as the dimension $n\to\infty$. 
This framework has been used in the seminal paper \cite{KS09} to describe the microscopic scaling limit of $\beta$-ensembles. 
This is a stationary point process on $\bR$, called the sine$_\beta$ process, which is universal, e.g.~\cite{BEY14}. 
It can be simply described in terms of the Pr\"ufer phases, 
scaling $\theta \leftarrow \frac\lambda n$ and approximated $\{\alpha_k\}$ by complex Gaussians with variance $\frac2{\beta k}$ (Lemma~\ref{lem:trunc1}) in \eqref{Phase}, one obtains a diffusive limit;
\[\begin{cases}
\dd w_t(\lambda)= \lambda \dd t - \tfrac{2}{\sqrt{\beta t}} \Im(\dd W_t e^{\i w_t(\lambda)}) , &t\in(0,1] \\
w_0(\lambda) = 0 &\lambda\in\bR
\end{cases}\]
where $\{W_t\}_{t\in[0,1]}$ is a complex Brownian motion. It is established in \cite{KS09} that this SDE system has a unique solution with $\bE w_t(\lambda)= \lambda t$ and the sine$_\beta$ process is the point process
\(
\mathfrak{Z}= \big\{ \lambda\in\bR :  w_1(\lambda)  = -\eta [2\pi] \big\}  
\) 
where the random random variable $\eta$ is uniform in $\bT$.
One can also view the CMV operators has a a discrete form
of one-dimensional Dirac operator 
and taking this perspective, one can construct an operator whose spectrum is the sine$_\beta$ point process \cite{VV20} and the scaling limit of the C$\beta$E characteristic polynomial \cite{VV22,NN}. 

\medskip

\paragraph{\bf Fyodorov-Hiary-Keating Conjecture}
The CLT for eigenvalue statistics can be interpreted as the convergence of the log characteristic polynomial in a Sobolev space of generalised functions to a log-correlated field \cite{HKO01}. 
Based on the statistical mechanic property of random energy models and log-correlated landscapes \cite{FB08,FHK12}, this perspective allows to predict the limits from Theorem~\ref{thm:charpoly} as well as the asymptotic behavior of \emph{extreme values} of the characteristic polynomial. 
Significant progress on these conjectures have been achieved recently for C$\beta$E \cite{CMN,PZ} based on the framework described in Section~\ref{sec:CbE}. The state of the art \cite[Thm 1.1]{PZ} gives a distributional convergence for the \emph{centred maximum}; as $n\to\infty$
\begin{equation} \label{FHK}
\Big(\max_{\theta\in\bT}\log|\mathcal{X}_n(e^{\i\theta})| 
- \sqrt{\tfrac2\beta}\big(\log N - \tfrac34 \log\log N\big)\Big) 
\Rightarrow C_\beta + \mathcal{G}_\beta + \frac{\log \mathcal{B}_\beta}{\sqrt{2\beta}}
\end{equation}
where $C_\beta$ is a (deterministic) constant,  $\mathcal{G}_\beta$ is a Gumbel with parameter $1/{\sqrt{2\beta}}$, independent of  $\mathcal{B}_\beta$.
Similar asymptotics hold for the eigenvalue field $\mathcal{Y}_n$. The random variable  $\mathcal{B}_\beta$ is constructed in \cite[Sect 1.2]{PZ} has the limit of a \emph{derivative martingale}, conjecturally it relates to \emph{critical multiplicative chaos} (the counterpart of Theorem~\ref{thm:charpoly} for $\hat\gamma=1$) and one expects that  $\mathcal{B}_\beta = c\mu'(\bT)$ where 
$\mu'(\bT) = \lim_{\hat\gamma \to 1} \frac{\mu^{\hat\gamma}(\bT)}{1-\hat\gamma}$ and $c>0$ is an explicit constant.  
Then, it follows from the Fyodorov--Bouchaud formula (\cite{FB08} and Corollary~\ref{cor:phi}) that $\mathcal{B}_\beta \equiv c^{-1}\mathbf{e}^{-1}$, with $\mathbf{e}$ a standard exponential random variable.
In particular, up to an additive constant, 
$\mathcal{G}_\beta \equiv \frac{\log \mathcal{B}_\beta}{\sqrt{2\beta}} \equiv \frac{-\log\mathbf{e}}{\sqrt{2\beta}}$ are independent Gumbel random variables.
Using the CUE characteristic polynomial as a probabilistic model, \cite{FHK12} also proposed asymptotics for the maximum of the Riemann $\zeta$ function in a typical short interval on the critical line. 
The precise tails for the maximum have been obtained in a series of work 
\cite{Harper} and \cite{ABR} combining methods from analytic number theory and the theory of branching processes. 

\medskip

\paragraph{\bf Multiplicative chaos}
For CUE, part of Theorem~\ref{thm:charpoly}, for $|\mathcal{X}_n|^\gamma$ with $0<\gamma<2$, is due to \cite{Webb15,NSW20}.
The proofs rely on the determinantal structure of the CUE to compute asymptotics of joint moments of the characteristic polynomial. 
The approach involves the asymptotics of Riemann-Hilbert problems with Fisher-Hartwig singularities. 
In \cite{CFLW}, these results have been generalized to other Hermitian unitary-invariant matrix ensembles, including the GUE, and application to \emph{eigenvalue rigidity} are discussed. 
\emph{Moments of moments} of CUE characteristic polynomials have also been studied in several regimes, we refer to \cite{BK22} for an overview. 
Recently, the characteristic polynomial a Brownian motion on $\bU_n$ has been consider in \cite{BF}. The authors obtained multi-time Fisher-Hartwig asymptotics, as well as the convergence to a two-dimensional GMC on a cylinder.
In contrast to previous works, Theorem~\ref{thm:charpoly} is the only GMC result in random matrix theory which does not rely on Fisher-Hartwig type asymptotics. It is also the only result valid for any $\beta>0$ in the whole subcritical regime, including for $\gamma<0$. 

A related concept of \emph{Holomorphic chaos} has been introduced in \cite{NPS23} to describe the limiting random field $\{ \Phi_\infty^*(u) ; u\in\bU\}$.
In contrast to Theorem~\ref{thm:charpoly}, this field is well-defined, without renormalization, as a complex-valued generalised function on $\bU$. 
This property has been used to derive asymptotics for the (Fourier) coefficients of the characteristic polynomial \eqref{charpoly0}; see \cite[Thm~1.10]{NPS23}

Finally, it has been established in \cite[Thm~2.1]{CN} (see Corollary~\ref{cor:phi}) that 
the spectral measure of the CMV operator for C$\beta$E is a normalized GMC measure in the subcritical phase $\beta>2$. 
The exact Hausdorff dimension of these measure have been computed in \cite{AN22} for all $\beta>0$. In particular, there is a similar \emph{freezing transition} as for GMC; one has almost surely 
\begin{itemize} \setlength\itemsep{.2em}
\item $\bmu_\beta$ is singular continuous with Hausdorff dimension   $1-2/\beta$ if $\beta\ge 2$. 
\item $\bmu_\beta$ is purely atomic if $\beta<2$. 
\end{itemize}
This raises the question on how to describe $\bmu_\beta$ in the supercritical phase ($\beta<2$); see \cite[Sec~2.5]{CN}.

\medskip

\paragraph{\bf Open questions}
There are several natural continuations of this work and, to conclude this section, we collect some open problems in the field. 
\begin{itemize}[leftmargin=*] \setlength\itemsep{.3em}
\item By \eqref{FHK}, the final key step to establish the Fyodorov--Bouchaud asymptotics for the maximum of the C$\beta$E characteristic polynomial is to prove that the \emph{derivative martingale} $\mathcal{B}_\beta = c\mu'(\bT)$. This goes beyond the scope of this paper as it pertains to \emph{critical multiplicative chaos}, which involves a different renormalisation scheme, and we intend to return to this problem in a subsequent~work.
\item For CUE, asymptotics of joint moments of characteristic polynomials are known using Toeplitz determinant with Fisher-Hartwig singularities; see \cite{Fahs21} for optimal results, and \cite{BF} for an alternative proof. 
Such Fisher-Hartwig asymptotics are still an open problem for C$\beta$E.

\item The GMC measures from Theorem~\ref{thm:charpoly} are supported on the \emph{thick points of the characteristic polynomial}. 
By \cite[Thm 1.5]{JLW}, for $\gamma >0$ with $
\hat\gamma<1$, and $g\in C(\bT,\bR)$, it holds in probability as $r\to1$,
\[
 \frac{\1\{ \Re \varphi(re^{\i\theta}) \ge \gamma \log N - g(\theta)/\sqrt{2} \}}{N^{-\hat\gamma^2}\sqrt{\pi/\hat\gamma^2 \log N} } \dd\theta
 \to e^{\hat\gamma g(\theta)} \mu^{\hat\gamma}(\dd\theta). 
\]
One also expects that an analogous result holds directly for the characteristic polynomial. 
\item The GMC $\mu^{\hat\gamma}$ from Proposition~\ref{prop:GMC} are random analytic functions, taking values in the space of Schwartz distributions, for $\hat\gamma \in \mathcal{L}$ where $\mathcal{L} \subset \bC$ is a deterministic domain containing $(-1,1)$, see \cite{Lac22}. 
It is of interest to generalize the results of Theorem~\ref{thm:charpoly} to complex GMC in the subcritical domain $\mathcal{L}$. 

\item Viewing the C$\beta$E as a Coulomb gas, $z\in\bU \mapsto -\log|\mathcal{X}_n(z)|$ corresponds to the electric potential generated by the configuration of charges and Theorem~\ref{thm:charpoly} gives the asymptotics of the corresponding Gibbs measures in the subcritical phase.  
Such question are important for two-dimensional Coulomb system at equilibrium in a background potential. 
There are strong motivations, coming from the connection with the two-dimensional Gaussian free field and  two-dimensional quantum gravity to consider the asymptotics of characteristic polynomials of two-dimensional Coulomb gases. 
\end{itemize}




\subsection{Notations} \label{sec:not}
We collect the main notations that will be used in this article.

\medskip 
\noindent
Let $\bD = \{z\in\bC : |z|<1\} $, $\bU = \{z\in\bC : |z|=1\}$ be the boundary of $\bD$, and let $\bT=\bR/2\pi\mathbb Z$. 
Let $\frac{\dd\theta}{2\pi}$ denotes the uniform (Lebesgue) measure on $\bT$ and let $\dd u$ denotes the uniform  measure on $\bU$; 
$\dd u$ is the pushforward of $\frac{\dd\theta}{2\pi}$ by the map $\theta\in\bT \mapsto e^{\i\theta}$. 
In the sequel, we identify measures on $\bT$ and $\bU$.

\medskip 
\noindent
We always consider the (principal) branch of $w\mapsto\log(1-w)$ which is analytic for $w\in\bD$ and vanishes at $w=0$. 
By continuity, we define  $\mathrm{h}(\theta) = \displaystyle\lim_{r\to1}\Im\log(1-re^{\i\theta})$ for $\theta\in\bT$. Then 
$\mathrm{h}: \bT \mapsto [-\frac\pi2, \frac\pi2]$ 
\begin{equation} \label{Y2}
\mathrm{h}(0)=0, \qquad
\mathrm{h}(\theta) =\frac{\theta-\pi}{2} \quad\text{for $\theta\in(0,2\pi)$}. 
\end{equation}
We also let 
\[
\dd(z) := 1-|z|^2, \qquad z\in \bD.
\]

\medskip 
\noindent
Throughout this article, $\beta>0$ is any fixed parameter and, for $\gamma\in\bR$, we write $\hat\gamma:=\gamma/\sqrt{2\beta}$. 
In the context of Proposition~\ref{prop:GAF}, for any $\gamma\in\bR$,\begin{equation} \label{mom}
\bE e^{\gamma\Re\varphi(z)}  = \bE e^{\gamma\Im\varphi(z)}  =\dd(z)^{-\hat\gamma^2} ,\qquad z\in\bD. 
\end{equation}
Then, $\Re\varphi$, $\Im\varphi$ are identically distributed, rotation-invariant, Gaussian fields in $\bD$.
Recall that with the C$\beta$E coupling,  
\[
\Phi^*_{\infty}(z) = \lim_{n\to\infty} \Phi^*_{n}(z) = e^{\varphi(z)}, \qquad z\in\bD.
\]
In the context of Proposition~\ref{prop:GMC}, if $ \hat\gamma^2<1$, then  for any $f\in L^1(\bU,\bR)$, it holds in $\bL^{q}$ as $r\to1$, 
\begin{equation} \label{GMC}
\int_{\bU} \frac{|\Phi^*_{\infty}(ru)|^\gamma}{ \bE |\Phi^*_{\infty}(ru)|^\gamma} f(u) \dd u \to \mu^{\hat\gamma}(f) ,\qquad 
\int  \frac{\bE_n e^{\gamma\Im \varphi (ru)}}{\bE e^{\gamma\Im \varphi(r)}} f(u) \dd u \to \nu^{\hat\gamma}(f)
\end{equation}
for any $q\in\bR$ with $q \hat\gamma^2<1$ -- see e.g.~\cite{RV14}.
In particular,
$\bE \mu^{\hat\gamma}(f) = \mu^0(f) =f_0 $, the mean of $f$, in the subcritical regime.
 
\medskip
\noindent
Throughout the article, the Verblunsky coefficients $\{\alpha_k\}_{k\in\bN_0}$ and $\eta$ are as in Fact~\ref{fact:Verb}. We endow the probability space with the filtration 
\[
\F_{k} = \sigma\big(\alpha_{0},\cdots, \alpha_{k-1}\big) , \qquad k\in\bN_0. 
\]
We will use the shorthand notation $\bE_k = \bE[\,\cdot\,|\F_k]$ for $k\in\bN_0$.  

\medskip
\noindent
Recall that for $n\in\bN_0$, 
\[
\varphi_n(z) = \log \Phi_n^*(z) = \bE_n\varphi(z) , \qquad z\in\bD,
\] 
are well-defined  analytic functions in $\bD$, continuous on $\overline{\bD}$.
Moreover, by  \eqref{S1}, it holds for $n\in\bN_0$,
\begin{equation}\label{S3}
\varphi_{n}(z) =  {\textstyle \sum_{k=0}^{n-1}} \log(1-\alpha_k B_k(z)) , \qquad z\in\overline\bD,
\end{equation}
where $B_k(z) : = z  \Phi_{k}(z)/ \Phi_{k}^*(z) $. One has  $|B_k(z)| \le |z|$ for $z\in \overline{\bD}$ with $|B_k(z)|=1$ for $z\in \bU$.

\medskip
\noindent
The Pr\"ufer phase from Proposition~\ref{prop:phase} are given by, for $n\in\bN_0$, 
\begin{equation} \label{phase0}
B_n(e^{\i\theta}) = e^{\i \varpi_n(\theta)} , \qquad 
\varpi_n(\theta) = (n+1)\theta-2 \psi_n(\theta),  \qquad \theta\in[0,2\pi] .
\end{equation}

\medskip
\noindent
We let $\psi_{n}(\theta)=\Im\varphi_n(e^{\i\theta})$ for $\theta\in\bT$. In Section~\ref{sec:CbE}, we considered the measures
\begin{equation} \label{mm} 
\mu_{n}^{\gamma}(\dd u) 
=  \frac{e^{\gamma\Re \varphi_n(u)}}{\bE e^{\gamma\Re\varphi_n(1)}} \dd u 
=  \frac{|\Phi_n^*(e^{\i\theta})|^\gamma}{\bE |\Phi_n^*(1)|^\gamma} \frac{\dd\theta}{2\pi} ,
\qquad 
\nu_{n}^{\gamma}(\dd u) =  \frac{e^{\gamma\Im\varphi_n(u)}}{\bE e^{\gamma\Im\varphi_n(1)}} \dd u = \frac{e^{\gamma\psi_n(\theta)}}{\bE e^{\gamma\psi_n(0)}}\frac{\dd\theta}{2\pi} .
\end{equation}

\medskip
\noindent
For $n\in\bN$, the C$\beta$E$_n$ characteristic polynomial \eqref{charpoly0} is given, in terms of \eqref{S3}, by
\begin{equation}  \label{charpoly3}
\mathcal{X}_n(z) = e^{\varphi_{n-1}(z)}(1-e^{\i\eta}B_{n-1}(z)) , \qquad z\in\overline{\bD}.
\end{equation}
Then, for $n\in\bN_0$, one has on the unit circle, 
\begin{equation}  \label{charpoly2}
\mathcal{X}_{n+1}(e^{\i\theta})= e^{\varphi_{n-1}(e^{\i\theta})}\big(1-e^{\i \eta+\i  \varpi_{n}(\theta)}\big) , \qquad \theta\in\bT .
\end{equation}

\medskip
\noindent
We write 
${\rm X} \lesssim {\rm Z}$ if there is constant $C>0$, possibly depending on $\beta>0$,  such that 
$0<{\rm X} \le C {\rm Z}$. Such constant $C>0$ are deterministic and vary from line to line.

\medskip
\noindent
We write $\epsilon_n \ll\delta_n$ if $\{\epsilon_n\}$ and $\{\delta_n\}$ are two deterministic sequences of positive numbers so that $\displaystyle\lim_{n\to\infty} \epsilon_n  \delta_n^{-1}=0$. 

\smallskip
\noindent
If $X,X'$ are two random variables, we write $X\equiv X'$ if both have the same distribution.

\section{Martingale convergence; Proof of Theorem~\ref{thm:phi}} \label{sec:phi}

We rely on the framework from Section~\ref{sec:CbE} and the main goal is to prove Proposition~\ref{prop:phi}. In the end, Theorem~\ref{thm:phi} is a direct consequence of the martingale approximation. 
Recall that $\beta>0$ is fixed, the Verblunsky coefficients  $\{\alpha_k\}$ are as in Fact~\ref{fact:Verb} and the notation \eqref{mm}.  
The proof is organized as follows; 

\begin{itemize}[leftmargin=*] \setlength\itemsep{.5em}
\item In Section~\ref{sec:opuc}, we review some additional basic facts from OPUC theory that we will need. In particular, what happens while \emph{shifting} and \emph{rotating} the Verblunsky coefficients. 
\item In Section~\ref{sec:real}, we show that for the real part,  $\bE_n\mu^{\hat\gamma} =  \mu_{n}^{\gamma} $ when $\gamma>0$ with $\hat\gamma<1$ by a simple arguments. 
\item In Section~\ref{sec:neg}, we adapt the arguments to prove that $\bE_n\mu^{-\hat\gamma} =  \mu_{n}^{-\gamma} $ when $\gamma>0$ with $\hat\gamma<1$. This requires more delicate estimates to obtain the \emph{uniform integrality}. This is of particular interest when $\gamma=2$ in the context of the Bernstein-Szeg\H{o} approximation \eqref{BS}. 
\item In Section~\ref{sec:imag}, we adapt the arguments of Section~\ref{sec:real} to the imaginary part to show that $\bE_n\nu^{\hat\gamma} =  \nu_{n}^{\gamma} $ when  $|\hat\gamma|<1$. This is mainly a case of checking that the relevant quantities are well-defined. 
\end{itemize}

\subsection{OPUC theory} \label{sec:opuc}
To set up the analysis for this theorem, we review some properties of OPUC, following \cite[Chap.~3.2]{Sim04}. 
We will use the following conventions;
\begin{itemize}[leftmargin=*]   \setlength\itemsep{0.5em}
\item For $k\in\bN_0$ and $\lambda\in\bU$, let $(\Phi_{k,n}^{\lambda})_{n\ge 0}$ the OPUC family  associated with the Verblunsky coefficients $(\lambda\alpha_{n+k})_{n\in\bN_0}$.
In particular, 
$(\Phi_{k,n})_{n\ge 0}$  denotes the OPUC family associated with the Verblunsky coefficients $(\alpha_{n+k})_{n\in\bN_0}$ and one has $(\Phi_n)_{n\ge 0} =(\Phi_{0,n}^1)_{n\ge 0}$. 
\item For $k\in\bN_0$ and $\lambda\in\bU$, we denote by $(\Psi_{k,n}^\lambda)_{n\ge 0} = (\Phi_{k,n}^{-\lambda})_{n\ge 0}$.
The sequence $(\Psi_n)_{n\ge 0}=(\Phi_{0,n}^{-1})_{n\ge 0}$ are called the second kind polynomials (dual to  $(\Phi_n)_{n\ge 0}$). Then, the Szeg\H{o} recursion \eqref{S0} implies that for $n\in\bN_0$, 
\begin{equation} \label{S2}
\left(\begin{matrix} \Psi_{n+1}(z) & \Phi_{n+1}(z) \\  - \Psi^*_{n+1}(z) & \Phi^*_{n+1}(z) \end{matrix} \right) = 
\left(\begin{matrix} z &  -\bar{\alpha}_n \\  - \alpha_n z & 1 \end{matrix} \right)
\left(\begin{matrix} \Psi_{n}(z) & \Phi_{n}(z) \\  - \Psi^*_{n}(z) & \Phi^*_{n}(z) \end{matrix} \right) , \qquad z\in\bC. 
\end{equation}
\item By linearity of the Szeg\H{o} recursion, one has for any $\lambda\in\bU$, $k \in\bN_0$, 
\begin{equation} \label{S4}
\Phi_{k,n}^{\lambda*} = \Phi_{k,n}^{*}\tfrac{1+\lambda}{2} +  \Psi_{k,n}^{*}\tfrac{1-\lambda}{2} \qquad n\in\bN_{\ge k}. 
\end{equation}
\item We denote for $k\in\bN$,
\begin{equation} \label{ratio}
\frac{ \Phi_k(z)}{\Phi^*_k(z)} = \rho_k(z) \lambda_k(z) \qquad z\in\overline{\bD},
\end{equation}
where  $\rho_k : \overline{\bD} \to [0,1]$ is a continuous function, {\blue the phase $\lambda_k:\overline{\bD} \to \bU$ can be defined pointwise, e.g.~taking the value $\lambda_k(z)=1$ for $z\in\{ \Phi_k = 0\}$.}
This decomposition is related to $B_k$ (Property 2 in Section~\ref{sec:CbE}) and $\Phi_k/\Phi^*_k $ is analytic in a neighborhood of $\overline{\bD}$ with $|\Phi_k/\Phi^*_k | =1$ on $\partial \overline{\bD}$. So, the modulus $\rho_k=1$ on~$\bU$ and $\rho_k < 1$ on $\bD$, by the maximum principle for subharmonic functions. 
\item Consequently the functions $\rho_k$ are 
Lipschitz-continuous on $\bD$. 
That is, there are random constant $\mathcal C_k$  ($\F_k$-measurable) so that 
\begin{equation} \label{Lip}
|1-\rho_k(z)| \le \mathcal C_k \dd(z) , \qquad z\in  \overline{\bD}  . 
\end{equation}


\end{itemize}

\medskip

The starting point for our analysis and the proof of Proposition~\ref{prop:phi} is the following (determinstic) relationship between  $\Phi_\infty^* = e^{\varphi}$ and $\Phi_k^*$ for $k\in\bN$.

\begin{lemma} \label{lem:decoup}
Almost surely, for every $\lambda\in\bU$ and $k\in\bN_0$, $\displaystyle\Phi_{k,\infty}^{\lambda*} :=\lim_{n\to\infty}  \Phi_{k,n}^{\lambda*}$ exists, and  $\Phi_{k,\infty}^{\lambda*}$ is an analytic function without zero in $\bD$. 
In addition, for any $\gamma>0$, $\bE\exp\big(\gamma|\log\Phi_{k,\infty}^{\lambda*}(z)|\big) <\infty$ for $z\in\bD$.
Then, one has for every $k\in\bN$, 
\begin{equation}\label{decoup2}
\begin{aligned}
\Phi_{\infty}^* & =  \Phi^*_k\bigg( \frac{\Psi^*_{k,\infty}+\Phi^*_{k, \infty} }{2}+ \rho_k\lambda_k  \frac{\Phi^*_{k, \infty}-\Psi^*_{k, \infty}}{2} \bigg)  \\
& = \Phi^*_k \Phi^{\lambda_k*}_{k,\infty}\bigg(1+ {\blue\frac{1-\rho_k}2}\bigg(\frac{\Psi^{\lambda_k *}_{k,\infty}}{\Phi^{\lambda_k *}_{k, \infty}}-1\bigg)  \bigg) 
\end{aligned}
\end{equation}
where 
$(\Phi^*_k ,\rho_k, \lambda_k)$ are $\F_k$-measurable and 
$(\Phi_{k,\infty}^*, \Psi_{k,\infty}^*)$ are independent of~$\F_k$. 
\end{lemma}

\begin{proof}
Proposition~\ref{prop:cvg} from \cite{CN} implies that almost surely, for every $k\in\bN_0$, the limits 
$(\Phi_{k,\infty}^{*},\Psi_{k,\infty}^{*})$ exist and are analytic functions without zero in $\bD$. 
It also yields the bound  $\bE|\Phi_{k,\infty}^{*}(z)|^\gamma <\infty$ for $\gamma\ge 1$, $k\in\bN_0$, $z\in\bD$.
Moreover, by rotation-invariance (of the Verblunsky coefficients), 
$\Phi_{k,\infty}^{\lambda*}\equiv \Phi_{k,\infty}^{*}$
for $\lambda\in\bU$, $k\in\bN_0$. 

Then, by \eqref{S4}, one can define (almost surely) for any $\lambda\in\bU$ and $k \in\bN_0$, 
\begin{equation*}
\Phi_{k,\infty}^{\lambda*} = \Phi_{k,\infty}^{*}\tfrac{1+\lambda}{2} +  \Psi_{k,\infty}^{*}\tfrac{1-\lambda}{2} . 
\end{equation*}
In particular, we have the relationships; for $k \in\bN_0$ and $\lambda\in\bU$,
\begin{equation} \label{ulinearity}
\frac{\Phi_{k,\infty}^{\lambda*}+\Psi_{k,\infty}^{\lambda*}}{2} =  \frac{\Phi_{k,\infty}^{*}+ \Psi_{k,\infty}^{*}}{2} ,\qquad 
\frac{\Phi_{k,\infty}^{\lambda*}-\Psi_{k,\infty}^{\lambda*}}{2}  = \lambda \frac{\Phi_{k,\infty}^{*}-\Psi_{k,\infty}^{*}}{2}  .
\end{equation}

Now, fix $k\in\bN_0$.
Using \eqref{S2}, one has for any $n\ge k$,  
\[ 
\left(\begin{matrix} \Psi_{n} & \Phi_{n} \\  - \Psi^*_{n} & \Phi^*_{n} \end{matrix} \right) =
\left(\begin{matrix} \Psi_{k,n-k} & \Phi_{k,n-k} \\  - \Psi^*_{k,n-k} & \Phi^*_{k,n-k} \end{matrix} \right) 
\left(\begin{matrix} 1& 1 \\  - 1& 1 \end{matrix} \right)^{-1}
\left(\begin{matrix} \Psi_{k} & \Phi_{k} \\  - \Psi^*_{k} & \Phi^*_{k} \end{matrix} \right) . 
\]
Taking the limit as $n\to\infty$, this implies that almost surely,  
\[
\Phi_\infty^* = \Psi^*_{k,\infty}  \frac{\Phi^*_k- \Phi_k}{2}  + \Phi^*_{k, \infty}  \frac{\Phi^*_k+ \Phi_k}{2}  , \qquad\text{in $\bD$}. 
\]

Then, rearranging the previous formula using \eqref{ratio}, we obtain
\[\begin{aligned}
\Phi_\infty^* &= \Phi^*_k\bigg( \Psi^*_{k,\infty} \frac{1-\rho_k\lambda_k}{2} + \Phi^*_{k, \infty}  \frac{1+\rho_k\lambda_k}{2} \bigg) \\
&=  \Phi^*_k\bigg( \frac{\Psi^*_{k,\infty}+\Phi^*_{k, \infty} }{2}+ \rho_k\lambda_k  \frac{\Phi^*_{k, \infty}-\Psi^*_{k, \infty}}{2} \bigg) .
\end{aligned}\]
Obviously, $(\Phi^*_k ,\rho_k, \lambda_k)$ are $\F_k$-measurable and for $u\in\bU$, 
$(\Phi_{k,\infty}^{u *}, \Psi_{k,\infty}^{u*})$ are independent of~$\F_k$, since they are defined in terms of $(\alpha_{k})_{k\ge n}$ which are independent of $(\alpha_{k})_{k< n}$.  
Finally, using \eqref{ulinearity}, we conclude that
\[\begin{aligned}
\Phi_\infty^* 
&=  \Phi^*_k\bigg( \frac{\Psi^{\lambda_k *}_{k,\infty}+\Phi^{\lambda_k *}_{k, \infty}}{2}+ \rho_k  \frac{\Phi^{\lambda_k *}_{k, \infty}-\Psi^{\lambda_k *}_{k, \infty}}{2} \bigg) \\
&= \Phi^*_k \Phi^{\lambda_k*}_{k,\infty} \bigg(1+ \frac{1-\rho_k}2\bigg(\frac{\Psi^{\lambda_k *}_{k,\infty}}{\Phi^{\lambda_k *}_{k, \infty}}-1\bigg)  \bigg)
\end{aligned}\]
Note that  $\Phi^{\lambda *}_{k, \infty}$ have no zero in $\bD$, for any $\lambda\in\bU$,  so the last formula is well-posed in $\bD$. 
\end{proof}

We illustrate how Lemma~\ref{lem:decoup} is used to prove Proposition~\ref{prop:phi} in a simplified case. 
By \eqref{decoup2}, conditioning on $\F_k$, we obtain
\[
\bE_k \Phi_{\infty}^*  =  \Phi^*_k\bigg(  \frac{1-\rho_k\lambda_k}{2} \bE\Psi^*_{k,\infty} +  \frac{1+\rho_k\lambda_k}{2} \bE\Phi^*_{k, \infty} \bigg) 
\]
since $(\Phi_{k,\infty}^*, \Psi_{k,\infty}^*)$ are independent of~$\F_k$. 
Then, by rotation-invariance $\bE\Phi_{k,\infty}^*=\bE\Psi_{k,\infty}^*$, we obtain for $k\in\bN$, 
\[
\bE_k \Phi_{\infty}^*  =  \Phi^*_k \bE\Phi_{k,\infty}^* \qquad\text{on $\bD$}. 
\]
Taking expectation, one has  $\bE \Phi^*_k \bE\Phi_{k,\infty}^* = \bE  \Phi_{\infty}^* =1$ in $\bD$, since $\Phi_\infty^* = e^{\varphi}$ with  $\varphi$ the GAF from \eqref{GAF}. 
This implies that for $r\in[0,1)$, $k\in\bN$, 
\[
\bE_k \Phi_{\infty}^*(r u)  =  \frac{\Phi^*_k(r u)}{\bE\Phi_k^*(r)} \qquad u\in\bU. 
\]
This gives a \emph{martingale approximation} for 
\(
\int \Phi_{\infty}^*(r u)  f(u) \dd u
\)
where $f\in C^\infty(\bU)$ is a test function.
Consequently, if there is a random distribution $\mathcal L$ so that
\(
\int {\Phi_{\infty}^*(ru)}  f(u) \dd u\to \mathcal{L}(f)
\)
in $\bL^1$ as $r\to 1$, then one has for any $k\in\bN_0$
\[
\bE_k  \mathcal{L}(f) = \int \frac{\Phi^*_k(ru)}{\bE\Phi_k^*(r)} f(u) \frac{\dd u}{2\pi} . 
\]
{\blue In \cite{NPS23}, the random distribution $\mathcal L$ is called a \emph{holomorphic multiplicative chaos}, it is well-defined in a Sobolev space $H^{-\alpha}$ for $\alpha$ sufficiently large (depending on $\beta$).
Note that the convergence occurs without renormalization and this argument is therefore much simpler than the proof of Proposition~\ref{prop:phi}.
}

\medskip

We will also need two additional facts about OPUC. 
By \cite[Thm~3.2.4]{Sim04}, if $\boldsymbol\mu$ is the spectral measure associated with the Verblunsky coefficients $\{\alpha_k\}$, then 
\begin{equation} \label{cara}
\int_{\bU} \frac{u +z}{u-z} d\boldsymbol\mu(u) = \frac{\Psi_\infty^*(z)}{\Phi_\infty^*(z)} , \qquad z\in\bD. 
\end{equation}
This function is called the  \emph{Carath\'eodory function} (or \emph{Schur function}, or \emph{m-function}) of $\boldsymbol\mu$, it plays an important role in the spectral theory of CMV operators. 
By convention, it is analytic and takes value 1 at $z=0$ and for $z\in\bD$, 
\[
\Re\bigg( \frac{\Psi_\infty^*(z)}{\Phi_\infty^*(z)}\bigg) = \int_{\bU} \frac{1-|z|^2}{|u-z|^2} d\boldsymbol\mu(u) > 0
\]

Moreover, we compute with $\mathfrak{M} =\Psi_{\infty}^{*}/ \Phi_{\infty}^{*}$, 
\[
\bigg| \frac{\Phi_{\infty}^{*}-\Psi_{\infty}^{*}}{\Phi_{\infty}^{*}+\Psi_{\infty}^{*}}\bigg|^2
{\blue = \bigg| \frac{1-\mathfrak{M}}{1+\mathfrak{M}}\bigg|^2
=\frac{1+|\mathfrak{M}|^2-2\Re \mathfrak{M}}{1+|\mathfrak{M}|^2+2\Re \mathfrak{M}}
<1}
\]
since $\Re \mathfrak{M}>0$ in $\bD$. 
Consequently, for the OPUC sequence $(\Phi_{k,n}^{\lambda *})_{n\ge 0}$, with $k\in\bN$, $\lambda\in\bT$, we record the following estimates. 

\begin{fact} \label{fact:cara}
$\bullet$ For $k\in\bN$, one has (almost surely)
\begin{equation*} \label{carabound}
\big|\tfrac{\Phi_{k,\infty}^{*}-\Psi_{k,\infty}^{*}}{\Phi_{k,\infty}^{*}+\Psi_{k,\infty}^{*}}\big| <1 \qquad\text{ on $\bD$}. 
\end{equation*}
$\bullet$ We define for $k\in\bN$ and $\lambda\in\bT$, 
\[  
\mathfrak{g}_k^\lambda(z) :=\frac1{2}\bigg(\frac{\Psi^{\lambda *}_{k,\infty}}{\Phi^{\lambda *}_{k, \infty}}(z)-1 \bigg) , \qquad z\in\bD.
\]
We also let $\mathfrak{g}_k= \mathfrak{g}_k^1$ for $k\in\bN$. 
$\mathfrak{g}_k^\lambda$ is analytic on $\bD$ with 
$\mathfrak{g}_k^\lambda(0)=0$ and 
\(
\Re\mathfrak{g}_k^\lambda(z)  >-\frac12. 
\)
for $z\in\bD$. 
\end{fact}


\medskip

Finally, we record the following  estimates for moments of $\mathfrak{g}_k$. 

\begin{lemma} \label{lem:gmom}
Fix $\epsilon>0$ and recall that $\dd(z) := 1-|z|^2$ for $z\in \bD$. 
For any $k\in\bN_0$ and for any $z\in\bD$ 
\[
\bE \big| \mathfrak{g}_k(z)\big|^{1+\epsilon} \le C_\epsilon \dd(z)^{-\epsilon}  . 
\]
\end{lemma}

\begin{proof}
Here $k\in\bN_0$ is fixed. 
From \eqref{cara},
\[
\mathfrak{g}_k(z) = z \int_{\bU} \frac{\boldsymbol\mu_k(\dd u)}{u-z}  , \qquad z\in\bD ,
\]
where $\boldsymbol\mu_k$ is the spectral measure for the Verblunsky coefficients $\{\alpha_{k+j}\}_{j\ge 0}$. 
Since $\boldsymbol\mu_k$ is a (random) probability measure on $\bU$, by Jensen's inequality, almost surely for any $\gamma\ge 1$ and $|z|=r <1$, 
\[
|\mathfrak{g}_k(z)|^\gamma \le r^\gamma  \int_{\bU} \frac{\boldsymbol\mu_k(\dd u) }{|u-r|^\gamma}.
\]
Observe that by rotation-invariance, the spectral measures  $\boldsymbol\mu_k$ satisfy  $\bE \boldsymbol\mu_k(\dd u) = \dd u$, then
\[\begin{aligned}
\bE |\mathfrak{g}_k(z)|^\lambda  & \le   r^\lambda\int_{\bT}  |e^{\i\theta}-r|^{-\lambda} \frac{\dd\theta}{2\pi} \\
&= \int_{\bT}  \bigg(\frac{r^2}{(1-r)^2+2r(1-\cos \theta)}\bigg)^{\lambda/2} \frac{\dd\theta}{2\pi}  \\
& = \frac2\pi \int_0^{\pi/2} \bigg(\frac{r}{(1-r)^2/r+4\sin^2(\theta)}\bigg)^{\lambda/2} \dd\theta .
\end{aligned}\]
Using that $(1-r)\ge \dd(r)/2$ and $16\sin^2\theta \ge \theta^2 $, we obtain
\[
\bE |\mathfrak{g}_k(z)|^\lambda\lesssim  \dd(z)^{-\lambda} \int_0^{\pi/2} \bigg(\frac{1}{1 +{\blue(\theta/2\dd(z))^2}}\bigg)^{\lambda/2} \dd\theta 
\lesssim \dd(z)^{1-\lambda}
\]
where the implied constants depend only on $\lambda>1$. 
\end{proof}

\subsection{Convergence for real part in case $ \gamma \ge 0$} \label{sec:real}
The goal of this section is to show that for 
$0\le \hat\gamma = \gamma/\sqrt{2\beta}<1$ and $f\in L^1(\bU)$, 
\begin{equation} \label{mart1}
\bE_n\mu^{\hat\gamma}(f)
=  \mu_{n}^{\gamma}(f) =\int \frac{|\Phi_n^*(u)|^\gamma}{\bE |\Phi_n^*(1)|^\gamma} f(u) \dd u. 
\end{equation}
Without loss of generality, we assume that $f\ge 0$. 

\medskip

We start by some estimates which follow directly from Lemma~\ref{lem:decoup}. 

\begin{lemma} \label{lem:bdreal}
For $\gamma \ge 0$, there is a (deterministic) constant $c_\gamma$  so that (almost surely) for any $z\in\bD$ and $n\in\bN$,
\begin{equation*} 
0\le 1- \frac{\bE_n |\Phi_{\infty}^*(z)|^\gamma}{|\Phi_n^*(z)|^\gamma \bE  |\Phi_{n,\infty}^{*}(z)|^\gamma} \le  c_\gamma (1-\rho_n(z)) . 
\end{equation*}
\end{lemma}

\begin{proof}
Fix $\gamma\in\bR$ and $n\in\bN$. 
We define for $\rho\in[0,1]$ and $z\in\bD$,  
\[
\Theta_n(\rho;z) : = \bE_n \bigg| \frac{\Phi_{n,\infty}^{*}(z)+\Psi_{n,\infty}^{*}(z)}{2} + \rho \lambda_n(z)  
\frac{\Phi_{n,\infty}^{*}(z)-\Psi_{n,\infty}^{*}(z)}{2} \bigg|^\gamma 
\]
with  $\lambda_n$ as in \eqref{ratio} -- $\lambda_n$ is $\F_n$-measurable.
Then, by \eqref{decoup2}, it holds for $z\in\bD$, 
\begin{equation} \label{Phiinftyproj}
\bE_n |\Phi_{\infty}^*(z)|^\gamma =  |\Phi_n^*(z)|^\gamma \Theta_n(\rho_n(z);z) .
\end{equation}

Moreover, by rotation-invariance;
for every $u\in\bU$, $(\Phi_{k,n}^{u*},\Psi_{k,n}^{u*})_{n\ge 0}$ are  independent of $\F_k$ and
$(\Phi_{k,n}^{u*},\Psi_{k,n}^{u*})_{n\ge 0} \equiv (\Phi_{k,n}^{*}, \Psi_{k,n}^{*})_{n\ge 0}$ so that for $k\in\bN$, 
\[\begin{aligned}
\Theta_k(\rho;\cdot)
& = \bE_k \bigg| \frac{\Phi_{k,\infty}^{u*}+\Psi_{k,\infty}^{u*}}{2} + \rho \lambda_k  
\frac{\Phi_{k,\infty}^{u*}-\Psi_{k,\infty}^{u*}}{2} \bigg|^\gamma \\
&=   \bE_k \bigg| \frac{\Phi_{k,\infty}^{*}+\Psi_{k,\infty}^{*}}{2} + \rho \lambda_k u 
\frac{\Phi_{k,\infty}^{*}-\Psi_{k,\infty}^{*}}{2} \bigg|^\gamma 
\end{aligned}\]
using the relationships \eqref{ulinearity}. 
Hence, by Fact~\ref{fact:cara}, integrating this formula over $\bU$,  we obtain 
\begin{equation}\label{Theta}
\Theta_n(\rho;\cdot)
= \bE_n \bigg[\ \bigg| \frac{\Phi_{n,\infty}^{*}+\Psi_{n,\infty}^{*}}{2} \bigg|^\gamma \mathrm{F}_\gamma\Big(\rho\big|\tfrac{\Phi_{n,\infty}^{*}-\Psi_{n,\infty}^{*}}{\Phi_{n,\infty}^{*}+\Psi_{n,\infty}^{*}}\big|\Big) \bigg]
\end{equation}
where, for $\gamma\in\bR$,
\begin{equation*}
\mathrm{F}_\gamma : r\in [0,1)  \mapsto \int_{\bU} |1+ ur |^\gamma \dd u.
\end{equation*}
Using the binomial formula, we record that 
\begin{equation} \label{F1}
\mathrm{F}_\gamma(r) = 1+ \sum_{k\ge 1} {\gamma/2 \choose k}^2 r^{2k} , \qquad r\in[0,1). 
\end{equation}
The coefficients ${\gamma/2 \choose k}^2 \sim c_\gamma k^{-2-\gamma}$ as $k\to\infty$ where $c_\gamma = \Gamma(-\gamma/2)^{-2}$ for $\gamma \in \bR \setminus 2\bN_0$ or $c_\gamma=0$ else.

\medskip

We now specialize to $\gamma > 0$, in which case, $\mathrm{F}_\gamma : [0,1] \to [1,\infty)$ is increasing, continuous, convex with $\mathrm{F}_\gamma(1) , \mathrm{F}_\gamma'(1) <\infty$. 
This implies that for any $\rho, r\in [0,1]$,
\[
0\le \mathrm{F}_\gamma(r)- \mathrm{F}_\gamma(\rho r) \le \mathrm{F}_\gamma'(1) (1-\rho) \le  \mathrm{F}_\gamma'(1)(1-\rho)\mathrm{F}_\gamma(r) 
\]
since $\mathrm{F}_\gamma(r)\ge 1$.
Combining this bound and \eqref{Theta}, with $c_\gamma = F'_\gamma(1)$, we obtain that (almost surely) for $n\in\bN$, 
\[
\big(1- c_\gamma   (1-\rho) \big) \Theta_n(1;\cdot) \le \Theta_n(\rho;\cdot)\le  \Theta_n(1;\cdot)  .
\]

Using the relationship \eqref{ulinearity}, by definition of $\Theta_n$, one has
\begin{equation}\label{Theta1}
\Theta_n(1;\cdot) = \bE_n |\Phi_{n,\infty}^{*\lambda_n}|^\gamma  = \bE  |\Phi_{n,\infty}^{*}|^\gamma  
\end{equation}
where the last equality follows from the fact that conditionally on $\F_n$,
$\Phi_{n,\infty}^{*\lambda_n}$ has the same law as $\Phi_{n,\infty}^{*}$ (by rotation-invariance) and $\Phi_{n,\infty}^{*}$ is independent of $\F_n$. 
Using the previous estimates with $\rho=\rho_n(z)$ (which is $\F_n$-measurable),
by \eqref{Phiinftyproj}, we conclude that (almost surely),  for any $\gamma\ge 0$ and $z\in\bD$,
\[
\big(1-  c_\gamma   (1-\rho_n(z)) \big)
|\Phi_n^*(z)|^\gamma  \bE |\Phi_{n,\infty}^{*}|^\gamma  \le \bE_n |\Phi_{\infty}^*(z)|^\gamma \le   |\Phi_n^*(z)|^\gamma  \bE |\Phi_{n,\infty}^{*}|^\gamma. \qedhere
\]
\end{proof}

\begin{remark} \label{rk:UB}
If $\gamma<0$, by \eqref{F1}, the function $\mathrm{F}_\gamma$ is still increasing on $[0,1)$ so that by \eqref{Theta}, we have $\Theta_k(\rho;z)\le  \Theta_k(1;z)$ for any $\rho\in[0,1]$ and $z\in\bD$. 
Thus, by \eqref{Phiinftyproj} and \eqref{Theta1}, we still have the estimate for $\gamma>0$ and $n\in\bN$,  
\[
\bE_n |\Phi_{\infty}^*(z)|^{-\gamma} \le  |\Phi_n^*(z)|^{-\gamma} \bE  |\Phi_{n,\infty}^{*}(z)|^{-\gamma}, \qquad z\in\bD. 
\]
\end{remark}

To finish the proof of convergence, we will also need the following  estimates.

\begin{lemma} \label{lem:BD}
For every $n\in\bN$, $\max_{z\in\overline{\bD}}|\Im \varphi_n(z)| <2n$ and  $\max_{z\in\overline{\bD}} \Re \varphi_n(z)< n$. Moreover, for any $\gamma\in\bR$ with $|\hat\gamma|<1$,  one has for $z\in\overline\bD$, 
\[
\bE |\Phi^*_n(z)|^\gamma \le \bE  |\Phi^*_n(1)|^\gamma
={\textstyle \prod_{k=0}^{n-1}} \bE [ |1 - \alpha_k|^\gamma] <\infty.
\]

\end{lemma}

\begin{proof}
Here $n\in\bN$ is fixed. 
Recall that $\varphi_n$ is an analytic function in a neighborhood of $\bD$ and \eqref{S3} holds. 
Since $w\in\bD\mapsto |\Im\log(1+w)|$ is bounded by $\pi/2$, one has $\max_{z\in\overline{\bD}}|\Im \varphi_n(z)| < 2n$. 
Similarly,  $w\in\bD\mapsto  \log|1+w|$ is bounded above by $\log 2$ so $\max_{z\in\overline{\bD}} \Re \varphi_n(z)< n$. However, there is no  deterministic lower-bound for the real~part.

Let $\gamma\in\bR$. 
By subharmonicity (the function $x\in\bR\mapsto e^{\gamma x}$ is convex), one has for $z\in\bD$, 
\[
|\Phi^*_k(z)|^\gamma \le \int_{\bU} |\Phi^*_k(u)|^\gamma P(z;\dd u)
\]
where $P$ denotes Poisson kernel of $\bD$. Taking expectation, by rotation-invariance, we obtain 
\[
\bE |\Phi^*_k(z)|^\gamma \le \bE  |\Phi^*_k(1)|^\gamma , \qquad z\in\overline\bD .
\]
This estimate holds for all $\gamma\in\bR$, but the RHS may be $+\infty$ if $\gamma$ is too negative. 

For $k\in\bN_0$, $B_k(1)\in\bU$ is $\F_k$-measurable, independent of $\alpha_k$, so by rotation-invariance,
\[
\Phi^*_n(1) \equiv {\textstyle \sum_{j=0}^{n-1}} \log(1-\alpha_j) .
\]
This shows that for $\gamma\in\bR$, $\bE  |\Phi^*_n(1)|^\gamma
={\textstyle \prod_{j=0}^{n-1}} \bE [ |1 - \alpha_j|^\gamma] $. 
{\blue These quantities are finite if and only if $\gamma>-1-\beta_0$; see \eqref{expmom} in the Appendix and fact that $\{\beta_j\}$ are increasing. 
Note that $1+\frac\beta2 \ge \sqrt{2\beta}$ for any $\beta>0$, so the condition $\hat\gamma = \gamma/\sqrt{2\beta}>-1$ implies that $\gamma>-1-\beta_0$
with $\beta_0 = \frac\beta2$. 
Then, we conclude that $\bE  |\Phi^*_k(1)|^\gamma < \infty$ if $|\hat\gamma|<1$.} Finally, note that 
$1+\frac\beta2 - \sqrt{2\beta} = (1- \sqrt{\beta/2})^2$, so this inequality is sharp at the critical value~$\beta=2$.
\end{proof}

We now turn to the proof of \eqref{mart1}.  
First, taking expectation in Lemma~\ref{lem:bdreal}, one has for $n\in\bN$ and $r<1$, 
\[
\bE\big[ |\Phi_n^*(r)|^\gamma (1-c_\gamma(1-\rho_n(r)))\big] \le \frac{\bE |\Phi_{\infty}^*(r)|^\gamma}{\bE  |\Phi_{n,\infty}^{*}(r)|^\gamma}  \le \bE |\Phi_n^*(r)|^\gamma . 
\]
By Lemma~\ref{lem:BD}, $|\Phi_n^*(z)|\le e^n$ for $z\in\overline\bD$, and since $\rho_n(r) \in (0,1)$ and $\rho_n(r) \to 1$ as $r\to1$ (almost surely), by the dominated convergence theorem for $\gamma\ge 0$,
\[
\lim_{r\to 1}  \bE\big[ |\Phi_n^*(r)|^\gamma (1-c_\gamma(1-\rho_n(r)))\big] = \bE |\Phi_n^*(1)|^\gamma .
\]
Hence, for any $n\in\bN$, 
\begin{equation} \label{norm1}
\lim_{r\to 1} \frac{\bE |\Phi_{\infty}^*(r)|^\gamma}{\bE  |\Phi_{n,\infty}^{*}(r)|^\gamma} = \bE |\Phi_n^*(1)|^\gamma . 
\end{equation}
Using Lemma~\ref{lem:bdreal} again, with $|z|=r<1$, 
\[
\frac{|\Phi_n^*(z)|^\gamma}{\bE |\Phi_n^*(r)|^\gamma} \ge
\frac{\bE_n |\Phi_{\infty}^*(z)|^\gamma}{\bE |\Phi_\infty^*(r)|^\gamma} 
\frac{\bE |\Phi_\infty^*(r)|^\gamma}{\bE |\Phi_n^*(r)|^\gamma \bE  |\Phi_{n,\infty}^{*}(r)|^\gamma}
\]
so that by \eqref{norm1}, for any $f\in L^1(\bU,\bR_+)$, 
\[
\liminf_{r\to1}\int f(u) 
\frac{|\Phi_n^*(ru)|^\gamma}{\bE |\Phi_n^*(r)|^\gamma} \dd u \ge
\liminf_{r\to1} \bE_n \bigg( \int f(u) 
\frac{ |\Phi_{\infty}^*(ru)|^\gamma}{\bE |\Phi_\infty^*(r)|^\gamma} \dd u \bigg) .
\]
Using the (uniform) continuity of $\Phi_n^*$ on $\overline\bD$, the LHS limit exists and equals $\mu_{n}^{\gamma}(f)$. For the RHS, using the GMC convergence in $\bL^1$, see \eqref{GMC}, if $\hat\gamma<1$, 
\[
\lim_{r\to1} \bE_n \bigg( \int f(u) 
\frac{ |\Phi_{\infty}^*(ru)|^\gamma}{\bE |\Phi_\infty^*(r)|^\gamma} \dd u \bigg)  = \bE_n \mu^{\hat\gamma}(f) 
\]
almost surely. 
Thus, for $n\in\bN$, $f\in L^1(\bU,\bR_+)$ and  $0\le \hat\gamma<1$,  
\[
\mu_{n}^{\gamma}(f)\ge  \bE_n\mu^{\hat\gamma}(f) .
\] 
Since both random variables have the same expectation (equal to $f_0$), they are equal. 
This concludes the proof of Proposition~\ref{prop:phi} in case $\gamma\ge 0$. 

\subsection{Convergence for real part in case $ \gamma \le 0$}\label{sec:neg}
In this case, only the lower bound from Lemma~\ref{lem:bdreal} holds (see Remark~\ref{rk:UB}), so we need a different strategy to replace \eqref{norm1}. We rely on the following estimates.

\begin{proposition} \label{prop:asmom}
Recall Fact~\ref{fact:cara}.
Let  $ \gamma \in\bR $ with $|\hat\gamma|<1$ and $\delta>0$. It holds almost surely, for every $k\in\bN$, 
\[
\lim_{r\to 1} \dd(r)^{\hat\gamma^2} \sup_{|z|=r}\bE_k\big[ |\Phi^*_{k, \infty}(z)|^{\gamma} \1\{(1-\rho_k(z))|\mathfrak{g}_k(z)| >\delta \}\big] =0 .
\]
\end{proposition}

We postpone the proof of Proposition~\ref{prop:asmom} and we now proceed to deduce that for 
$0\le \hat\gamma <1$ and $f\in L^1(\bU,\bR_+)$, 
\begin{equation} \label{mart2}
\bE_n\mu^{-\hat\gamma}(f)
\ge  \mu_{n}^{-\gamma}(f) =\int  \frac{|\Phi_n^*(u)|^{-\gamma}}{\bE |\Phi_n^*(1)|^{-\gamma}} f(u)\dd u 
\end{equation}
Observe that according to Lemma~\ref{lem:BD}, the RHS is well-defined if the condition $\hat\gamma <1$ holds. Moreover, like in the previous section, since both random variables have the same expectation, \eqref{mart2} suffices to prove of Proposition~\ref{prop:phi}. 

\medskip

Let $\delta\in (0,\frac12]$ and $n\in\bN$. We consider the events $\A(z;\delta):= \big\{(1-\rho_n(z))|\mathfrak{g}_n^{\lambda_n}(z)| \le\delta \big\}$ for $z\in\bD$. 
Since $\Re\mathfrak{g}_n^{\lambda_n}>-\frac12$, by \eqref{decoup2}, one has for $\gamma\ge 0$, 
\[\begin{aligned}
|\Phi^*_{\infty}|^{-\gamma} 
& =  |\Phi^*_n |^{-\gamma} |\Phi^{\lambda_n *}_{n,\infty}|^{-\gamma} \big|1+ (1-\rho_n)\mathfrak{g}_n^{\lambda_n} \big|^{-\gamma}  \\
&\ge (1+ \delta)^{-\gamma}  |\Phi^*_n |^{-\gamma} |\Phi^{\lambda_n *}_{n,\infty}|^{-\gamma}  \1_\A
\end{aligned}\]

Recall that $(\rho_n,\lambda_n)$ are $\F_n$ measurable and that, by rotation-invariance, conditionally on $\F_n$,  $(\Phi^{\lambda_n *}_{n,\infty}, \Psi^{\lambda_n*}_{n,\infty}) \equiv (\Phi^{*}_{n,\infty}, \Psi^{*}_{n,\infty})$ as processes. This implies that inside $\bD$, 
\[
\bE_n |\Phi^*_{\infty} |^{-\gamma} 
\ge (1+ \delta)^{-\gamma} |\Phi^*_n |^{-\gamma}  \bE_n\big[ |\Phi^{*}_{n,\infty} |^{-\gamma}  \1_{\A'}\big] 
\]
where $\A'(z;\delta):= \big\{(1-\rho_n(z))|\mathfrak{g}_n(z)| \le\delta \big\}$ for $z\in\bD$.
{\blue
Here, we used that by rotation-invariance ($\lambda_n\in\bU$ is $\F_n$-measurable), conditionally on $\F_n$, $(\Phi^{\lambda_n *}_{n,\infty},\mathfrak{g}_n^{\lambda_n})\equiv (\Phi^{*}_{n,\infty},\mathfrak{g}_n)$. 
Moreover, by Proposition~\ref{prop:asmom},
\[
\liminf_{r\to 1}\big( \dd(r)^{\hat\gamma^2}\bE_n\big[ |\Phi^{*}_{n,\infty} |^{-\gamma}  \1_{\A'}\big]  \big) = \liminf_{r\to 1}\big( \dd(r)^{\hat\gamma^2}\bE_n\big[ |\Phi^{*}_{n,\infty} |^{-\gamma}\big]  \big)=
\liminf_{r\to 1}\big( \dd(r)^{\hat\gamma^2}\bE|\Phi^{*}_{n,\infty} |^{-\gamma} \big)
\]
since $\Phi^{*}_{n,\infty}$ is independent of $\F_n$.}
Hence, as $|\Phi^*_n|$ is continuous and positive on $\overline{\bD}$, we obtain that almost surely; for $u\in\bU$, 
\begin{equation} \label{LBreal}
\liminf_{r\to 1}\big( \dd(r)^{\hat\gamma^2}\bE_n  |\Phi^*_{\infty}(ru)|^{-\gamma} \big)
\ge (1+ \delta)^{-\gamma} |\Phi^*_n(u) |^{-\gamma} \liminf_{r\to1}\big(  \dd(r)^{\hat\gamma^2} \bE |\Phi^{*}_{n,\infty}(r)|^{-\gamma}   \big).
\end{equation}
Now, by Remark~\ref{rk:UB}, for $\gamma\ge 0$ and $r<1$, 
\[
\dd(r)^{-\hat\gamma^2} =  \bE |\Phi_{\infty}^*(r)|^{-\gamma} \le \bE |\Phi_n^*(r)|^{-\gamma} \bE  |\Phi_{n,\infty}^{*}(r)|^{-\gamma} .
\]
Then, by Lemma~\ref{lem:BD},
\begin{equation} \label{UBreal}
 \bE  |\Phi^*_n(1)|^{-\gamma}
\liminf_{r\to1}\big(  \dd(r)^{\hat\gamma^2} \bE |\Phi^{*}_{n,\infty}(r)|^{-\gamma}   \big)  
\ge \liminf_{r\to1}\big(  \dd(r)^{\hat\gamma^2} \bE |\Phi^{*}_{n,\infty}(r)|^{-\gamma}  
\bE |\Phi_n^*(r)|^{-\gamma}  \big)  \ge 1 . 
\end{equation}
Combining \eqref{LBreal} (which holds for any $\delta>0$) with \eqref{UBreal}, this implies  that almost surely, for $u\in\bU$, 
\[
\liminf_{r\to 1}\big( \dd(r)^{\hat\gamma^2}\bE_n  |\Phi^*_{\infty}(ru)|^{-\gamma} \big)
\ge \frac{|\Phi^*_n(u) |^{-\gamma}}{ \bE  |\Phi^*_n(1)|^{-\gamma}} .
\]

Thus, by Fatou's Lemma, for $f\in L^1(\bT,\bR_+)$,  
\[
\liminf_{r\to 1} \bigg( \int 
\frac{\bE_n  |\Phi^*_{\infty}(ru)|^{-\gamma}}{\bE  |\Phi^*_{\infty}(r)|^{-\gamma}}  f(u) \dd u \bigg)
\ge \int \frac{|\Phi_n^*(u)|^{-\gamma}}{\bE |\Phi_n^*(1)|^{-\gamma}} f(\theta)\dd u
=\mu_{n}^{-\gamma}(f) .
\]
If $\hat\gamma <1$, using the GMC convergence \eqref{GMC},  the LHS equals $\bE_n \mu^{-\hat\gamma}(f) $ which concludes the proof of  \eqref{mart2}. 
\medskip

We now return to the proof of Proposition~\ref{prop:asmom}.
This relies on Lemma~\ref{lem:gmom} the following estimates for the moments of $\Phi_{k,\infty}^*$. 
The proof is based on a change of measure to \emph{shift the sequence of Verblunsky coefficients}. 
Compared to \eqref{mom}, we expect that a sharp estimates holds with $\epsilon=0$. 

\begin{lemma} \label{lem:bdshift}
Fix $\gamma\in\bR$. 
For any $k\in\bN$ and $\epsilon>0$, there exists $C_{k,\epsilon}$ so that for $z\in\bD$, 
\[
\bE |\Phi_{k,\infty}^*(z)|^{\gamma} \le C_{k,\epsilon} \dd(z)^{-\hat\gamma^2(1+\epsilon)}  . 
\]
\end{lemma}
\begin{proof}
Let $\varkappa:=\sqrt{2/\beta}$ and fix $k\in\bN$. 
For C$\beta$E (Fact~\ref{fact:Verb}), for any $j\in\bN_0$, the random variable $( 1 - |\alpha_j|^2)^{1/\varkappa^2}$ has p.d.f. $x\in [0,1] \mapsto (j+1) x^j$.
Let for $n\in\bN_0$, 
\begin{equation} \label{Qmart}
\mathcal{Q}_{k,n} := \prod_{j=0}^{n-1} \frac{( 1 - |\alpha_j|^2)^{k/\varkappa^2}}{\bE( 1 - |\alpha_j|^2)^{k/\varkappa^2}}.
\end{equation}
Clearly $\big\{\mathcal{Q}_{k,n}\big\}_{n\in\bN_0}$ is a positive $\mathcal{F}_n$-martingale and  there is a numerical constant $c$ so that for any $\gamma \ge 0$, 
\[\begin{aligned} 
\bE \mathcal{Q}_{k,n}^\gamma &= \prod_{j=0}^{n-1} \bigg(\frac{j+k+1}{j+1} \bigg)^\gamma \bigg(\frac{j+1}{j+\gamma k+1} \bigg) \le  \prod_{j=0}^{n-1} \exp\bigg( \frac{\gamma^2k^2}{2(j+1)^2}\bigg) \le \exp \big( c \gamma^2 k^2 \big)  
\end{aligned}\]
using that by convexity, $(1+\lambda)^\gamma \le e^{\gamma\lambda}$ and $(1+\gamma\lambda)^{-1} \le e^{-\gamma\lambda+(\gamma\lambda)^2/2} $ for $              \lambda,\gamma\ge 0$. 

In particular, $\{\mathcal{Q}_{k,n}\}_{n\in\bN_0}$ is uniformly bounded in $L^2$ and  for $n\in\bN_0$, 
\[
\mathcal{Q}_{k,n} 
= \bE_{n} \mathcal{Q}_{k,\infty}  ,\qquad  
\mathcal{Q}_{k,\infty} = \lim_{n\to\infty} \mathcal{Q}_{k,n}  \, \text{(almost surely)}.
\]
Moreover, if $\gamma>1$, then as $n\to\infty$
\[
\bE \mathcal{Q}_{k,n}^\gamma \nearrow \bE \mathcal{Q}_{k,\infty}^\gamma  .
\]
The idea is to make a change of measure
\(
\frac{d\mathbb{Q}_k}{d\mathbb{P}} = \mathcal{Q}_{k,\infty} .
\)
The sequence $(\alpha_j)_{j\ge 0}$ under $\mathbb{Q}_k$  has the same law as $(\alpha_{j+k})_{j\ge 0}$ under $\mathbb{P}$. 
Then, for $\gamma\in\bR$ and $z\in\bD$, 
\[  \bE|\Phi^*_{k,\infty}(z) \big|^\gamma
= \bE\big[ |\Phi^*_{\infty}(z) \big|^\gamma \mathcal{Q}_{k,\infty} \big] .
\]
By H\"older's inequality and using that $\bE |\Phi^*_{\infty}(z)|^\gamma = \dd(z)^{-\hat\gamma^2}$, we have for any $\epsilon>0$, 
\[\begin{aligned}
\bE|\Phi^*_{k,\infty}(z) \big|^\gamma
& \le  \bE\big[ |\Phi^*_{\infty}(z) \big|^{\gamma(1+\epsilon)}\big]^{\frac{1}{1+\epsilon}} 
\bE\big[ \mathcal{Q}_{k,\infty}^{1+1/\epsilon} \big]^{\frac{\epsilon}{1+\epsilon}}  \\
& = C_{k,\epsilon} \dd(z)^{-\hat\gamma^2(1+\epsilon)} . \qedhere
\end{aligned}\]
\end{proof}

\medskip

\begin{proof}[Proof of Proposition~\ref{prop:asmom}]
We record a version of H\"older's inequality; if $Z$ is a non-negative random variable, for any event $\A$ and any $\epsilon>0$,
\begin{equation}\label{Holder0}
\bE[Z \1\{\A\}]^{1+\epsilon} \le \bE[Z^{1+\epsilon}] \bP[\A]^\epsilon . 
\end{equation}
Fix  $k\in\bN$, $\gamma\in\bR$,  $\delta>0$ and set $\Phi^*_{k, \infty}= \Phi^*_{k, \infty}(z)$, $\mathfrak{g}_k=\mathfrak{g}_k(z)$,
$\rho_k = \rho_k(z)$, $\dd=\dd(z)$  and $r=|z|$ for $z\in\bD$.\\
By \eqref{Holder0} and Markov's inequality, for any $\epsilon_1, \epsilon_2>0$, 
\[\begin{aligned}
\bE_k\big[ |\Phi^*_{k, \infty}|^\gamma \1\{(1-\rho_k)|\mathfrak{g}_k| >\delta \}\big]^{1+\epsilon_1}
&\le \bE_k\big[ |\Phi^*_{k, \infty}|^{\gamma(1+\epsilon_1)}\big]
\bP_k\big[ (1-\rho_k)|\mathfrak{g}_k| >\delta\big]^{\epsilon_1}\\
&\le \delta^{-(1+\epsilon_2)} \bE\big[ |\Phi^*_{k, \infty}|^{\gamma(1+\epsilon_1)}\big] \bE\big[ |\mathfrak{g}_k|^{1+\epsilon_2}\big]^{\epsilon_1} (1-\rho_k)^{(1+\epsilon_2)\epsilon_1}
\end{aligned}\]
where we used that $\rho_k$ is $\F_k$-measurable and  $(\Phi^*_{k, \infty},\mathfrak{g}_k)$ are independent of $\F_k$.
Then, using the estimates from Lemmas~\ref{lem:gmom} and~\ref{lem:bdshift}, there exists a constant $C=C_{k,\epsilon_1,\epsilon_2.\epsilon_3, \delta}$ so that for $\epsilon_1, \epsilon_2, \epsilon_3>0$, 
\[
\bE_k\big[ |\Phi^*_{k, \infty}|^\gamma \1\{(1-\rho_k)|\mathfrak{g}_k| >\delta \}\big]^{1+\epsilon_1}
\le C \dd^{-\hat\gamma^2(1+\epsilon_1)^2(1+\epsilon_3)- \epsilon_2 \epsilon_1}  (1-\rho_k)^{(1+\epsilon_2)\epsilon_1}.
\]
so that
\[
\dd^{\hat\gamma^2}\bE_k\big[ |\Phi^*_{k, \infty}|^\gamma \1\{(1-\rho_k)|\mathfrak{g}_k| >\delta \}\big]
\le C \dd^{-\hat\gamma^2 (\epsilon_1+\epsilon_3(1+ \epsilon_1))- \frac{\epsilon_2\epsilon_1}{1+\epsilon_1}}  (1-\rho_k)^{\frac{1+\epsilon_2}{1+\epsilon_1}\epsilon_1}.
\]

In the subcritical regime, $\hat\gamma^2<1$, we can choose small parameters $\epsilon_1, \epsilon_2, \epsilon_3>0$ so that
\[
\hat\gamma^2 (\epsilon_1+\epsilon_3(1+ \epsilon_1))+ \frac{\epsilon_2\epsilon_1}{1+\epsilon_1} < \frac{1+\epsilon_2}{1+\epsilon_1}\epsilon_1 .
\]
Hence, by \eqref{Lip} ($\rho_k$ is Lipschitz-continuous), we conclude that almost surely, 
\[
\lim_{r\to 1} \sup_{|z|=r} \Big\{\dd(r)^{-\hat\gamma^2 (\epsilon_1+\epsilon_3(1+ \epsilon_1))- \frac{\epsilon_2\epsilon_1}{1+\epsilon_1}}  (1-\rho_k(z))^{\frac{1+\epsilon_2}{1+\epsilon_1}\epsilon_1}\Big\} =0 . 
\]
This proves the claim. 
\end{proof}

\begin{remark}
The argument shows that one can take $\delta = \dd(r)^{\epsilon}$ for a small $\epsilon>0$ depending on $|\hat\gamma|$ in the statement of  Proposition~\ref{prop:asmom}. 
\end{remark}

\subsection{Convergence for imaginary part} \label{sec:imag}
In this section, we prove the second part of Proposition~\ref{prop:phi}. We rely on the same method as in Section~\ref{sec:real} and our goal is to show that, for $\gamma\in\bR$ with $|\hat\gamma|<1$ and $f\in L^1(\bU,\bR_+)$, 
\begin{equation} \label{mart3}
\bE_n\nu^{\hat\gamma}(f)
\ge  \nu_{n}^{\gamma}(f) =\int \frac{e^{\gamma \Im\varphi(u)}}{\bE e^{\gamma \Im\varphi_n(1)}} f(u)\dd u . 
\end{equation}
Since both random measures have the same expectation 
($\bE\nu^{\hat\gamma}(f)=\bE\nu_{n}^{\gamma}(f)= \nu^0(f)$ for $|\hat\gamma|<1$), this bound suffices to conclude that 
$\bE_n\nu^{\hat\gamma}(f) = \nu_{n}^{\gamma}(f)$ (almost surely). 

\medskip

We need the counterpart of Lemma~\ref{lem:bdreal} for the imaginary part. 

\begin{lemma} \label{lem:bdim}
There is a (deterministic) constant $C$ (depending only on $\beta$) so that for any $\gamma\in\bR$, with $|\hat\gamma|<1$, it holds (almost surely) for any $n\in\bN$ and $z\in\bD$, 
\[
0\le  e^{\gamma\Im\varphi_n(z)}  \bE\big[e^{\gamma \Im \varphi_{n,\infty}(z)}\big] -\bE_n\big[e^{\gamma\Im \varphi(z)}\big]  \le C \sqrt{1-\rho_n(z)}  \bE\big[e^{\gamma \Im \varphi_{n,\infty}(z)}\big]  
\]
where $\varphi_{n,\infty} := \log \Phi^*_{n,\infty}$. 
\end{lemma}

\begin{proof}
Here $n\in\bN$ is fixed. 
We start by taking the $\log(\cdot)$ of \eqref{decoup2}, we claim that inside $\bD$, 
\[
\varphi = \varphi_n +
\log\big( \tfrac{\Phi_{n,\infty}^{*}+\Psi_{n,\infty}^{*}}{2} + \rho_n \lambda_n  
\tfrac{\Phi_{n,\infty}^{*}-\Psi_{n,\infty}^{*}}{2}  \big) 
\]
with the interpretation that, 
\begin{equation} \label{decoup3}
\log\big( \tfrac{\Phi_{n,\infty}^{*}+\Psi_{n,\infty}^{*}}{2} + \rho_n \lambda_n  \tfrac{\Phi_{n,\infty}^{*}-\Psi_{n,\infty}^{*}}{2}  \big) 
=  \log \Phi^*_{n,\infty}
+ \log\big(1+\mathfrak{g}_n\big)
+ \log\big(1+  \rho_n \lambda_n \tfrac{\Phi_{n,\infty}^{*}-\Psi_{n,\infty}^{*}}{\Phi_{n,\infty}^{*}+\Psi_{n,\infty}^{*}}\big) . 
\end{equation}
In \eqref{decoup3}, by Fact~\ref{fact:cara}, all $\log(\cdot)$ are well-defined and vanish at $z=0$. 

\medskip

Fix $\gamma\in\bR$. 
We define for $\rho\in[0,1]$ and $z\in\bD$, 
\[
\Upsilon_n (\rho;z) : = \bE_n \Big[\exp\Big(\gamma \Im \log\big( \tfrac{\Phi_{n,\infty}^{*}(z)+\Psi_{n,\infty}^{*}(z)}{2} + \rho \lambda_n(z)  
\tfrac{\Phi_{n,\infty}^{*}(z)-\Psi_{n,\infty}^{*}(z)}{2}  \big)\Big)\Big] 
\]
where the $\log(\cdot)$ is interpreted as in \eqref{decoup3}. 
In particular, since $\rho_n$ is $\F_n$ measurable, it holds for $z\in\bD$, 
\begin{equation}  \label{Upsilon0}
\bE_ne^{\gamma\Im \varphi(z)} = e^{\gamma\Im\varphi_n(z)}
\Upsilon_n (\rho_n(z);z) .
\end{equation}

Observe that using the relationships \eqref{ulinearity}, by rotation-invariance, one has
\begin{equation} \label{Upsilon1}
\Upsilon_n (1;\cdot) = \bE_n\Big[\exp\big(\gamma \Im \log \Phi_{n,\infty}^{ \lambda_n *}\big) \Big]
=  \bE_n\Big[\exp\big(\gamma \Im \log \Phi_{n,\infty}^{*}\big) \Big]
= \bE\Big[\exp\big(\gamma \Im \varphi_{n,\infty}\big) \Big] .
\end{equation}
At last, we used $\Phi_{n,\infty}^{*}$ is independent from $\F_n$ and 
$\log \Phi_{n,\infty}^{*} = \varphi_{n,\infty}$ is well-defined in $\bD$. 

Similarly, for every $u\in\bU$, $(\Phi_{k,n}^{u*},\Psi_{k,n}^{u*})_{n\ge 0}$ are  independent of $\F_k$ and
$(\Phi_{k,n}^{u*},\Psi_{k,n}^{u*})_{n\ge 0} \equiv (\Phi_{k,n}^{*}, \Psi_{k,n}^{*})_{n\ge 0}$ so that for $k\in\bN$, 
\[\begin{aligned}
\Upsilon_k(\rho;\cdot)
&= \bE_k\Big[\exp\big(\gamma \Im \log\big( \tfrac{\Phi_{k,\infty}^{u*}+\Psi_{k,\infty}^{u*}}{2} + \rho \lambda_k
\tfrac{\Phi_{k,\infty}^{u*}-\Psi_{k,\infty}^{u*}}{2}  \big)\big)\Big]  \\
& =  \bE_k\Big[\exp\Big(\gamma \Im \log\big( \tfrac{\Phi_{k,\infty}^{*}+\Psi_{k,\infty}^{*}}{2}\big) + \Im \log\big( 1+ \rho \lambda_nu
\tfrac{\Phi_{k,\infty}^{*}-\Psi_{k,\infty}^{*}}{\Phi_{k,\infty}^{*}+\Psi_{k,\infty}^{*}} \big)\Big)\Big] 
\end{aligned}\]
using \eqref{decoup3} with 
$\log\big( \tfrac{\Phi_{k,\infty}^{*}+\Psi_{k,\infty}^{*}}{2}\big) 
= \log \Phi^*_{k,\infty}+ \log\big(1+\mathfrak{g}_k\big)$. 
Then, averaging over $u\in\bU$, we obtain
\begin{equation} \label{Upsilon2}
\Upsilon_n (\rho;\cdot)
=  \bE_n\Big[\exp\big(\gamma \Im \log\big( \tfrac{\Phi_{n,\infty}^{*}+\Psi_{n,\infty}^{*}}{2}\big) \mathrm{G}_\gamma\Big(\rho
\big|\tfrac{\Phi_{n,\infty}^{*}-\Psi_{n,\infty}^{*}}{\Phi_{n,\infty}^{*}+\Psi_{n,\infty}^{*}}\big|\Big)\Big]
\end{equation}
where, for $\gamma\in\bR$,
\begin{equation}\label{G1}
\mathrm{G}_\gamma : r\in [0,1)  \mapsto \int_{\bU} \exp\big(\gamma  \Im \log( 1+ ru)  \big) \dd u . 
\end{equation}
We record the following properties of the function $\mathrm{G}$. 

\begin{fact} \label{fact:G}
For $\gamma\in\bR$, the function $\mathrm{G}_\gamma : [0,1) \to\bR_+$, is smooth, bounded, increasing, convex, with $\mathrm{G}_\gamma(0)=1$. 
For $\alpha<1$ and $|\hat\gamma|<1$,  $\mathrm{G}_\gamma$ is $\alpha$-H\"older continuous with a constant $C$ depending only on $\alpha,\beta$. 
\end{fact}

We postpone the (elementary) proof of Fact~\ref{fact:G} to the end of this section and we record the following consequences for \eqref{Upsilon2}. 
Since $\mathrm{G}_\gamma$ is increasing, by \eqref{Upsilon0}, one has for $z\in\bD$
\[
\bE_ne^{\gamma\Im \varphi (z)} \le e^{\gamma\Im\varphi_n(z)}
\Upsilon_n (1;z) 
\]
Moreover, since $\mathrm{G}_\gamma$ is H\"older-continuous, one has for $\rho\in[0,1]$, 
\[
\mathrm{G}_\gamma\Big(\rho
\big|\tfrac{\Phi_{k,\infty}^{*}-\Psi_{k,\infty}^{*}}{\Phi_{k,\infty}^{*}+\Psi_{k,\infty}^{*}}\big|\Big) \ge \mathrm{G}_\gamma\Big(
\big|\tfrac{\Phi_{k,\infty}^{*}-\Psi_{k,\infty}^{*}}{\Phi_{k,\infty}^{*}+\Psi_{k,\infty}^{*}}\big|\Big) - C\sqrt{1-\rho} .
\]
Since $\mathrm{G}_\gamma(0)=1 \le \mathrm{G}_\gamma(1) $, this implies that for some deterministic constant $C$, 
\[
\Upsilon_n (\rho_n;\cdot) \ge \Upsilon_n (1;\cdot)\big(1 - C\sqrt{1-\rho_n} \big)
\]
and, by \eqref{Upsilon0}, one has for $z\in\bD$, 
\[
\bE_n e^{\gamma\Im \varphi(z)} \ge e^{\gamma\Im\varphi_n(z)}
 \Upsilon_n (1;z)\big(1-C\sqrt{1-\rho_n(z)}\big). 
\]
Combining these bounds, according to \eqref{Upsilon1}, this completes the proof.
\end{proof}

We return to the proof of \eqref{mart3}, which is straightforward using 
Lemma~\ref{lem:bdim}. 
Recall that for $\gamma\in\bR$ and $r<1$, $\bE e^{\gamma\Im \varphi(r)}  = \dd(r)^{-\hat\gamma}$. 
Using the lower-bound and taking expectation, we obtain
\begin{equation} \label{LBim}
1 \le  \liminf_{r\to\infty} \big( \dd(r)^{\hat\gamma}\bE\big[e^{\gamma \Im \varphi_{n,\infty}(r)}\big] \bE\big[e^{\gamma\Im\varphi_n(r)}\big] \big)  . 
\end{equation} 

Then, using the upper-bound and that 
$\bE\big[e^{\gamma\Im\varphi_n(r)}\big] \ge 1$ (by Jensen's inequality since $\bE \varphi_n(z) = \varphi_0(z) =0$ for $z\in\bD$), one has  for $f\in L^1(\bU,\bR_+)$,
\[
\int  \frac{\bE_n e^{\gamma\Im \varphi(ru)}}{\bE e^{\gamma\Im \varphi(r)}} f(u) \dd u 
\ge  \dd(r)^{\hat\gamma}\bE\big[e^{\gamma \Im \varphi_{n,\infty}(r)}\big] \bE\big[e^{\gamma\Im\varphi_n(r)}\big]\bigg( \int  \frac{ e^{\gamma\Im \varphi_n(ru)}}{\bE e^{\gamma\Im \varphi_n(r)}} f(u) \dd u 
-C\int \sqrt{1-\rho_n(ru)}  f(u) \dd u\bigg) .
\]
If $\hat\gamma<1$, using the GMC convergence (see \eqref{GMC}), the RHS converges almost surely to $\bE_n\nu^{\hat\gamma}(f)$ as $r\to 1$. 
Moreover, since $\varphi_n$ is (uniformly) continuous on $\overline\bD$, one has almost surely, as $r\to1$,
\[
\int  \frac{ e^{\gamma\Im \varphi_n(ru)}}{\bE e^{\gamma\Im \varphi_n(r)}} f(u) \dd u  \to \nu_n^\gamma(f). 
\] 
Thus, since $\rho_n(ru) \to 1$ uniformly for $u\in \bU$ (almost surely) as $r\to1$, the error converges to 0. By \eqref{LBim}, we conclude that 
\[
\bE_n\nu^{\hat\gamma}(f) \ge \int  \frac{ e^{\gamma\Im \varphi_n(u)}}{\bE e^{\gamma\Im \varphi_n(1)}} f(u) \dd u
\]
as required.  This completes the whole proof of Proposition~\ref{prop:phi}. 

\medskip

To finish this section, we prove Fact~\ref{fact:G} which is a simple analysis exercise. 

\begin{proof}[Proof of Fact~\ref{fact:G}]
Let $\gamma\in\bR$ and $\kappa = \gamma/2\i$. 
According to \eqref{G1}, one has for $r<1$, 
\[\begin{aligned}
\mathrm{G}_\gamma(r)  &=\int_{\bU} \exp\big(\kappa\log( 1 + ru)\big)\exp\big(-\kappa\log( 1 + r\overline{u})\big) \dd u \\
&=\int_{\bU} \frac{(1+ru)^\kappa}{(1+r\overline{u})^{\kappa}}\, \dd u . 
\end{aligned}\]
Using the binomial formula $(1+z)^\kappa = \sum_{k=0}^\infty {\kappa \choose k} z^k$, for $\kappa\in\bC$,  where both sides are analytic for $z\in\bD$, and  rotation-invariance of the uniform measure on $\bU$, we obtain for $r<1$, 
\[
\mathrm{G}_\gamma(r) = \sum_{k\ge 0}  {\kappa \choose k}   {-\kappa \choose k} r^{2k} .
\]
The coefficients ${\kappa \choose 0}=1$, {\blue ${\kappa \choose 1}=\kappa$ and for $k\in\bN_{\ge 2}$,}
\[
{\kappa \choose k} {-\kappa \choose k}
=  - \frac{\kappa(1-\kappa)\cdots(1-\kappa/(k-1))\kappa(1+\kappa)\cdots(1+\kappa/(k-1))}{k^2}
= - \frac{\kappa^2(1-\kappa^2)\cdots(1-\kappa^2/(k-1)^2)}{k^2} .
\]
Here $\kappa^2 = -\gamma^2/4 = -\hat\gamma^2 \frac\beta2$, so that for $r\in[0,1]$
\begin{equation}\label{G2}
\mathrm{G}_\gamma(r) = 1+ \sum_{k\ge 1}  \frac{c_{k-1}(\gamma)}{k^2} r^{2k} , \qquad 
\begin{cases}
c_0(\gamma) = \frac{\beta\hat\gamma^2}{2} \\
c_k(\gamma)= \frac{\beta\hat\gamma^2}{2}(1+\frac{\beta\hat\gamma^2}{2})(1+\frac{\beta\hat\gamma^2}{2\cdot4})\cdots(1+\frac{\beta\hat\gamma^2}{2\cdot k^2}) , &k\ge 1
\end{cases} .
\end{equation}
In particular, $0\le c_k(\gamma) \le c_\beta =\frac\beta2 \exp \frac\beta2$ for all $\gamma\in\bR$ with $|\hat\gamma|\le 1$. 

Formula \eqref{G2} implies that $\mathrm{G}_\gamma$, is smooth on $[0,1)$, increasing, convex, bounded on $[0,1]$  with $\mathrm{G}_\gamma(0)=1$. 
The fact that $\mathrm{G}_\gamma$ is $\alpha$-H\"older continuous is a consequence of the following observation. 
{\blue If $\alpha\in[0,1]$ and $r\in[0,1]$, one has for any  $k\in\bN$,
\begin{equation*} 
1- r^k \le \min( k(1-r),1) \le \min( k(1-r),1)^\alpha \le k^\alpha (1-r)^\alpha .  
\end{equation*}
Then, if $\alpha<1$, 
\[
0\le \mathrm{G}_\gamma(r) - \mathrm{G}_\gamma(\rho r) =   \sum_{k\ge 1}  \frac{c_{k-1}(\gamma)}{k^2} r^{2k}(1-\rho^{2k})  
\le (1-\rho)^\alpha \sum_{k\ge 1}  \frac{c_\beta}{k^{2-\alpha}} r^{2k}  \lesssim r^2 (1-\rho)^\alpha
\]
where the implied constant depends only on $\alpha,\beta$.}
\end{proof}

\section{Convergence for the characteristic polynomials} \label{sec:char}

\subsection{Proof of Theorem~\ref{thm:charpoly}} 
The main goal of this section will be to prove the following results. 

\begin{proposition} \label{prop:charpoly}
Let $f\in L^1(\bT,\bR)$ and $\gamma \in\bR$ with $|\hat\gamma|<1$. For a small $\delta>0$ $($depending on $|\hat\gamma|<1$$)$,
it holds in $\bL^{1+\delta}$ as $n\to\infty$, 
\[
\int_{\bT}\frac{e^{\gamma \mathcal{Y}_n(\theta)}}{\bE e^{\gamma \mathcal{Y}_n(0)}} f(\theta)\frac{\dd\theta}{2\pi}\to \nu^{\hat\gamma}(f)   
\qquad\text{and}\qquad
\int_{\bT} \frac{|\mathcal{X}_n(e^{\i\theta})|^\gamma}{\bE|\mathcal{X}_n(1)|^\gamma} f(\theta) \frac{\dd\theta}{2\pi} \to \mu^{\hat\gamma}(f) \quad\text{if $\gamma \ge 0$}. 
\]
The second limit still holds in probability if $\gamma >-1$. 
\end{proposition}

Proposition~\ref{prop:charpoly} imply Theorem~\ref{thm:charpoly} by standard arguments, we refer to \cite[Appendix B]{Lac} for instance.
As we explain in the introduction the conditions $|\hat\gamma|<1$ and  $\gamma>-1$ are necessary and sufficient and the main step is to obtain the following properties.

\begin{proposition} \label{prop:RL}
Recall that $(\varpi_k)_{k\ge 0}$ denotes the Pr\"ufer phases and the notations \eqref{mm}. 
Let $f\in L^\infty(\bT,\bR)$, $\gamma \in\bR$ with $|\hat\gamma|<1$ and $\kappa\in\bN$. 
It holds in probability as $n\to\infty$, 
\[
\int_{\bT}  \mathrm e^{i\kappa\varpi_{n}(\theta)} f(\theta)\mu^{\gamma}_n(\dd\theta) \to 0 
\qquad\text{and}\qquad
\int_{\bT}  \mathrm e^{i\kappa\varpi_{n}(\theta)} f(\theta)\nu^{\gamma}_n(\dd\theta) \to 0.
\]
\end{proposition}

Recall that $\eta \in \bT$ is a uniform random variable independent of $\F_\infty$. 
Proposition~\ref{prop:RL} implies that for any trigonometric polynomial $g :\bT\to \bR$ with $g_0 =0$,
\[
\int_{\bT} f(\theta) g(\varpi_{n}(\theta)+\eta)   \mu_{n}^{\gamma}(d\theta ) \to 0 \qquad\text{in probability as $n\to\infty$} . 
\]
Moreover, if $g\in L^1(\bT)$, by averaging over $\eta$ first ($\eta$ is independent of $\F_\infty$), 
, 
\[
\bE \bigg| \int_{\bT} f(\theta) g(\varpi_{n}(\theta)+\eta)   \mu_{n}^{\gamma}(d\theta) \bigg| \le \|f\|_{L^1}\|g\|_{L^1}
\]
since $\bE\mu_{n}^{\gamma}(d\theta)= \frac{\dd\theta}{2\pi}$ for $n\in\bN$, $\gamma\in\bR$. 

Consequently, by density (of trigonometric polynomials in  $L^1(\bT)$) and \eqref{cvg}, for any  $f, g\in L^1(\bT)$, if $|\hat\gamma|<1$, 
\begin{equation} \label{mart5}
\int_{\bT} f(\theta) g(\varpi_{n}(\theta)+\eta)   \mu_{n}^{\gamma}(d\theta ) \to g_0\cdot\mu^{\hat\gamma}(f)  \qquad\text{in probability as $n\to\infty$} . 
\end{equation}
An analogous result holds for the measures $\nu_{n}^{\gamma}$ and $\nu^{\hat\gamma}$.

\medskip

Let  $\mathrm f :\theta \in\bT \mapsto |1+ e^{\i\theta}|^\gamma$. 
$\mathrm f \in L^1(\bT)$ if and only if $\gamma >-1$ (it is positive with mean $\mathrm f_0 = F_\gamma(1)$, see \eqref{F1})
Then, according to \eqref{charpoly2}, for $\gamma>-1$ and $n\in\bN_0$, 
with $u=e^{\i\theta}$ for $\theta\in\bT$, 
\begin{equation} \label{mart6}
\bE_n |\mathcal{X}_{n+1}(u)|^\gamma=  |\Phi_n^*(u)|^\gamma \bE_n |1-e^{\i \eta+\i  \varpi_{n}(\theta)}|^\gamma = |\Phi_n^*(u)|^\gamma \mathrm f_0 .
\end{equation}
Thus, for $\gamma>-1$  and $n\in\bN_0$, 
\[
 \frac{|\mathcal{X}_{n+1}(u)|^\gamma}{\bE|\mathcal{X}_{n+1}(1)|^\gamma}
 = \mathrm{f}_0^{-1}\cdot \mathrm{f}(\varpi_{n}(\theta)+\eta)    \mu^{\gamma}_n(\dd u) .
\]
By \eqref{mart5}, we conclude that if $\gamma>-1$ with $|\hat\gamma|<1$, for any $f \in L^1(\bT,\bR_+)$, 
\begin{equation*} \label{charpolycvg}
\int \frac{|\mathcal{X}_n(e^{\i\theta})|^\gamma}{\bE|\mathcal{X}_n(1)|^\gamma} f(\theta)\frac{\dd\theta}{2\pi} \to \mu^{\hat\gamma}(f) \qquad\text{in probability as $n\to\infty$} .
\end{equation*}
In case $\mathrm f \le C$ for some constant, we can upgrade this convergence in $\bL^{1+\delta}$. 
Suppose that $f\ge 0$, $\gamma\ge 0$ and let $\mathrm{X}_n = \int \frac{|\mathcal{X}_{n+1}(e^{\i\theta})|^\gamma}{\bE|\mathcal{X}_{n+1}(1)|^\gamma} f(\theta)\frac{\dd\theta}{2\pi}$
for $n\in\bN$. 
Observe that according to Proposition~\ref{prop:phi}, 
\[
 \mathrm{X}_n \le C \mu_{n}^{\gamma}(f) = C \bE_n\mu^{\hat\gamma}(f).
\]
By \eqref{GMC}, since the limit $\mu^{\hat\gamma}(f)\in \bL^{1+\delta}$, by Jensen's inequality, 
\[
\bE\big[ \mathrm{X}_n^{1+\delta}\big]
\lesssim  \bE\big[ \mu^{\hat\gamma}(f)^{1+\delta}\big]<\infty .
\]
This yields the required uniform integrability condition. 

\medskip

For the imaginary part of the characteristic polynomial, according to \eqref{charpoly1}, for $z\in\bD$, 
\[
\Im\log\mathcal{X}_n(z) =  \Im\varphi_{n-1}(z) + \Im\log(1-e^{\i\eta}B_{n-1}(z)) . 
\]
The previous $\log(\cdot)$ is well-defined since $|B_{n}(z)|<1$ for $z\in\bD$ and $c\in\bN_0$.
Recall that $\mathcal{Y}_n(\theta)= \Im\log\mathcal{X}_n(e^{\i\theta})$ for $\theta\in\bT$ defined as in  \eqref{Y0}. 
Since $B_{n} = e^{\i\varpi_n(\theta)}$ for $\theta\in\bT$ with 
$\varpi_n(\theta) = (n+1)\theta  - 2\psi_n(\theta)$ for $\theta\in[0,2\pi]$, by continuity, we obtain
\begin{equation} \label{Y1}
\mathcal{Y}_{n+1}(\theta) = \psi_n(\theta) + \mathrm{h}(\varpi_{n}(\theta)+\eta)  
\end{equation}
with $\mathrm{h}$ as in \eqref{Y2}. 
Thus
\[
e^{\gamma \mathcal{Y}_{n+1}} = e^{\gamma\psi_n} \mathrm{g}(\varpi_{n}(\theta)+\eta)
\]
where $\mathrm{g}: \bT \mapsto  e^{\gamma\mathrm{h}(\theta)}$, satisfies 
$\mathrm{g}\in L^\infty(\bT)$ with $\mathrm{g}_0>0$.   
Again, $\bE_n e^{\gamma \mathcal{Y}_{n+1}} = \mathrm{g}_0\cdot e^{\gamma\psi_n} $, so that by \eqref{mart5},   if  $|\hat\gamma|<1$, for any $f \in L^1(\bT,\bR_+)$, 
\[
\int \frac{e^{\gamma \mathcal{Y}_n(\theta)}}{\bE e^{\gamma \mathcal{Y}_n(0)}}f(\theta)\frac{\dd\theta}{2\pi}\to \nu^{\hat\gamma}(f) . 
\qquad\text{in probability as $n\to\infty$} .
\]
This completes the proof of Proposition~\ref{prop:charpoly}. \qed 

\begin{remark}\label{rek:impart}
To check that formulae \eqref{Y0} and \eqref{Y1} are consistent, observe that both functions are smooth on $\bT \setminus\mathfrak{Z}_{n}$ and they jump by $-\pi$ on $\mathfrak{Z}_{n}$; see Proposition~\ref{prop:phase}. 
Moreover, taking a derivative, one checks that for $n\in\bN$, 
\[
\mathcal{Y}_{n}'(\theta) = \tfrac{n}{2}  =
\psi_{n-1}' + \varpi_{n-1}'(\theta)/2 , \qquad  \theta \notin \mathfrak{Z}_{n}, 
\]
where the previous equality follows from \eqref{phase0}. 
The value of both functions at $\theta=0$ satisfies $\mathcal{Y}_{n}(0) = \displaystyle\lim_{r\to 1}\Im\log \mathcal{X}_n(r)$ (almost surely). 
Formula  \eqref{Y1} implies that for any $n\in\bN$, 
\[\begin{cases}
\mathcal{Y}_{n}(\theta) = \psi_{n-1}(\theta) &\theta\in \mathfrak{Z}_{n} \\
|\mathcal{Y}_{n}(\theta) - \psi_{n-1}(\theta)|<\pi/2 &\theta\in\bT .
\end{cases}\]
In this sense, $\mathcal{Y}_{n}$ is a linear interpolation of the smooth function $\psi_{n-1}$ on $\bT$. 
\end{remark}

\subsection{Main steps}\label{sec:strat}
In this section, we go over the proof of Proposition~\ref{prop:RL}. 
We fix $f\in L^1(\bT,[0,1])$, $\gamma \in\bR$ with $|\hat\gamma|<1$ and $\kappa\in\bN$. The method is  the same for both $\mu^{\gamma}_n$ (real part) and $\nu^{\gamma}_n$ (imaginary part), so we focus on the real part. 
We introduce a small parameter $\epsilon_n>0$ and, by expansing the square, we split
\[
Z_n  : = \bigg| \int_{\bT} e^{\i\kappa\varpi_{n}(\theta)} f(\theta)\mu^{\gamma}_n(\dd\theta)  \bigg|^2  \le  Z_n^1 + Z_n^2
\]
where
\begin{equation} \label{split1}
 Z_n^1 :=  \iint_{{|\theta-\vartheta| < \epsilon_n}} \hspace{-.8cm}
\mu^{\gamma}_n(\dd\theta)  \mu^{\gamma}_n(\dd\vartheta) \, ,
\qquad
 Z_n^2 := \iint_{{|\theta-\vartheta| \ge \epsilon_n}} \hspace{-.8cm}
e^{\i\kappa(\varpi_{n}(\theta)-\varpi_{n}(\vartheta))} f(\theta)f(\vartheta)\mu^{\gamma}_n(\dd\theta)  \mu^{\gamma}_n(\dd\vartheta) \, . 
\end{equation}
Then, for any event $\A_n$ and $m\in\bN$,  by Jensen's inequality ($x\in\bR_+\mapsto x\wedge 1$ is concave and subadditive),
\begin{equation} \label{split2}
\bE\big[ Z_n  \wedge 1 \big]  \le \bP[\A_n^c]+
\bE\big[ |Z_n^1| \1\{\A_n\} \big] +\bE\big[\1\{\A_n\}  |\bE_mZ_n^2|\wedge 1  \big] . 
\end{equation}
In particular, if $ \bP[\A_n] \to 1$,  and we show that both
$\bE\big[ Z_n^1 \1\{\A_n\} \big] \to 0$ and $\bE_mZ_n^2 \to 0$ in probability, as $n\to\infty$. 
Then, we conclude that also $Z_n\to0$  in probability as $n\to\infty$, as required.

\begin{itemize}[leftmargin=*] \setlength\itemsep{.5em}
\item The \emph{local part} $ Z_n^1$ will be controlled on the event that the maximum of the field $\chi_n$ on $\bT$ is typical, in which case the density $\mu_n^\gamma$ is uniformly bounded. 
On this event, $ Z_n^1\to0$ provided that $\epsilon_n \le  n^{\delta-1}$
for a small $\delta>0$ (depending on $|\hat\gamma|<1$). 
\item For the part  $Z_n^2$, choosing a sequence $m(n)$ such that $n\ge m(n) \gg \epsilon_n^{-1}$, for $|\theta-\vartheta| \ge \epsilon_n$, the increments
\[
\big\{\varphi_n(\theta)-\varphi_m(\theta), \varphi_n(\vartheta)-\varphi_m(\vartheta)\big\}
\]
are approximately independent complex Gaussian random variables, with mean 0 and variance $\tfrac2\beta\log\frac{n}{m}$. 
To explain this property, if $m\gg 1$, one can approximate the C$\beta$E Verblunsky coefficients
$\alpha_k \approx G_k (1+\beta_k)^{-1/2}$ where $G_k$ are independent standard complex Gaussians.
Then, linearizing the recursion \eqref{S3}, one obtains a toy model
\[
\varphi_n(\theta) \approx \varphi_m(\theta) - {\textstyle \sum_{k=m}^{n-1}} \beta_k^{-1} G_k e^{\i(k-m)\theta+\i \varpi_m(\theta) } .
\]
Conditionally on $\F_m$, the RHS is a Gaussian process for all $\theta
\in\bT$ and its increments are independent if $m\gg |\theta-\vartheta|^{-1}$ because of the \emph{variation of the phase}. 
To control the approximation errors,  we introduce  events $\cB_m$ so that for $(\theta,\vartheta)\in\bT^2$ with $|\theta-\vartheta| \ge \epsilon_n$, 
\[
\big| \bE_m\big[\1\{\cB_m(\theta),\cB_m(\vartheta)\} e^{\i\kappa(\varpi_{n}(\theta)-\varpi_{n}(\vartheta))} \mu^{\gamma}_n(\theta)  \mu^{\gamma}_n(\vartheta)\big]\big| \lesssim
\bE\big[e^{2\i\kappa G\sqrt{\beta^{-1}\log\frac{n}{m}}}\big]^2 \mu^{\gamma}_m(\theta)  \mu^{\gamma}_m(\vartheta)
\]
where $G$ is a standard real Gaussian, and there is a $\varkappa>0$ so that $ \bP_m\big[\cB_m^{\rm c}(\theta)\big] \lesssim e^{-n^{-\varkappa}}$  uniformly for $\theta\in\bT$. Then  
\[
|\bE_mZ_n^2| \lesssim (\tfrac nm)^{-8\hat\kappa^2}  \mu^{\gamma}_m(f)^2 .
\]
Based on this heuristic, we conclude that $\bE_mZ_n^2\to 0$ (almost surely) as $n\to\infty$ provided that $|\hat\gamma|<1$ and we choose $m(n)$ appropriately within the range $\epsilon_n^{-1} \ll m(n) \ll n$. 
\end{itemize}

\medskip

In the rest of the subsection, we elaborate on the details of the method. 

\paragraph{\bf Control of $Z_n^1$} 
For $\theta\in\bT$, let $\chi_n(\theta) := \Re\varphi_n(e^{\i\theta})$ and $\psi_n(\theta) :=\Im\varphi_n(e^{\i\theta})$.
We need the following bounds. 

\begin{lemma} \label{lem:mom}
Let $\gamma\in\bR$. 
It holds for all  $n \in \bN$ and for any $\theta\in\bT$, 
\[
\bE[e^{\gamma \psi_n(\theta)}] = \exp\big(\hat\gamma^2 \log n + \O(1)\big) \qquad\text{and}\qquad
\bE[e^{\gamma \chi_n(\theta)}] = \exp\big(\hat\gamma^2 \log n + \O(1)\big) ,
\]
where the implied constants depend only on $(\gamma,\beta)$. 
\end{lemma}

Lemma~\ref{lem:mom} is proved in the Appendix~\ref{A:model} using the explicit distribution of the C$\beta$E Verblunsky coefficients. 
An alternative proof, albeit more technical, follows from the Gaussian approximations of Proposition~\ref{prop:approx1} below. 

\medskip

We fix a small $\delta>0$ and let $\epsilon_n = n^{\delta-1}$ in \eqref{split1}. 
We introduce the following events,
\[
\A_n(\delta) := \left\{  \max_{\theta\in\bT} |\chi_{n}(\theta)| , \max_{\theta\in\bT} |\psi_{n}(\theta)|   \le \tfrac{2+\delta}{\sqrt{2\beta}} \log n   \right\}. 
\]
By Theorem~\ref{thm:max}, $\bP[\A_n(\delta)] \to 1$ as $n\to\infty$. 
On this event, by Lemma~\ref{lem:mom}, we can bound the density of the random measure at hand,  
\[
\max\big\{ \mu_{n}^{\gamma}(\theta) : \theta\in\bT\big\} \lesssim  n^{(2+\delta)\hat\gamma  -\hat\gamma^2} \le n^{1+\delta -(1-\hat\gamma)^2}
\]
using that $|\hat\gamma|<1$ -- the last bound holds for $n$ sufficiently large. 

This implies that  
\begin{equation*}
|Z_n^1| \1\{\A_n\} \le  n^{1+\delta -(1-\hat\gamma)^2} \underset{|\theta-\vartheta| < \epsilon_n}{\iint_{\bT^2}}    \mu_{n}^{\gamma}(d\vartheta) \frac{d\theta}{2\pi}   \le n^{2\delta -(1-\hat\gamma)^2}  
\mu_{n}^{\gamma}(\bT) . 
\end{equation*}
In the subcritical regime ($\bE\mu_{n}^{\gamma}(\bT)=1$),  choosing $\delta>0$ sufficiently small (depending on $|\hat\gamma|<1$),  
we conclude that 
\begin{equation} \label{Z1}
\lim_{n\to\infty}\bE[ |Z_n^1| \1\{\A_n\}] =0 .
\end{equation}

\medskip

\paragraph{\bf Control of $Z_n^2$}
Recall \eqref{S3}--\eqref{phase0}. We  set for $n\ge m\ge 0$, 
\begin{equation}  \label{inc}
\begin{cases}
\varphi_{n,m}(u) :=  {\textstyle \sum_{k=m}^{n-1}}   \log(1-\alpha_k  e^{\i\varpi_k(\theta)}) \\
\chi_{n,m}(\theta)  = \Re\varphi_{n,m}(u) , \quad \psi_{n,m}(\theta)  = \Im \varphi_{n,m}(u)
\end{cases} \qquad u = e^{\i\theta},    \theta\in\bT , 
\end{equation}
where $\varpi_k(\theta) = (k+1)\theta-2 \psi_k(\theta)$ for $k\in\bN_0$.
Then, one decompose for any $n\ge m\ge 1$, 
\[
\chi_n = \chi_{n,0} = \chi_{n,m} + \chi_{m} , \qquad 
\psi_n = \psi_{n,0} = \psi_{n,m} + \psi_{m} . 
\] 
In particular, $(\chi_{m},\psi_{m})$ are $\F_m$-measurable and 
$\bE[e^{\gamma\chi_n(0)}] = \bE[e^{\gamma\chi_{n,m}(0)}]\bE[e^{\gamma\chi_m(0)}]$ (by independence of the Verblunsky coefficient), so that for $(\theta,\vartheta) \in\bT^2$ with $\theta\neq \vartheta$, 
\[
\big| \bE_m\big[\1\{\cB_n(\theta)\}\1\{\cB_n(\vartheta)\} e^{\i\kappa(\varpi_{n}(\theta)-\varpi_{n}(\vartheta))} \mu_{n}^{\gamma}(\theta) \mu_{n}^{\gamma}(\vartheta) \big] \big| \le 
\mathcal{W}_n(\vartheta,\vartheta)   \mu_{m}^{\gamma}(\theta) \mu_{m}^{\gamma}(\vartheta) 
\]
where, with $m\le n$, 
\begin{equation} \label{W1}
\mathcal{W}_n(\theta,\vartheta) : =
\frac{\big|\bE_m\big[\1\{\cB_n(\theta)\}\1\{\cB_n(\vartheta)\} e^{\i2\kappa(\psi_{m,n}(\theta)-\psi_{m,n}(\vartheta))}e^{\gamma \chi_{m,n}(\theta)+\gamma \chi_{m,n}(\vartheta)}\big]\big|}{\bE[e^{\gamma \chi_{m,n}(0)}]^2} .
\end{equation}
Then, by \eqref{split1}, 
\[
 |\bE_mZ_n^2|   \le \max\big\{ \mathcal{W}_n(\vartheta,\vartheta)  ; (\theta,\vartheta)\in\bT^2 : |\theta-\vartheta| \ge \epsilon_n \big\} 
\mu_{m}^{\gamma}(f)^2 + \mathcal E_n
\]
where the error satisfies on $\A_n$ (using a crude bound for the density $\mu_{n}^{\gamma} \le n$ if $|\hat\gamma|<1$), 
\[
\mathcal E_n \le 2 n^2 \bP_m[\cB_n^{\rm c}(\theta)] . 
\]

We will prove the following estimates.
\begin{proposition} \label{prop:W1}
Let $\epsilon_n = n^{\delta-1}$ for a small $\delta>0$ and let $m=m(n)= \lfloor n^{1-\frac\delta8}\rfloor $.
Then, almost surely, uniformly over $\big\{(\theta,\vartheta) \in\bT^2 ; |\theta-\vartheta| \ge \epsilon_n\big\}$,  
\[
\mathcal{W}_n(\theta,\vartheta) \lesssim n^{-\tfrac{\delta/5}{1+\beta}} .
\]
Moreover,  there is a $\varkappa>0$ (depending on $(\delta,\beta)$) so that almost surely $\bP_m[\cB_n^{\rm c}(\theta)] \lesssim e^{-n^\varkappa}$ uniformly for $\theta\in\bT$. 
\end{proposition}

Hence, by Proposition~\ref{prop:W1}, we conclude that for a small $c_\beta>0$ and large $C>0$, 
\[
\bE\big[\1\{\A_n\}  |\bE_mZ_n^2|\wedge 1  \big]  \le \O(n^{-c_\beta}) + \bP\big[\mu_{m}^{\gamma}(f)\ge C\big] .
\]
Note that the contribution from $\mathcal E_n $ is asymptotically negligible. Consequently, by Theorem~\ref{thm:phi}, for every $C>0$,
\[
\limsup_{n\to\infty}\bE\big[\1\{\A_n\}  |\bE_mZ_n^2|\wedge 1  \big]  \le  \bP\big[\mu^{\hat\gamma}(f)\ge C\big] .
\]
Taking $C\to\infty$, both sides vanish, then  going back to \eqref{split1} and \eqref{Z1}, we conclude that
\[
\lim_{n\to\infty}\bE\bigg[ \bigg| \int_{\bT} e^{\i\kappa\varpi_{n}(\theta)} f(\theta)\mu^{\gamma}_n(\dd\theta)  \bigg|^2 \wedge 1 \bigg] =0.
\]
This completes the proof of Proposition~\ref{prop:RL} for the real part.
The method for the imaginary part is exactly the same, one only needs to change \eqref{W1} to  
\[
\mathcal{W}_n(\theta,\vartheta)  =
\frac{\big|\bE_m\big[\1\{\cB_n(\theta)\}\1\{\cB_n(\vartheta)\} e^{\zeta\psi_{m,n}(\theta)+\overline\zeta\psi_{m,n}(\vartheta))}\big]\big|}{\bE[e^{\gamma \chi_{m,n}(0)}]^2} 
\]
where $\zeta=\gamma+2\i\kappa$, $m\le n$ and $(\theta,\vartheta) \in\bT^2$ with $\theta\neq \vartheta$. 
Then, the estimates from Proposition~\ref{prop:W1} hold true with the same events $\{B_n\}$. These events control both the real and imaginary part of $\varphi_{n,m}(\theta)$ for a fixed $\theta\in\bT$; see Section~\ref{sec:W}. 

\medskip

The rest of this section is organized as follows; 
\begin{itemize}[leftmargin=*] \setlength\itemsep{.5em}
\item In Section~\ref{sec:lin}, we linearize the recurrence \eqref{S3} and develop the necessary Gaussian approximations. This section is independent from the rest of the paper and we rely on the convention from the Appendix~\ref{A:conc}. 
\item In Section~\ref{sec:W}, we deduce Proposition~\ref{prop:W1} from these approximations.  
\item In Section~\ref{sec:Gauss}, we prove some Gaussian estimates which are required in the proof. 
\end{itemize}

\subsection{Gaussian coupling and Linearization.} \label{sec:lin}
The first step of the proof of Proposition~\ref{prop:W1} is to show that one can \emph{linearize} the recursion \eqref{S4} and then replace the Verblunsky coefficients $\{\alpha_k\}_{k\ge m}$ by independent (complex) Gaussian random variables, up to a small error, with overwhelming probability.
To perform such approximations, we rely on the following property of the C$\beta$E model which has been observed in \cite{CMN}. 
In this section, $c>0$ denotes a numerical constant and $C>0$ a constant which depends on the parameter $\beta>0$ of the model.
Recall Fact~\ref{fact:Verb}  and that $\beta_k = \tfrac\beta2 (k+1)$ for $k\in\bN_0$.  

\begin{fact} \label{fact:Gamma}

We can enlarge our probability space and the filtration \eqref{filt} so that for $k\in\bN_0$, 
\[
\F_{k} := \sigma\big(\alpha_{0},\Gamma_0,\cdots, \alpha_{k-1},\Gamma_{k-1}\big) ,  
\]
where  $\{\Gamma_k\}$ is a sequence independent random variables, independent of $\{\alpha_k\}$, and $\Gamma_k$ is Gamma-distributed  with p.d.f. $x\in\bR_+ \mapsto \Gamma(1+\beta_k)^{-1} x^{\beta_k} e^{-x}$ for $k\in\bN_0$. 
Then 
$G_k := \alpha_k \sqrt{\Gamma_k}  $ are i.i.d.~complex Gaussian with 
\[
\bE G_k = 0 , \quad  \bE G_k^2 = 0, \quad
\bE  |G_k|^2 = 1.
\]
\end{fact}

Then to linearize the recursion \eqref{S4}, we can work on the following event.

\begin{lemma} \label{lem:trunc1}
For any $0 \le \eta<1/2$, there is a $\epsilon=\epsilon(\eta)$ so that if $m$ is sufficiently large, 
\begin{equation} \label{trunc1}
\bP\big[ |\alpha_k| \le  k^{-\eta} \text{ for all }k\ge m\big]
\ge 1- \exp\big( -\beta m^{-\epsilon}\big) . 
\end{equation}
Moreover, there is a numerical constant $c$ so that
\(
\big\| \alpha_k - G_k/\sqrt{\beta_k} \big\|_{\Psi1} \le c   \beta_k^{-1}
\)
for all $k\in\bN_0$. 
\end{lemma}

\begin{proof}
First, we record that by \eqref{verb0}, for $k\in\bN_0$, 
\begin{equation} \label{exptb}
\bP\left[ |\alpha_k|^2 \ge t\right] \le  e^{-\beta_k t} ,\qquad t\in[0,1] .
\end{equation}
Equivalently, $\|\alpha_k\|_{\Psi2} \le c   \beta_k^{-1/2}$  for all $k\in\bN_0$. 
By a union bound, this implies that for some numerical constant $c>0$, 
\[
\bP\big[ |\alpha_k| > k^{-\eta}  \text{ for a }k\ge m\big]
\le  2 \sum_{k\ge m} \exp\big(- c\beta k^{1-2\eta} \big)
\]
which proves \eqref{trunc1}.

Let $\gamma_k := \sqrt{\Gamma_k/\beta_k} -1$ for $k\in\bN$.  
$\gamma_k$ has a p.d.f.~$\propto (x+1)^{\beta k}e^{-\beta_k(x+1)^2}$ on $[-1,\infty)$, using that $2\beta_k-1 =\beta k$. 
One has for $x\ge -1$, 
\[
(x+1)^{\beta k}e^{-(1+\beta_k)(x+1)^2} \le e^{- \beta_k x^2} . 
\]
By scaling, this shows that for some numerical constant, $\|\gamma_k\|_{\Psi2} \le c   \beta_k^{-1/2}$ for all $k\in\bN_0$. 
Now, we have 
\[
G_k/\sqrt{\beta_k}-\alpha_k =
\alpha_k\gamma_k  
\]
and the second claim follows from the bound
\(
\| \alpha_k\gamma_k \|_{\Psi1}
\le \|\alpha_k\|_{\Psi2}\|\gamma_k\|_{\Psi2} 
\).
\end{proof}

\begin{definition} \label{def:Gauss1}
Let $m\in\bN$. 
For $k\in\bN_{\ge m}$, let $X_k := -G_k/\sqrt{\beta_k}$ where $G_k$ are the i.i.d.~complex Gaussians  from Fact~\ref{fact:Gamma}.
In terms of the Pr\"ufer phase $\{\varpi_k\}_{k\in\bN_0}$,  we define
\[
\bphi_{m,n}(\theta):= {\textstyle \sum_{k=m}^{n-1}}  X_k e^{\i\varpi_k(\theta)} , \qquad \theta\in\bT,\quad n\ge m . 
\]
The process $\{ \bphi_{m,n}\}_{n\ge m}$ is a continuous $\{\F_n\}$ martingale.
\end{definition}

Using  Lemma~\ref{lem:trunc1}, one can show that for a given $\theta\in\bT$, the processes
$\{\psi_{m,n}(\theta)\}_{n\ge m}$ from \eqref{S4} and $\{\bphi_{m,n}(\theta)\}_{n\ge m}$ are close with overwhelming probability as $m\to\infty$.
In particular, for a fixed $\theta\in\bT$, the new process $\{\bphi_{m,n}(\theta)\}_{n\ge m}$ has independent (complex) Gaussian increments and quadratic variation $[\bphi_{m,n}(\theta)] = \sum_{k=m}^{n-1} (1+\beta_k)^{-1} \simeq \frac2\beta \log(\frac nm) $ for $n\ge m$. 

\begin{proposition}[Linearization] \label{prop:approx1}
For any $0<\eta<1/2$, there is a $\epsilon=\epsilon(\eta)$ so that  for $\theta\in\bT$ and $m\in\bN$, 
\[
\bP_m\left[  \sup_{n\ge m} |\varphi_{m,n}(\theta)-\bphi_{m,n}(\theta)|  \ge  m^{-\eta}\right] \lesssim \exp\big( -\beta m^{-\epsilon}\big) . 
\]
\end{proposition}

\begin{proof}
On the event \eqref{trunc1}, it holds for $k\ge m$, 
\[
\big|\log(1-\alpha_k e^{\i\varpi_k(\theta)}) -  \alpha_k e^{\i\varpi_k(\theta)} + \tfrac12 \alpha_k^2 e^{2\i\varpi_k(\theta)} \big|
\le  |\alpha_k|^3 \le k^{-3\eta} .
\]
Choosing $1/3<\eta<1/2$, these error are summable.
By \eqref{inc}, this implies that for any $n\ge m$,  
\[\begin{aligned}
\varphi_{m,n}(\theta) &=  {\textstyle \sum_{k=m}^{n-1}}  \log(1-\alpha_k e^{\i\varpi_k(\theta)}) \\
&= -{\textstyle \sum_{k=m}^{n-1}}  \alpha_k e^{\i\varpi_k(\theta)} 
+\tfrac12 {\textstyle \sum_{k=m}^{n-1}} \alpha_k^2 e^{2\i\varpi_k(\theta)}  
+O(m^{1-3\eta}) .
\end{aligned}\]
The error control (deterministically) on the event \eqref{trunc1}, uniformly for all $\theta\in\bT$. 
Since the process $\{\varpi_k\}_{k\in\bN_0}$ is adapted and $\bE_k \alpha_k^2=0$ (by rotation invariance),
$\big\{{\textstyle \sum_{k=m}^{n-1}} \alpha_k^2 e^{2\i\varpi_k(\theta)}\big\}_{n\ge m}$ is a complex-valued martingale whose increments satisfy by \eqref{exptb}, 
\[
\big\|\alpha_k^2\big\|_{\Psi1} \le c/\beta k , \qquad k\in\bN. 
\] 
Hence, by Proposition~\ref{prop:conc1}, there is a numerical constant $c$ so that for any $\theta\in\bT$ and any $\lambda \le \beta^{-1}$, 
\begin{equation*}
\bP_m\left[\sup_{n\ge m}\big|{\textstyle \sum_{k=m}^{n-1}} \alpha_k^2 e^{2\i\varpi_k(\theta)}\big|   \ge \lambda\right] \le 2 \exp\left(-c\beta^2 m \lambda^2 \right). 
\end{equation*}
Taking $\lambda=m^{-\eta}$, this implies that for any $0<\eta<1/2$, 
there is a $\epsilon = \epsilon(\eta)>0$ so that for a fixed $\theta\in\bT$,
\[
\psi_{m,n}(\theta) = -{\textstyle \sum_{k=m}^{n-1}}  \alpha_k e^{\i\varpi_k(\theta)}  
+O(m^{-\eta}) \qquad\text{uniformly for all $n\ge m$},  
\]
with probability $1- \exp\big( -\beta m^{-\epsilon}\big)$ if $m$ is sufficiently large. 
This implies that for $n\ge m$, 
\[
\bphi_{m,n}(\theta)-\varphi_{m,n}(\theta) = {\textstyle \sum_{k=m}^{n-1}}  \big(\alpha_k -G_k\beta_k^{-1/2}\big) e^{\i\varpi_k(\theta)}  
+O(m^{-\eta}) .
\]
The main term is again a  martingale whose increments are controlled by Lemma~\ref{lem:trunc1} -- its quadratic variation is summable and bounded by $\beta_m^{-1}$.   
Hence, as above, $\sup_{n\ge m}\big|{\textstyle \sum_{k=m}^{n-1}}  \big(\alpha_k -G_k\beta_k^{-1/2}\big) e^{\i\varpi_k(\theta)}\big| \lesssim m^{-\eta}$ with probability at least $1- \exp\big( -\beta m^{-\epsilon}\big)$. This proves the claim. 
\end{proof}

We now focus on the process from Definition~\ref{def:Gauss1}

\begin{definition} \label{def:Gauss2}
Recall that $m=m(n)= \lfloor n \Lambda^{-1}\rfloor $ where $\Lambda(n)\to\infty$ as $n\to\infty$. 
We write $\Lambda = e^{\eta L}$
where  $\eta(n) \to0$ as $n\to\infty$ is a parameter to be chosen later on and $L(n) \le n$.   
The relevant condition will be that $\eta \gg \Lambda  n^{-\delta}$ and it will be convenient to choose  $\Lambda := n^{\frac\delta8}$ and $\eta  := n^{-\frac{\delta}2}= \Lambda n^{-\frac{3\delta}4} $ and
Then, we let
 \[
 n_\ell := \lfloor n  e^{\eta(\ell- L)} \rfloor , \qquad \ell\in[0,L]
 \]
so that $n_0=m$, $n_L=n$ and $n_\ell$ is a geometric progression at rate $e^\eta$. 
We introduce a new process;
\[\begin{cases}
\hvarphi_{m,n}(\theta):= \sum_{k=m}^{n-1}  X_k e^{\i\hpi_k(\theta)} ,   \qquad \theta\in\bT,\quad n\ge m   \\
\hpi_k(\theta) = (k-n_j)\theta + \varpi_{n_j}(\theta) , \qquad k\in[n_j,n_{j+1}) , \, j\in\bN_0
\end{cases} . \]
Again $\{\hvarphi_{m,n}\}_{n\ge m}(\theta)$ is a Gaussian process, with a piecewise continuous phase, and also a continuous $\{\F_n\}$ martingale. 
\end{definition}  

\begin{proposition}[Approximation] \label{prop:approx2}
For any $R\ge 1$ sufficiently large, and for any $\theta\in\bT$, with  $\epsilon=R \sqrt{\eta \log \Lambda}$,
\[
\bP_m\bigg[  \max_{k\in[m,n]} |\hvarphi_{m,k}(\theta)-\bphi_{m,k}(\theta)|  \ge  \epsilon \bigg] \lesssim  n e^{-c\beta R}
\]
Using the conventions from Definition~\ref{def:Gauss2}, take $R= \frac{n^{\varkappa}}{\sqrt{\log n}}$ with $0<\varkappa< \frac\delta4$, then 
$\epsilon \le  n^{\varkappa-\frac \delta4}$. 
\end{proposition}

\begin{proof}
Here we fix $\theta\in\bT$ (by rotation-invariance, one can also assume that $\theta=0$), so we omit the $\theta$-dependence. 
Let $\epsilon>0$ be a small parameter. We introduce a stopping time
\[
\tau : = \inf\big\{ k\ge m ; |\varpi_k -\hpi_k| >2\sqrt{\epsilon}\big\} . 
\]
Recall that by \eqref{S4}, for any $j\in\bN_0$
\[
{\varpi_k -\hpi_k } = -2 \psi_{k,n_j} , \qquad  k\in[n_j,n_{j+1}).  
\]
For $j\in\bN_0$, the process $\{\psi_{k,n_j}\}_{k\ge n_j}$ is a martingale and according to \eqref{phasetb}, 
\[
\bP_{n_j}\Big[ \max_{n_j<k\le n_{j+1}}|\psi_{k,n_j} |\ge \sqrt\epsilon\Big] \le 2 \exp\big(- c\beta \epsilon\eta^{-1}\big) . 
\]
using that $\log\frac{n_{j+1}}{n_j} \le \eta $. Then, one has
\[\begin{aligned}
\bP\big[\tau > n_{j+1}\big] &= \bP\Big[ \max_{n_j<k\le n_{j+1}}|\psi_{k,n_j} |\le \sqrt\epsilon ; \tau > n_{j} \Big] \\
& \ge \bP\big[\tau > n_{j}\big] \big(1- 2 \exp\big(- c\beta \epsilon\eta^{-1}\big) \big) \\
&\ge  \big(1- 2 \exp\big(- c\beta \epsilon\eta^{-1}\big) \big)^j
\ge  1- 2j \exp\big(- c\beta \epsilon\eta^{-1}\big) 
\end{aligned}\]
where the last bound follows from convexity. 
Since $n_L =n$ and $L\le n$, this shows that 
\begin{equation}\label{ST}
\bP\big[\tau \le n\big] \le 2n\exp\big(- c\beta \epsilon\eta^{-1}\big)  .
\end{equation}

Let $\Delta_{m,n} := (\hvarphi- \bphi)_{m,n\wedge \tau}$ for $m\ge n$.  
The process  $\{\Delta_{m,n}\}_{n\ge m}$ is also a martingale, whose increments are given by 
\[
\1\{k<\tau\}X_k(e^{\i\hpi_k}-e^{\i\varpi_k})=
X_k e^{\i \frac{\varpi_k+\hpi_k}{2}}\sin\big(\tfrac{\hpi_k-\varpi_k}{2}\big)\1\{k<\tau\}  
\qquad \text{for }k\ge m . 
\]
On the event $\{k<\tau\}$, 
\[
\big\|X_k e^{\i \frac{\hpi_k+\varpi_k}{2}}\sin\big(\tfrac{\hpi_k-\varpi_k}{2}\big)\big\|_{\Psi2,\bP_k}^2 \le \epsilon \big\|X_k  \big\|_{\Psi2}^2 \lesssim \frac{\epsilon}{\beta(k+1)}
\]
since $X_k$  is a complex Gaussian, independent of $\F_k$, with variance $\beta_k^{-1} = \frac{2}{\beta(k+1)}$. 

Then, by Proposition~\ref{prop:conc2} with $\sigma_k^2 =\frac{\epsilon}{\beta(k+1)}$, we obtain for $\lambda>0$
\[
\bP_m\Big[  \max_{k\le n\wedge \tau}|\Delta_{m,k}| \ge \lambda\Big]
\le 2 \exp\bigg(- \frac{c\beta\lambda^2}{\epsilon\log \Lambda}\bigg)
\]
where $\Lambda = \frac nm$. 
Choosing $\lambda = \epsilon \sqrt{\eta^{-1}\log \Lambda}$, by \eqref{ST}, this implies that 
\[\begin{aligned}
\bP_m\Big[  \max_{k\le n}|\Delta_{m,k}| \ge \lambda\Big]
 & \le 2 \exp\big(- c\beta \epsilon\eta^{-1}\big)  + \bP\big[\tau \le n\big]  \\
 &\le 2(n+1)\exp\big(- c\beta \epsilon\eta^{-1}\big) 
\end{aligned}\]
Using that $L\le n$ and choosing $\epsilon= R \eta$, this completes the proof. 
\end{proof}

\subsection{Proof of Proposition~\ref{prop:W1}} \label{sec:W}
We use the convention from Definition~\ref{def:Gauss2}. Define the events, for $\theta\in\bT$ 
\[
B_n(\theta):= \left\{\max_{k\in[m,n]} |\hvarphi_{m,k}(\theta)-\varphi_{m,k}(\theta)|  \le n^{\varkappa-\frac\delta5}\right\}  .  
\]
Combining Proposition~\ref{prop:approx1} (with e.g.~$\varkappa=\frac14$) and Proposition~\ref{prop:approx2}, there is a $\varkappa>0$ so that for any fixed $\theta\in\bT$, 
\begin{equation} \label{PB1}
\bP_m\big[B_n^{\rm c}(\theta)\big] \lesssim e^{-\beta n^{-\varkappa}} 
\end{equation}
and, according to \eqref{W1}, one has for $(\theta,\vartheta) \in\bT^2$ with $\theta\neq \vartheta$, 
\[
\mathcal{W}_n(\theta,\vartheta) =
\frac{\big|\bE_m\big[\1\{\cB_n(\theta)\}\1\{\cB_n(\vartheta)\} e^{\i2\kappa(\hpsi_{m,n}(\theta)-\hpsi_{m,n}(\vartheta)) + \gamma (\hchi_{m,n}(\theta)+\hchi_{m,n}(\vartheta))+\O(n^{\varkappa-\delta/4})}\big]\big|}{\bE[e^{\gamma \chi_{m,n}(0)}]^2} 
\end{equation*}
where the implies constants are deterministic. 
Define
$\Sigma :=\sum_{k= m}^{n -1}\beta_k^{-1}$ and for $(\theta,\vartheta) \in\bT^2$ with $\theta\neq \vartheta$, 
\begin{equation*}
\widehat{\mathcal W}_n^\kappa(\theta,\vartheta)  :=  \exp\big(-\tfrac{\gamma^2}2 \Sigma\big)\big| \bE_m\big[ e^{\i2\kappa(\hpsi_{m,n}(\theta)-\hpsi_{m,n}(\vartheta))}e^{\gamma \hchi_{m,n}(\theta)+\gamma \hchi_{m,n}(\vartheta)}\big] \big| . 
\end{equation*}

\begin{proposition} \label{prop:W2}
Almost surely, $\widehat{\mathcal W}_n^0(\theta,\vartheta) \lesssim 1$ uniformly for $(\theta,\vartheta)\in\bT^2$ and, for any $\kappa\ge 1$,  
\[
\widehat{\mathcal W}_n^\kappa(\theta,\vartheta) \lesssim n^{-\tfrac{\delta/4}{1+\beta}} \qquad\text{uniformly for $|\theta-\vartheta| \ge \epsilon_n = n^{\delta-1}$}.  
\]
\end{proposition}

These estimates are proved in Section~\ref{sec:Gauss}.
Moreover, according to Lemma~\ref{lem:exp}, one has 
$\bE[e^{\gamma \chi_{m,n}(0)}]^2 \simeq \exp(-\tfrac{\gamma^2}2 \Sigma)\simeq {\Lambda^{2\hat\gamma^2}}$, so that
\[
\mathcal{W}_n(\theta,\vartheta) \simeq 
\exp\big(-\tfrac{\gamma^2}2 \Sigma\big)\big| \bE_m\big[\1\{\cB_n(\theta)\}\1\{\cB_n(\vartheta)\} e^{\i2\kappa(\hpsi_{m,n}(\theta)-\hpsi_{m,n}(\vartheta))}e^{\gamma \hchi_{m,n}(\theta)+\gamma \hchi_{m,n}(\vartheta)}\big] \big|  
+ \O\big(\widehat{\mathcal W}_n^0(\theta,\vartheta) n^{-\delta/5}\big)
\]
By \eqref{PB1}, one can remove the events $\cB_n$ up to a negligible error so that (almost surely),  for $(\theta,\vartheta) \in\bT^2$ with $\theta\neq \vartheta$, 
\[
\mathcal{W}_n(\theta,\vartheta) \simeq \widehat{\mathcal W}_n^\kappa(\theta,\vartheta) + \O\big(n^{-\delta/5}\big)
\]
Hence, by Proposition~\ref{prop:W2}, we conclude that uniformly for $|\theta-\vartheta| \ge \epsilon_n = n^{\delta-1}$, 
\[
\mathcal{W}_n(\theta,\vartheta) \le n^{-\tfrac{\delta/5}{1+\beta}}. \qed
\]

\subsection{Gaussian calculations; proof of  Proposition~\ref{prop:W2}}
\label{sec:Gauss}
Recall that $\kappa \ge 0$, $\gamma\in\bR$ are fixed and $\Sigma =\sum_{k= m}^{n -1}\beta_k^{-1}$. 
Here $(\theta,\vartheta) \in\bT^2$, we let $\Delta := |\theta-\vartheta|_{\bT}$ and we assume that $\Delta \ge \epsilon_n =n^{\delta-1}$. 

For $\ell\in\bN$, we write for $(\theta,\vartheta) \in\bT^2$, 
\[
\X_{\ell}(\theta): = \gamma \hchi_{n_{\ell},n_{\ell+1}}(\theta) + \i2\kappa \hpsi_{n_{\ell},n_{\ell+1}}(\theta) , \qquad 
\bZ_\ell := \exp\big(\X_{\ell}(\theta)+ \overline{\X_{\ell}}(\vartheta)\big). 
\]
Here, $\{\X_{\ell}(\theta) :\theta\in\bT\}$ are Gaussian fields with obvious independence properties. 
We do not emphasize that $\bZ_\ell $ depends on $(\gamma,\varkappa)$ and $(\theta,\vartheta) \in\bT^2$ as the only relevant quantity for the analysis is  $\Delta = |\theta-\vartheta|_{\bT}$. 
In particular, one has
\begin{equation} \label{W2}
\widehat{\mathcal W}_n^\kappa 
= \exp\big(-\tfrac{\gamma^2}2 \Sigma\big)  \bE_m\big[ {\textstyle\prod_{\ell=0}^{L-1}} \bZ_\ell \big] .
\end{equation}

\medskip
We rely on the following basic computations.
\begin{lemma}\label{lem:Gauss}
Let  $\zeta=(\tfrac{\gamma}{\sqrt 2}+\i \sqrt 2 \kappa)^2$, 
$\sigma_\ell := {\textstyle \sum_{k=n_\ell}^{n_{\ell+1}-1}} \beta_{k}^{-1}$ and $q_\ell := {\textstyle \sum_{k=n_\ell}^{n_{\ell+1}-1}} \beta_{k}^{-1} e^{\i k\Delta}$, for  $\ell\in\bN_0$.
From Definition~\ref{def:Gauss2}, one has $\Sigma = {\textstyle\sum_{\ell=0}^{L-1}} \sigma_\ell$ and  
$\sum_{\ell=0}^{L-1}|q_\ell| \lesssim  \beta^{-1} n^{-\delta/4}$ if $\Delta =|\theta-\vartheta|_{\bT}\ge n^{\delta-1}$. \\
For every $\ell\in\bN_0$, there is a random $\omega_\ell \in \bT$, $\F_{n_\ell}$-measurable, so that 
\[
 \bE_{n_\ell}[\bZ_\ell] = \exp\big( \big(\tfrac{\gamma^2}2-2\kappa^2\big)\sigma_\ell + \Re(\zeta e^{\i \omega_\ell} q_\ell)\big)
\]
and 
\begin{equation} \label{Zbd}
 \bE_{n_\ell}[|\bZ_\ell|] \le \exp\big(\tfrac{\gamma^2}2 \sigma_\ell + \tfrac{\gamma^2}2 |q_\ell|\big) . 
\end{equation}
\end{lemma}

\begin{proof}
For a (standard) complex Gaussian $X$ and deterministic $\omega_1,\omega_2 \in\bT$,
\[
\bE\big[\Re(X e^{\i\omega_1})\Re(X e^{\i\omega_2})\big] = \bE\big[\Im(X e^{\i\omega_1})\Im(X e^{\i\omega_2})\big] =\frac{\cos(\omega_1-\omega_2)}{2} , \qquad
\bE\big[\Re(X e^{\i\omega_1})\Im(X e^{\i\omega_2})\big] 
= \frac{\sin(\omega_1-\omega_2)}{2} . 
\]
Then, since $\{X_k\}$ are independent complex Gaussians, one has for $(\theta,\vartheta)\in\bT^2$,
\[
\bE_{n_\ell}\X_{\ell}(\theta)^2
=  \big(\tfrac{\gamma^2}2 - 2\kappa^2\big) \sigma_\ell  
\]
and
\[
\bE_{n_\ell}\X_{\ell}(\theta)\overline{\X_{\ell}}(\vartheta) =  \sum_{n_\ell\le k< n_{\ell+1}}  \beta_k^{-1}\big(\big(\tfrac{\gamma^2}2 - 2\kappa^2\big)  \cos \Delta_k
- 2\i\gamma\kappa \sin\Delta_k \big)
= \Re\bigg( \zeta \sum_{n_\ell\le k< n_{\ell+1}}  \beta_k^{-1} e^{\i\Delta_k}\bigg)
\]
 where $\Delta_k :=\hpi_k(\theta)- \hpi_k(\vartheta) = (k-n_j)(\theta-\vartheta)+ \varpi_{n_\ell}(\theta)-\varpi_{n_\ell}(\vartheta)$. 

This implies that
\[
 \bE_{n_\ell}\big[ \exp \X_{\ell}(\theta)\exp \overline{\X_{\ell}}(\vartheta)\big] = \exp \tfrac12 \bE_{n_\ell}( \X_{\ell}(\theta)+\overline{\X_{\ell}}(\vartheta))^2
 = \exp\big( \big(\tfrac{\gamma^2}2-2\kappa^2\big)\sigma_\ell + \Re(\zeta e^{\i \omega_\ell} q_\ell)\big)
\]
with $\omega_\ell = \pm(\varpi_{n_\ell}(\theta)-\varpi_{n_\ell}(\vartheta)) [2\pi]$ -- $\omega_\ell \in \bT$ is $\F_{n_\ell}$-measurable. 
The estimate \eqref{Zbd} follows from the previous formula with $\kappa=0$ and using that  $|\Re(e^{\i \omega_\ell} q_\ell)|\le |q_\ell|$. 
Finally, one has 
\[
|q_\ell| \le \frac2\beta
\bigg| \sum_{n_\ell\le k< n_{\ell+1}}  \frac{e^{\i k\Delta}}{k+1} \bigg|
\le \frac2\beta\bigg(\bigg| \sum_{n_\ell\le k< n_{\ell+1}}  \frac{e^{\i k\Delta}}{n_\ell} \bigg| + \frac{n_{\ell+1}}{n_\ell^2}\bigg)\lesssim \frac{\beta^{-1}}{n_\ell\Delta} . 
\]
using that $n_{\ell+1} = n_\ell e^\eta$ with  $e^\eta\le \frac1\pi \le  \Delta^{-1}$. 
Consequently, $\sum_{\ell=0}^{L-1}|q_\ell| \lesssim 
\frac{1}{\beta\Delta m} \sum_{\ell=0}^{L-1} e^{-\eta\ell}
\lesssim \frac{1}{\beta\Delta m\eta}$. 
With the convention from Definition~\ref{def:Gauss2},  $\eta  = \Lambda n^{-\frac{3\delta}4}$, $m\simeq n \Lambda^{-1} $ so that
$\sum_{\ell=0}^{L-1}|q_\ell| \lesssim \frac{1}{\beta n^{\delta/4}}$ 
if  $\Delta \ge n^{\delta-1}$. 
\end{proof}

Lemma~\ref{lem:Gauss} allows us to compute $\widehat{\mathcal W}_n^\kappa$ recursively using \eqref{W2}. 

First, using the estimate \eqref{Zbd}, one has for any $j\in\bN_{< L}$,  
\begin{equation} \label{varbd}
\widehat{\mathcal W}_n^0= \bE_m\big[ {\textstyle\prod_{\ell=0}^{j}}|Z_\ell|\big] \lesssim \exp\big(\tfrac{\gamma^2}2 \Sigma\big).
\end{equation}

Then, one has for $j\in\bN_{<L}$, 
\[\begin{aligned}
 \bE_m\big[ {\textstyle\prod_{\ell=0}^{j}} Z_\ell\big]
&=  \bE_m\big[ {\textstyle\prod_{\ell=0}^{j-1}} Z_\ell \bE_{n_{j}} Z_j\big] \\
& =  \exp\big( \big(\tfrac{\gamma^2}2-2\kappa^2\big)\sigma_j\big)
  \bE_m\big[ {\textstyle\prod_{\ell=0}^{j-1}} Z_\ell \exp\Re(\zeta e^{\i \omega_\ell} q_j)\big]\\
&=\exp\big( \big(\tfrac{\gamma^2}2-2\kappa^2\big)\sigma_j\big)
  \bE_m\big[ {\textstyle\prod_{\ell=0}^{j-1}} Z_\ell \big]
+ \O\big(|q_j| \bE_m\big[ {\textstyle\prod_{\ell=0}^{j-1}} |Z_\ell| \big] \exp\big(\tfrac{\gamma^2}2 \sigma_j\big)\big) \\
&=\exp\big( \big(\tfrac{\gamma^2}2-2\kappa^2\big)\sigma_j\big)
  \bE_m\big[ {\textstyle\prod_{\ell=0}^{j-1}} Z_\ell \big]
+ \O\big(|q_j| \exp\big(\tfrac{\gamma^2}2 \Sigma\big)\big) .
\end{aligned}\]
Consequently, by induction, using that $\Sigma = {\textstyle\sum_{\ell=0}^{L-1}} \sigma_\ell$ and
$\sum_{\ell=0}^{L-1}|q_\ell| \lesssim   n^{-\delta/4}$, 

\[\begin{aligned}
 \bE_m\big[ {\textstyle\prod_{\ell=0}^{L-1}} \bZ_\ell\big]
 & = \exp\big( \big(\tfrac{\gamma^2}2-2\kappa^2\big)\Sigma\big)
+ \O\big(n^{-\delta/4}  \exp\big(\tfrac{\gamma^2}2 \Sigma\big)\big)  \\
\widehat{\mathcal W}_n^\kappa &= \exp\big(-2\kappa^2 \Sigma\big) + \O(n^{-\delta/4}) . 
\end{aligned}\]
Recall that we choose  $\Lambda= n^{\frac\delta 8}$ and $\Sigma\simeq \frac2\beta \log \Lambda$, so we conclude that for $\kappa\ge 1$ \[
|\widehat{\mathcal W}_n | \lesssim n^{-\frac{\delta/4}{1+\beta}}. \qed
\]

\appendix 
\section{Properties of the C$\beta$E model}\label{A:model}

\noindent
In this section, we collect \emph{known facts} about the circular $\beta$-ensembles which are relevant in our context.

\medskip

\paragraph{\bf Selberg formulae} 
We record the following explicit formulae for the C$\beta$E model (Section~\ref{sec:CbE}). 

\begin{lemma} \label{lem:expmom}
Recall that $\beta_k := \beta \tfrac{k+1}{2}$ for $k\in\bN_0$,   $\chi_n : \theta\in\bT\mapsto \Re\varphi_n(e^{\i\theta})$ and $\psi_n: \theta\in\bT\mapsto \Im\varphi_n(e^{\i\theta})$ for $n\in\bN$.
It holds for  $\gamma\in\bR$, 
\[
\bE e^{\gamma \psi_{n}(\theta)} =\prod_{k<n} \frac{\Gamma(1+\beta_k)^2}{|\Gamma(1+\beta_k-\gamma/2\i)|^2} , \qquad 
\bE e^{\gamma \chi_{n}(\theta)} = \prod_{k< n} \frac{\Gamma(1+\beta_k)\Gamma(1+\gamma+\beta_k)}{\Gamma(1+\gamma/2+\beta_k)^2} \quad\text{for }\gamma>-1-\tfrac\beta2. 
\qquad
\]
\end{lemma}

\begin{proof}
By \eqref{S1}, one has $\varphi_n(e^{\i\theta}) = \sum_{k=0}^{n-1} \log(1-\alpha_k e^{\i\varpi_k(\theta)}) $ where the phase $\{\varpi_k\}$ is adapted. Then, by independence and rotation-invariance of the Verblunsky coefficients $\{\alpha_k\}$, for a fixed $\theta\in\bT$,
$\varphi_n(e^{\i\theta}) \equiv \sum_{k=0}^{n-1} \log(1-\alpha_k) $.
Thus these formulae follow from  and the facts (see e.g.~\cite[Lem~3.1]{BNR}), 
\begin{equation} \label{expmom}
\begin{aligned}
&\bE|1- \alpha_k|^\gamma= \frac{\Gamma(1+\beta_k)\Gamma(1+\gamma+\beta_k)}{\Gamma(1+\gamma/2+\beta_k)^2}
\qquad\text{for }\Re\gamma>-1-\beta_k , \\
&\bE \exp\big(\gamma  \Im \log( 1-\alpha_k)  \big) =\frac{\Gamma(1+\beta_k)^2}{|\Gamma(1+\beta_k-\gamma/2\i)|^2} \qquad\gamma\in\bC . 
\end{aligned}
\end{equation}
In particular, the formulae extend naturally for $\gamma\in\bC$. 
\end{proof}

Recall that the C$\beta$E characteristic polynomial satisfies \eqref{charpoly2}, where $\eta$ is a uniform random variable in $\bT$, independent of $\{\alpha_k\}$. Then
\[
\bE|1- e^{\i\eta}|^\gamma = \sum_{k\ge 0} {\gamma/2 \choose k}^2 =\frac{\Gamma(1+\gamma)}{\Gamma(1+\gamma/2)^2} 
\qquad \text{for } \gamma>-1.
\]
This corresponds to \eqref{charpoly2} with $\beta_k=0$, see also \eqref{F1} with $r=1$. Then by lemma~\ref{lem:expmom}, for  $n\in\bN_0$  and $\gamma>-1$, 
\[
\bE |\mathcal{X}_{n+1}(1)|^\gamma = \frac{\Gamma(1+\gamma)}{\Gamma(1+\gamma/2)^2} \prod_{k< n} \frac{\Gamma(1+\beta_k)\Gamma(1+\gamma+\beta_k)}{\Gamma(1+\gamma/2+\beta_k)^2} .
\]
Similarly, with $\Im \log( 1- e^{\i\eta}) = \mathrm{h}(\eta)$, \eqref{Y2}, one has for $\gamma\in\bR$ with $\kappa = \gamma/2\i$, 
\[
\bE \exp\big(\gamma  \mathrm{h}(\eta)  \big) = \sum_{k\ge 0}  {\kappa \choose k}   {-\kappa \choose k} = \frac1{\Gamma(1+\kappa) \Gamma(1-\kappa)} =  |\Gamma(1-\kappa)|^{-2} . 
\]
Then, by  \eqref{Y1}, the imaginary part of the characteristic polynomial satisfies for $\gamma\in\bR$
\[
\bE e^{\gamma \mathcal{Y}_{n+1}(1)} = \frac1{|\Gamma(1-\gamma/2\i)|^2}
 \prod_{k<n} \frac{\Gamma(1+\beta_k)^2}{|\Gamma(1+\beta_k-\gamma/2\i)|^2} .
\]

\paragraph{\bf Moment estimates}
We need the following estimates. In particular, the case $m=0$ corresponds to Lemma~\ref{lem:mom}. 
Recall that $\chi_{m,n} = \chi_n- \chi_m$ and $\psi_{m,n} = \psi_n- \psi_m$  for $n>m\ge 0$. 

\begin{lemma} \label{lem:exp}
Let $\gamma\in\bR$, $\hat\gamma=\gamma/\sqrt{2\beta}$ with $|\hat\gamma|<1$. 
It holds for all  $n,m \in \bN_0$ with $n>m$ and for any $\theta\in\bT$, 
\[
\bE_m[e^{\gamma \psi_{n,m}(\theta)}] = \exp\big(\hat\gamma^2 \log \tfrac n{m+1} + \O\big(\tfrac1{m+1}\big)\big) \qquad\text{and}\qquad
\bE_m[e^{\gamma \chi_{n,m}(\theta)}] = \exp\big(\hat\gamma^2 \log \tfrac n{m+1} + \O\big(\tfrac1{m+1}\big)\big) ,
\]
where the implied constants depend only on $(\gamma,\beta)$. 
\end{lemma}

\begin{proof}
For $\gamma\in\bC$ fixed, by e.g.~\cite{Tricomi}, one has as $\kappa\to\infty$, 
\[
\frac{\Gamma(1+\frac\kappa2)\Gamma(1+\gamma+\frac\kappa2)}{\Gamma(1+\gamma/2+\frac\kappa2)^2} = \bigg(1-\frac{\gamma^2}{4(\kappa+1)} +\O(\kappa^{-2})\bigg)\bigg(1+\frac{3\gamma^2}{4(\kappa+1)} +\O(\kappa^{-2})\bigg)
= 1+ \frac{\gamma^2}{2 \kappa }+\O(\kappa^{-2}).
\]
Similarly
\[
\frac{\Gamma(1+\frac\kappa2)^2}{\Gamma(1+\gamma/2+\frac\kappa2) \Gamma(1-\gamma/2+\frac\kappa2)}= \bigg(1-\frac{\gamma^2}{4(\kappa+1)} +\O(\kappa^{-2})\bigg)\bigg(1-\frac{\gamma^2}{4(\kappa+1)} +\O(\kappa^{-2})\bigg) = 1- \frac{\gamma^2}{2 \kappa }+\O(\kappa^{-2}).
\]
Then, by \eqref{expmom}, we obtain for $\gamma\in\bR$  with $|\hat\gamma|<1$ (this condition guarantees that $\gamma>-1-\tfrac\beta2$, so this expectation is well defined for any $m\in\bN_0$), 
\[
\bE e^{\gamma \chi_{n,m}} = \prod_{m\le k< n}\bigg(  1+ \frac{\hat\gamma^2}{k+1}+\O(k^{-2})\bigg) = \exp\big( \hat\gamma^2 \log \tfrac n{m+1} + \O\big(\tfrac1{m+1}\big)\big). 
\]
We used that the harmonic sum $\sum_{k=m}^{n-1} (k+1)^{-1} = \log \tfrac n{m+1} + \O\big(\tfrac1{m+1}\big)$. A similar expansion holds for 
$\bE e^{\gamma \psi_{n}} $.
\end{proof}

\paragraph{\bf Leading order of the maximum}
As a consequence of Theorem~\ref{thm:phi} and the bounds of Lemma~\ref{lem:exp} (with $m=0$) one controls the asymptotics of the maximums of the fields $\chi_n(\theta)=\Re\varphi_n(e^{\i\theta})$ and $\psi_n(\theta)=\Im\varphi_n(e^{\i\theta})$ on $\bT$. 

\begin{cor} \label{thm:max}
For any $\delta>0$,  as $n\to\infty$,
\[
\bP \left[   \tfrac{2-\delta}{\sqrt{2\beta}} \log n   \le  \max_{\theta\in\bT} |\chi_{n}(\theta)|    \le \tfrac{2+\delta}{\sqrt{2\beta}} \log n   \right] \to 1. 
\]
An analogous result holds for $ \max_{\theta\in\bT} |\psi_{n}(\theta)| $. 
\end{cor}

We refer to \cite[Sec 3]{CFLW} for details.
As explained in Section~\ref{sec:rmt}, following from \cite{CMN,PZ},  precise asymptotics for the maximums are available. 
The upper-bound from Corollary~\ref{thm:max} is used in Section~\ref{sec:strat} to prove Theorem~\ref{thm:charpoly}. 

\medskip

\paragraph{\bf OPUC convergence} 
We are interested in the asymptotics of $\{ \varphi_k = \log \Phi_{k}^*\}_{k\ge 0}$ where $\{ \Phi_{k}^*\}_{k\ge 0}$ are the Szeg\H{o} polynomial, \eqref{S0}. 
Recall that $\{ \varphi_k\}_{k\ge 0}$ is a  $\F_{k}$-martingale and we will need the following non-trivial result from \cite[Proposition 3.6]{CN}.  

\begin{proposition} \label{prop:cvg}
Assume that $\{\alpha_n\}_{n\in\bN_0}$  are independent, rotation-invariant, random variables such that for any $\gamma>0$,
\begin{equation} \label{mom3}
\bE \exp\big(\gamma{\textstyle\sum_{n\in\bN_0}} |\alpha_n|^3 \big) <\infty .
\end{equation}
Then, almost surely, $\varphi_k \to \varphi_{\infty}$ locally uniformly on $\bD$ as $k\to\infty$. 
Consequently, the function
$\varphi_{\infty}$  analytic  in $\bD$ and for any $\gamma>0$, $z\in\bD$,
\[
\sup_{n\in\bN}\bE \exp \big(\gamma |\varphi_{n}(z)|\big)  <\infty .
\]
\end{proposition}

The condition \eqref{mom3} is satisfied for the C$\beta$E model. 
For $\gamma>0$, there is a constant $c_\gamma>0$ so that for every $n\in\bN_0$, 
\[
\bE e^{\gamma|\alpha_n|^3} \le \int_0^1 (1+c_\gamma t^{3/2}) e^{-\beta_n t} \dd t \le 1+ c_\gamma' \beta_n^{-3/2} . 
\]
Thus, for any $\beta, \gamma>0$, 
\[
\bE \exp\big(\gamma{\textstyle\sum_{n\in\bN_0}} |\alpha_n|^3 \big)  \le \exp\big(c_\gamma' {\textstyle\sum_{n\in\bN_0}} \beta_n^{-3/2} \big) 
\le \exp\big(c_\gamma''\beta^{-3/2}\big). 
\]
This guarantees that the limits introduced in Lemma~\ref{lem:decoup} are well-defined (analytic in $\bD$) functions. 

\medskip

Next, we explain the connection with the central limit theorem for the C$\beta$E eigenvalue statistics. 

\begin{proof}[Proof of Proposition~\ref{prop:GAF}]
By \cite[Thm~1.7.4]{Sim04}, since the C$\beta$E Verblunsky coefficients $\alpha_n \to0$ almost surely as $k\to\infty$, then $B_k(z)  = z  \Phi_{k}(z)/ \Phi_{k}^*(z) \to 0$ as $n\to\infty$ locally uniformly for $z\in\bD$. 
Then, according to \eqref{charpoly3},  
\begin{equation} \label{ascvg}
\log \mathcal{X}_n(z) \to\varphi(z) \qquad\text{almost surely as $n\to\infty$, locally uniformly for $z\in\bD$}.  
\end{equation}

On the other hand, we deduce from \eqref{Joh} (taking $f : \theta\in\bT \mapsto \sum_{k=1}^\infty \frac{z^k}k e^{\i k\theta}$ for a fixed $z\in\bD$), that for any $z_1,\dots, z_k \in\bD$, 
\begin{equation} \label{fdcvg}
\big\{ \tfrac1\pi  \log \mathcal{X}_n(z_j) \big\}_{j=1}^k   \Rightarrow \mathcal N\big(0,\boldsymbol\Sigma\big) \qquad\text{in distribution as $n\to\infty$}, 
\end{equation}
where the RHS is a multivariate complex Gaussian distribution\footnote{\blue
Let $M_k(\mathbb X)$ be the set of $k\times k$ matrices with entries in $\mathbb X$. 
If $X\in \bC^k$ is a random vector with $X_j\in L^2(\bP)$, its covariance matrix $\boldsymbol\Sigma\in M_k(M_2(\bR))$ is defined by $\boldsymbol\Sigma_{ij} = \Big(\begin{smallmatrix} \bE( \Re X_i \Re X_j) & \bE( \Re X_i \Im X_j)  \\ \bE( \Im X_i \Re X_j)  & \bE( \Im X_i \Im X_j) \end{smallmatrix}\Big)$.
} 
with covariance matrix $\boldsymbol\Sigma_{ij} = \tfrac2\beta \log(1-z_i\overline{z_j})^{-1}\big(\begin{smallmatrix} 1 & 0 \\ 0 &1\end{smallmatrix}\big)$.
Combining \eqref{ascvg} and the finite dimensional distribution convergence \eqref{fdcvg}, we conclude that the field $\{\varphi(z) : z\in\bD\}$ is a GAF with covariance structure \eqref{GAF}. 
\end{proof}

\begin{remark}[GAF in $\bD$] \label{rk:traces}
The field $\{\varphi(z) : z\in\bD\}$  can be realized as follows; for a fixed $\beta>0$, there is a sequence of i.i.d.~complex Gaussian random variables $\{\boldsymbol{\mathcal N}_k\}_{k\in\bN}$ so that for $z\in\bD$, 
\[
\varphi(z) = \sum_{k\ge 1} \frac{\boldsymbol{\mathcal N}_k}{\sqrt k} {z^k} \qquad \bE|\boldsymbol{\mathcal N}_k|^2 =\tfrac2\beta, \quad 
\bE\boldsymbol{\mathcal N}_k^2 =0. 
\]
Observe that by \eqref{charpoly0}, the log characteristic polynomial has an expansion for $z\in\bD$, 
\[
\log \mathcal{X}_n(z) = -\sum_{k\ge 1} \frac{z^k}{k} \big( {\textstyle\sum_{j=1}^n} e^{-\i k\theta_j}\big) .
\]
Hence, within the C$\beta$E coupling, almost surely for every $k\in\bN$, 
\begin{equation} \label{trace}
\operatorname{Tr}\big[\mathcal U_\alpha^{(n)k}\big] = {\textstyle\sum_{j=1}^n} e^{\i k\theta_j} \to {\sqrt k} \boldsymbol{\mathcal N}_k \qquad\text{ as $n\to\infty$}. 
\end{equation}
Remarkably, for CUE $(\beta=2)$, the rate of convergence in \eqref{trace} is super-exponential. We refer to \cite{JL21,CJL24} for quantitative results. 

\end{remark}

\paragraph{\bf Pr\"ufer phases}
The (relative) Pr\"ufer phases have been introduced in \cite[Sec.~2]{KS09} to study the microscopic landscape of the C$\beta$E. 

\begin{proof}[Proof of Proposition~\ref{prop:phase}]
Since the Szeg\H{o} polynomials $\Phi_{n}^*$ have no zero in $\overline{\bD}$, it is a direct consequence of \eqref{charpoly3}--\eqref{charpoly2} that, within the C$\beta$E coupling,   the eigenvalues lie on the unit circle and for any $n\in\bN$,
\begin{equation*} \label{eig}
\mathfrak{Z}_n 
= \big\{ u\in\bU : \mathcal{X}_n(u)=0\big\}
= \big\{ \theta \in \bT : \varpi_{n-1}(\theta)  = -\eta [2\pi] \big\}  .
\end{equation*}
where $B_n(e^{\i\theta}) = e^{\i \varpi_n(\theta)}$ for $\theta\in\bT$. 
For $\theta\in\bT$, one has $\Phi_{n}(u)= u^n \overline{\Phi_n^*(u)} = u^n e^{\overline{\varphi_{n}(u)}}$
so that 
$B_n(u) : = u \Phi_{n}(u)/ \Phi_{n}^*(u)= u^{n+1}e^{-2\Im{\varphi_{n}(u)}}$. This yield for $n\in\bN_0$, $\varpi_n(\theta) = (n+1)\theta-2 \psi_n(\theta)$ for $\theta\in[0,2\pi]$, see \eqref{phase0}, and the recursion \eqref{Phase} is a direct consequence of \eqref{S1}.  
In particular, the determination of the Pr\"ufer phases are consistent with $\varpi_0(\theta) =\theta$ and \eqref{Phase} and with this choice, 
one has $\bE\varpi_n(\theta) = (n+1)\theta$ for all $n\in\bN$. 

\medskip

It remains to show that the functions $\varpi_n$ are continuously increasing. 
We consider the relative phase $\varpi_{n}'(\theta) := \varpi_n(\theta)- \varpi_n(0)$.
We will now show by induction that (almost surely) for any $n\in\bN $,  $\varpi_{n-1}' : [0,2\pi] \nearrow [0,2\pi n]$. Then, it follows that the eigenvalues
\[
\mathfrak{Z}_n = \big\{  \varpi_{n-1}'(\theta_j) =  2\pi j -\eta_n ;\, j\in[n] \big\}   
\]
where $\eta_n = \{\varpi_{n-1}(0)+\eta\}_{[2\pi]}$ for $n\in\bN$. In particular, 
$\mathfrak{Z}_n$ consists of $n$ ordered points $0<\theta_1<\cdots <\theta_n<2\pi$. 

For $\alpha \in\bD$, we consider the map 
\[
\Upsilon_\alpha : \omega \in[0,2\pi] \mapsto \omega - 2\Im\log(1-\alpha e^{\i \omega}) +2\Im\log(1-\alpha) . 
\]
This is a diffeomorphism $\Upsilon_\alpha : [0,2\pi] \nearrow [0,2\pi] $. 
Indeed its derivative is $\Upsilon_\alpha '(\omega )=\Re\big(\frac{1+\alpha e^{\i \omega}}{1+\alpha e^{\i \omega}}\big) >0$, that is the Poisson kernel $P_\alpha$ for $\bD$. 
Moreover, one can rewrite the recursion \eqref{Phase} in terms of the relative phase, for $n\in\bN$, 
\[
\varpi_{n}'(\theta)  = \theta + \Upsilon_{\alpha_n'}(\varpi_{n-1}'(\theta)) , \qquad \theta\in[0,2\pi] ,
\]
with $\varpi_0'(\theta) =\theta$ and $\alpha_n' = \alpha_n e^{\i \varpi_{n-1}(0)}$. 
Hence, if $\varpi_{n-1} : [0,2\pi] \nearrow [0,2\pi n]$, we can extend 
$\Upsilon_{\alpha_n'}(\varpi_{n-1}') : [0,2\pi] \nearrow [0,2\pi n]$ continuously and we conclude that $\varpi_{n}' : [0,2\pi] \nearrow [0,2\pi (n+1)]$. \end{proof}

\section{Concentration bounds}\label{A:conc}

We require different tail bounds for martingale sums in this paper. In this section, we review some relevant results. 

\medskip

\paragraph{\bf Sub-exponential distributions} 
We refer to \cite[Chap~2]{Vershynin} for some background on concentration estimates. 
For a complex-valued random variable $X$, define for $p\ge 1$, 
\[
\|X\|_{\Psi p} = \inf\left\{ t \geq 0 : \bE e^{|X|^p/t^p} \leq 2\right\}.
\]
$\|X\|_{\Psi 2}$ is called  the \emph{sub-gaussian norm}, and 
$\|X\|_{\Psi 1}$  the  \emph{sub-exponential norm} of the random variable $X$. One has 
\[
\|X\|_{\Psi 1}\lesssim \|X\|_{\Psi 2} \lesssim \|X\|_{\bL^\infty} .
\]
In fact, $\|XZ\|_{\Psi 1}\le \|X\|_{\Psi 2} \|Z\|_{\Psi 2}$ for any random variables $(X,Z)$. Moreover, Prop~2.6.1
\[
\big\|{\textstyle \sum_{k=1}^n}X_k \big\|_{\Psi 2}^2 \lesssim {\textstyle \sum_{k=1}^n}\|X_k\|_{\Psi 2}^2
\]

Generally, the \emph{Orlicz norm} controls the tails of a distribution. For $p\ge 1$, if $\|X\|_{\Psi p}^{-p} =c>0$, then (by Markov inequality), 
\begin{equation}\label{Xtail}
\bP \left[ |X| \geq t \right] \leq 2\exp( -ct^p) , \qquad t\ge 0.
\end{equation}
Conversely, if \eqref{Xtail} holds, then $\|X\|_{\Psi p}<\infty$. 
Moreover, if $\|X\|_{\Psi p}\le 1$, then (by Holder inequality), 
\[
\bE e^{|X|^p} \le 2^{\|X\|_{\Psi p}^p}
\]

\paragraph{\bf Concentration inequality for (discrete) martingales} 
We consider a martingale sequence $\{M_n\}_{n\in\bN_0}$ of the type, for $n\ge 0$, 
\[
M_n = {\textstyle \sum_{k=0}^{n-1} X_k} , \qquad  \bE |X_k|^2<\infty \text{ and }\bE(X_k|\F_k)=0  \text{ for $k\in\bN_0$}. 
\]
Here, $M_0=0$ and $\F_k= \sigma\big(X_0,\cdots, X_{k-1}\big)$  for $k\in\bN_0$. 
Moreover, in the framework of the article, the increments $X_k$ are generally complex-valued.  
%
%
%
%
%
%
For a  sequence of martingale increments $\{X_n\}_{n\in\bN_0}$, we define for $p\in\{1,2\}$, 
\[
\sigma_n^p :=  \left\| \inf\left\{ t \geq 0 :  \bE\big( e^{|X_n|^p/t} | \F_n\big) \leq 2\right\} \right\|_{\bL^\infty } . 
\]
 The infimum corresponds to the \emph{Orlicz norm} $\| X \|_{\Psi p,\bP_n}^p $ with respect to the conditional measure $\bP_n =\bP(\cdot|\F_n)$, so it is a random variable. However, if there is a sequence of independent random variables $\{Z_n\}_{n\in\bN_0}$ such that,  conditioning on $\F_n$, $|X_n| \lesssim Z_n$, then $\sigma_n^p \lesssim \| Z_n \|_{\Psi p}^p$ for all $n\in\bN_0$.


\medskip

The next Proposition is a  Hoeffding-type inequality for a complex-valued martingale with sub-gaussian increments. 

\begin{proposition} \label{prop:conc2}
Suppose that the martingale increments satisfy $ \sigma_n^2 <\infty$  all  for $n\in\bN_0$. 
Then, for any $n\in\bN$, 
\[
\big\| \max_{k\le n}|M_k| \big\|_2 \lesssim \sqrt{\textstyle \sum_{k=0}^{n-1} \sigma_k^2}  . 
\]
In particular, there is a numerical constant $c>0$, so that for any $n\in\bN$ and $\lambda>0$, 
\[
\bP\left[\max_{k\le n}|M_k|   \ge \lambda\right] \le 2 \exp\left(-\frac{c \lambda^2}{\sum_{k=0}^{n-1} \sigma_k^2} \right). 
\]
\end{proposition}

We will also need a Bernstein-type inequality for a complex-valued martingale with sub-exponential increments. 

\begin{proposition} \label{prop:conc1}
Suppose that the martingale increments satisfy $ \sigma_n^1 <\infty$  all  for $n\in\bN_0$.  
Then, there is a numerical constant $c>0$, so that for any $n\in\bN$ and $\lambda>0$,
\[
\bP\left[\max_{k\le n}|M_k|   \ge \lambda \right] \le 2 \exp\left(-\frac{c \lambda^2}{ \sum_{k=1}^n (\sigma_k^1)^2+ \lambda \max_{k\le n} \sigma_k} \right). 
\]
\end{proposition} 

We refer to \cite{LP} for proofs of Propositions~\ref{prop:conc2} and \ref{prop:conc1}.

\medskip

To illustrate these techniques, we consider the martingale $\psi_{n,m}(\theta) =  {\textstyle \sum_{k=m}^{n-1}} \Im\log(1-\alpha_k e^{\i\varpi_k(\theta)})$ for a fixed $\theta\in\bT$. 
Modulo a deterministic shift, this process corresponds to the Pr\"ufer phase and $\{\psi_{n,m}\}_{n\ge m}$ is a martingale by independence and rotation-invariance of the Verblunsky coefficients $\{\alpha_k\}$.
Observe that for a deterministic constant, 
\[
|\Im\log(1-\alpha_k e^{\i\varpi_k})| \lesssim |\alpha_k| .
\]
Then, by \eqref{exptb}, $\sigma^2_k \lesssim \|\alpha_k\|_{\Psi2}^2 \lesssim \tfrac{1}{\beta(k+1)} $  for any $k\in\bN_0$.
Thus, by Proposition~\ref{prop:conc2}, for any $n,m \in\bN_0$ with $n>m$, any $\lambda>0$ and  any fixed $\theta\in\bT$,  
\begin{equation} \label{phasetb}
\bP\Big[ \max_{m<k\le n}|\psi_{k,m}(\theta)|\ge \lambda\Big] \le 2 \exp\bigg(- \frac{c \beta \lambda^2}{\log\frac nm}\bigg) . 
\end{equation}

\end{document}